\documentclass[12pt]{amsart}%
\usepackage{amsmath}%
\usepackage{amssymb}%
\usepackage{amscd}%
\usepackage{color}%
\usepackage{enumitem}%
\usepackage{pifont}%
\usepackage{ifsym}%
\usepackage{graphicx}%
\usepackage{mathrsfs}%
\usepackage[all]{xy}%
\usepackage{stmaryrd}
\usepackage{newtxtext,newtxmath}%
\usepackage[all, warning]{onlyamsmath}
\usepackage{amsrefs}
\usepackage{bbm}
\usepackage{trimclip}
\makeatletter
\DeclareRobustCommand{\arr}{%
 \mathrel{\mathpalette\short@to\relax}%
}
\newcommand{\short@to}[2]{%
  \mkern2mu
  \clipbox{{.3\width} 0 0 0}{$\m@th#1\vphantom{+}{\shortrightarrow}$}%
  }
\makeatother
\def\ol#1{\overline{#1}}
\def\wt#1{\widetilde{#1}}
\theoremstyle{plain}
    \newtheorem{thmint}{Theorem}
    \newtheorem{theorem}{Theorem}[section]
    \newtheorem{proposition}[theorem]{Proposition}
    \newtheorem{lemma}[theorem]{Lemma}
    \newtheorem{corollary}[theorem]{Corollary}
    \newtheorem{corint}[thmint]{Corollary}
      \newtheorem*{question*}{Question}
\theoremstyle{definition}
    \newtheorem{definition}[theorem]{Definition}
    
    \newtheorem{example}[theorem]{Example}
    \newtheorem{remark}[theorem]{Remark}
\def\Alphabet{A,B,C,D,E,F,G,H,I,J,K,L,M,N,O,P,Q,R,S,T,U,V,W,X,Y,Z}
\def\grabet{a,b,c,d,e,f,g,h,i,j,k,l,m,n,o,p,q,r,s,t,u,v,w,x,y,z}
\def\endpiece{xxx}
\def\makeAlphabet[#1]{\expandafter\makeA#1,xxx,}
\def\makealphabet[#1]{\expandafter\makea#1,xxx,}
\def\makeA#1,{\def\temp{#1}\ifx\temp\endpiece\else%
\mkbb{#1}\mkfrak{#1}\mkbf{#1}\mkcal{#1}\mkscr{#1}\mkbs{#1}\expandafter\makeA\fi}%
\def\makea#1,{\def\temp{#1}\ifx\temp\endpiece\else\mkfrak{#1}\mkbf{#1}\mkbs{#1}\expandafter\makea\fi}%
\def\mkbb#1{\expandafter\def\csname bb#1\endcsname{\mathbb{#1}}}
\def\mkfrak#1{\expandafter\def\csname fr#1\endcsname{\mathfrak{#1}}}
\def\mkbf#1{\expandafter\def\csname b#1\endcsname{\mathbf{#1}}}
\def\mkcal#1{\expandafter\def\csname c#1\endcsname{\mathcal{#1}}}
\def\mkscr#1{\expandafter\def\csname s#1\endcsname{\mathscr{#1}}}
\def\mkbs#1{\expandafter\def\csname bs#1\endcsname{{\boldsymbol{#1}}}}
\def\makeop[#1]{\xmakeop#1,xxx,}
\def\mkop#1{\expandafter\def\csname #1\endcsname{{\mathrm{#1}}}} %
\def\xmakeop#1,{\def\temp{#1}\ifx\temp\endpiece\else\mkop{#1}\expandafter\xmakeop\fi}%
\def\makeup[#1]{\xmakeup#1,xxx,}
\def\mkup#1{\expandafter\def\csname #1\endcsname{{\mathrm{#1}\,}}} %
\def\xmakeup#1,{\def\temp{#1}\ifx\temp\endpiece\else\mkup{#1}\expandafter\xmakeup\fi}%
\makeAlphabet[\Alphabet]
\makealphabet[\grabet]
\makeop[Map,Hom,alt,id,loc,unif,len,Supp,Consv,pr,Aut,Ext,ab,tors,Coind,univ]
\makeup[Im,Ker,Coker]
\def\state{\eta}
\def\bsxi{{\boldsymbol{\xi}}}
\def\State{S^X}

\def\bsone{\boldsymbol{1}}

\def\Group{G}
\def\gra{\alpha}
\def\grb{\beta}
\def\grc{\gamma}
\def\grs{\sigma}
\def\grt{\tau}
\def\grhom{\pi}
\def\diam#1{\operatorname{diam}(#1)}
\def\La{\Lambda}

\def\subl{Y}
\def\sublc{Y^c}

\def\ic{\hat\imath}
\def\CS{C(\State_*)}
\def\CuS{C^0_\unif(\State)}

\def\nabe{\nabla_{\!e}}
\def\FC{\sF\kern-0.5mm\sC}

\def\Ba{B}
\def\bsx{x}
\def\bsy{y}

\def\indj{j}
\def\indi{i}

\def\Val{\cM}

\begin{document}

\setcounter{tocdepth}{1}
\newpage
\title[Topological Structures of Large Scale Interacting Systems]{Topological Structures of Large Scale Interacting Systems via Uniform Functions and Forms}
\author[Bannai]{Kenichi Bannai}\email{bannai@math.keio.ac.jp}
\author[Kametani]{Yukio Kametani}\email{kametani@math.keio.ac.jp}
\author[Sasada]{Makiko Sasada}\email{sasada@ms.u-tokyo.ac.jp}
\thanks{This work was supported in part by JST CREST Grant Number JPMJCR1913, KAKENHI 18H05233,
and the UTokyo Global Activity Support Program for Young Researchers.}
\address[Bannai, Kametani]{Department of Mathematics, Faculty of Science and Technology, Keio University, 3-14-1 Hiyoshi, Kouhoku-ku, Yokohama 223-8522, Japan.}
\address[Sasada]{Department of Mathematics, University of Tokyo, 3-8-1 Komaba, Meguro-ku, Tokyo 153-0041, Japan.}
\address[Bannai, Sasada]{Mathematical Science Team, RIKEN Center for Advanced Intelligence Project (AIP),1-4-1 Nihonbashi, Chuo-ku, Tokyo 103-0027, Japan.}

\date{\today}
\begin{abstract}
	In this article, we investigate the topological structure of
	large scale interacting systems on infinite graphs,
	by constructing a suitable cohomology 
	which we call the \textit{uniform cohomology}.
	The central idea for the construction is the introduction of a class of functions called \emph{uniform functions}.
	Uniform cohomology provides a new perspective for the identification of macroscopic observables
	from the microscopic system.
	As a straightforward application of our theory 
	when the underlying graph has a free action of a group, 
	we prove a certain decomposition theorem for shift-invariant closed uniform forms.
	This result is a uniform version in a very general setting of the 
	decomposition result for shift-invariant closed $L^2$-forms originally proposed by Varadhan,
	which has repeatedly played a key role in the proof of the
	hydrodynamic limits of \textit{nongradient} large scale
	interacting systems.  In a subsequent article \cite{BS21L2}, we use this result as a key to 
	prove Varadhan's decomposition theorem for a general class of large scale interacting systems.
\end{abstract}

\subjclass[2020]{Primary: 82C22, Secondary: 05C63, 55N91, 60J60, 70G40} 
\maketitle

\tableofcontents

%
%
%
\section{Introduction}
%
%
%

%
\subsection{Introduction}\label{subsec: introduction}
%

One of the fundamental problems in the natural and social sciences 
is to explain macroscopic phenomena that we can observe
from the rules governing the microscopic system
giving rise to the phenomena. \textit{Hydrodynamic limit} provides a 
rigorous mathematical method to derive
the deterministic partial differential equations describing the time evolution of macroscopic parameters, 
from the stochastic dynamics of a microscopic large scale interacting system.
The heart of this method is to take the limit with respect to proper space-time
scaling, so that the law of large numbers absorbs the large degree of freedom of 
the microscopic system, allowing to extract the behavior
of the macroscopic parameters which characterize the equilibrium states of the microscopic system.
Hence, techniques from probability theory including various estimates on Markov processes and their stationary
distributions have played a central role.
In this article, we introduce a novel, geometric
perspective to the theory of hydrodynamic limits.
Instead of using the law of large numbers, we construct a new cohomology theory for
microscopic models to identify the macroscopic observables and give 
interpretations to the mechanism giving the macroscopic partial differential equations.
Our main theorem 
gives an analogue of Varadhan's decomposition of closed $L^2$-forms,
which has played a key role in the proofs of the
hydrodynamic limits of \textit{nongradient} 
systems.  

Initially, many of the techniques developed in the theory of hydrodynamic limit
were specific to the 
interacting system under consideration.
In the seminal article \cite{GPV88}, Guo, Papanicolaou and Varadhan
introduced a widely applicable strategy known as the \textit{entropy method}
 for proving the hydrodynamic limit
when the interacting system satisfies a certain condition known as the \textit{gradient condition}.
Furthermore, Varadhan in \cite{Var93} introduced 
a novel, refined strategy for 
systems which do not necessarily satisfy the gradient condition,
relying on proving the so-called \textit{decomposition of closed $L^2$-forms}.
Although this strategy has been successful in proving the hydrodynamic
limit for a number of 
nongradient systems \cites{Qua92,KLO94,FUY96,VY97,Sas10,Sas11,OS13},
the implementation in practice has proven notoriously
difficult, requiring arguments with sharp spectral gap
estimates specific to the system under consideration (see for example \cite{KL99}*{Section 7}).
Due to the restrictiveness of the gradient condition,
many interesting microscopic systems are known or expected to be nongradient.
Thus it is vital to understand the mechanism of Varadhan's strategy and construct model independent 
criteria for implementation applicable to a wide variety of models.


The motivation of this article is
to systematically investigate various large scale interacting systems
in a unified fashion, 
especially to understand the mechanism in which similar decompositions
seemingly independent of the stochastic data appear
 in the proofs of the hydrodynamic limits.
 For this goal, we introduce a general framework encompassing
a wide variety of interacting systems,
including the different variants of the 
exclusion processes and the lattice-gas with energy
(see Examples \ref{example: interactions} and \ref{example: interactions2}).
We let $X$ be a certain infinite graph which we call a \textit{locale}, generalizing the
typical Euclidean lattice modeling the space where the
microscopic dynamics takes place. 
We let $S$ be a set expressing the possible states at each vertex,
such as the number of particles or amount of energy,
and let $\State\coloneqq\prod_{x\in X}S$
be the \textit{configuration space} expressing all of the possible configurations of states on $X$. 
The dynamics of a microscopic stochastic system is usually expressed by a generator.
However, in our framework, we focus on the \textit{interaction} $\phi$  --
a certain map $\phi\colon S\times S\rightarrow S\times S$
encoding the permitted change in states on adjacent vertices (see Definition \ref{def: interaction}).
The interaction gives $\State$ a geometric structure, that of a graph whose edges correspond to the 
\textit{transitions}, i.e., all possible change of the configuration at a single instant
 (see \S \ref{subsec: model}).  
This structure is independent of the
\textit{transition rate} -- stochastic data which encodes the expected frequency of the transitions.


In this article, we construct the \textit{uniform cohomology} 
reflecting the topological property of the geometric structure of $\State$,
by replacing the space of functions on $\State$ with a new class of functions called the 
\textit{uniform functions}, which considers the distances between the vertices of the locale.
Our key result, Theorem \ref{thm: B}, states that under general assumptions,
the \textit{zeroth} uniform cohomology is isomorphic to
the space of conserved quantities -- functions on $S$ whose sums are conserved by $\phi$.
This cohomology is finite dimensional even though $\State$ in general
has an infinite number of connected components.
For the cases where the hydrodynamic limit is proven, conserved quantities 
are known to correspond to the macroscopic parameters which characterize the equilibrium (or stationary) measures
of the microscopic system.
Thus, we believe uniform cohomology 
gives an alternative justification
for the origin of the macroscopic observables.
In addition, Theorem \ref{thm: B} also states that the uniform cohomology of $\State$
 for degrees other than \textit{zero} vanish.
The essential case is for degree \textit{one}, where we prove that any 
closed uniform form is the differential of a uniform function.

Our main theorem, Theorem \ref{thm: A}, gives a 
certain structure theorem for closed uniform forms
that are shift-invariant, i.e.,\
invariant by the action of a group.
Here, we assume the existence of a free action of a group on the locale,
which ensures a certain homogeneity. 
The theorem is obtained as a straightforward application of group cohomology to Theorem \ref{thm: B}.
If we choose a fundamental domain of $X$ for the action of the group,
then we obtain a decomposition theorem in the spirit of the decomposition of Varadhan
(see Theorems \ref{thm: V},  \ref{thm: A1} and \ref{thm: A2} of \S \ref{subsec: main}).
The closed forms of Varadhan are $L^2$-forms for the equilibrium measure
arising from the choice of the transition rate.
Although uniform functions and forms are defined algebraically without the need for any 
stochastic data, 
our shift-invariant forms in fact form a common core of the
various spaces of shift-invariant $L^2$-forms constructed for each choice of the transition rate,
and will in subsequent research play a crucial role in proving Varadhan's decomposition for $L^2$-forms
(see for example \cite{BS21L2}).
Our main theorem indicates that the specification of Varadhan's 
decomposition is determined by the underlying geometric
structure of the model.  Moreover, our theory gives a cohomological 
interpretation of the dimension of the space of 
shift-invariant closed forms modulo the exact forms -- whose origin up until now had been a mystery.
The proof of our main theorem does not require 
any spectral gap estimates
and can be applied universally to a wide variety of systems.


Currently, all existing research concerning hydrodynamic limits  
for nongradient systems deal exclusively with the case when the locale is the Euclidean lattice $\bbZ^d$,
with an action of 
$\Group=\bbZ^d$ given by the 
translation.
Our decomposition theorem 
is valid for far more general infinite locales and groups, including 
various crystal lattices with their group of translations and Cayley graphs associated
to finitely generated infinite groups with natural action of the group.
The theorem is also true for systems with multiple linearly independent conserved quantities.
Our result provides crucial insight into the formulation of
Varadhan's decomposition in these general settings.
One of our goals is to find a more intuitive and universal proof of the hydrodynamic limits for nongradient models.
In a subsequent article \cite{BS21L2}, we prove 
Varadhan's decomposition theorem for a general class of models, where the main result of the present article 
is the key for the generalization. Furthermore, we will use this decomposition theorem to perform scaling limits for the general class of models.

Our theory is constructed from scratch, using only algebraic and combinatorial methods.
In particular, no probability theory, measure theory, or analytic methods are used.
Most importantly, we have taken care to make this article including the proof of our main result
\textit{self-contained}, except for the proof of the 
well established long exact sequence arising in group cohomology (see \S\ref{subsec: group}).
Thus, we believe our article should be accessible to mathematicians in a wide 
range of disciplines.
We hope this article would introduce to a broad audience 
interesting mathematical concepts related to typical large scale 
interacting systems, and to researchers in probability theory 
potentially powerful cohomological techniques that may be relevant in identifying
important structures of stochastic models.

The remainder of this section is as follows.
In \S\ref{subsec: model}, we describe our framework and present some examples.
Then in \S\ref{subsec: main}, we state Theorem \ref{thm: A}, 
the main theorem of our article, asserting the decomposition for shift-invariant uniform
closed forms.  We then explain its relation to the decomposition by Varadhan.
Finally, in \S\ref{subsec: overview}, we provide an overview of our article
and the outline of the proof of our main theorem.

%
\subsection{The Large Scale Interacting System}\label{subsec: model}
%

In this subsection, we introduce the various objects in our framework describing
large scale interacting systems and give natural assumptions which ensure our main theorem.
The precise mathematical definitions of the objects in the triplet $(X,S,\phi)$ given in 
\S\ref{subsec: introduction} are as follows.
We define a locale $(X,E)$ to 
be any locally finite simple symmetric directed graph which is connected 
(see \S\ref{subsec: configuration} for details). 
Here, $X$ denotes the set of vertices and $E\subset X\times X$ 
denotes the set of directed edges of the locale. We regard a locale $(X,E)$ as a metric space equipped with the graph distance. By abuse of notation (see Remark \ref{rem: abuse}), we will 
often denote the locale $(X,E)$ with the same symbol as its set of vertices $X$.
The condition that $X$ is connected and locally finite implies that the set of vertices of $X$ is countable.
If the set of vertices is an infinite set, then we say that $X$ is an \textit{infinite locale}.
We define the set of states $S$ as a nonempty set with a designated element $*\in S$
which we call the \textit{base state}, and we define the \textit{symmetric binary interaction}, 
or simply an interaction $\phi$ on $S$ to be a 
 map $\phi\colon S\times S\rightarrow S\times S$
such that
for any pair of states $(s_1,s_2)\in S\times S$ satisfying $\phi(s_1,s_2)\neq(s_1,s_2)$,  
we have  $\ic\circ\phi\circ\ic\circ\phi(s_1,s_2)=(s_1,s_2)$,
where $\ic\colon S\times S\rightarrow S\times S$
is the bijection obtained by exchanging the components of $S\times S$.
The ordering of $S\times S$ determines the direction of the interaction,
and the condition intuitively means that if we execute the interaction and if it is nontrivial,
then further executing the interaction in the reverse direction takes us back to where we started.
To realize the full large scale interacting system, we also need to 
choose a transition rate.  However, this is outside the scope of the current article.

The most typical example of an infinite locale is given by the \textit{Euclidean lattice} 
$\bbZ^d=(\bbZ^d,\bbE^d)$ for integers $d\geq 1$, where
$\bbZ^d$ is the $d$-fold product of the set of integers $\bbZ$, and
\[
	\bbE^d\coloneqq \bigl\{  (\bsx,\bsy) \in \bbZ^d\times\bbZ^d \,\,\big|\,\,  |\bsx-\bsy|=1\bigr\}.
\]
Here, we let $|\bsx-\bsy|\coloneqq\sum_{\indj=1}^d|x_\indj-y_\indj|$ for any 
$\bsx=(x_1,\ldots,x_d)$, $\bsy=(y_1,\ldots,y_d)$ in $\bbZ^d$.
Crystal lattices such as the triangular and hexagonal lattices as well as Cayley graphs
associated to finitely generated infinite groups  (see Figure \ref{fig: 1})
are other examples of infinite locales.
\begin{figure}[ht]
	\centering
	\includegraphics[width=13cm]{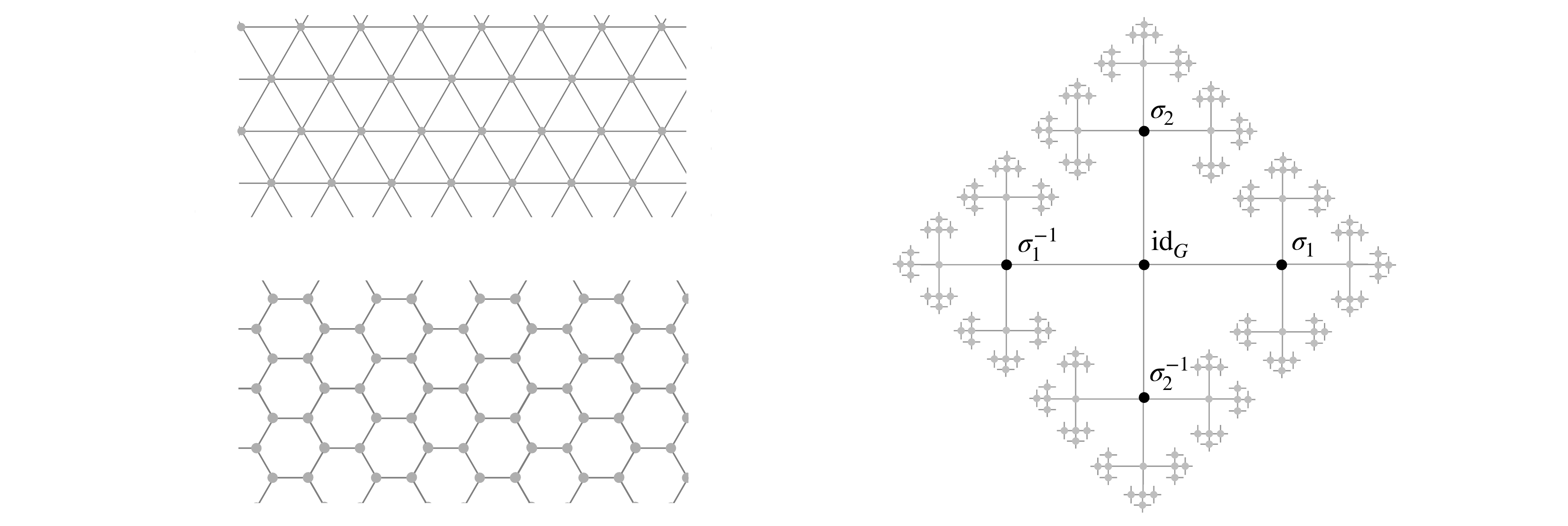}
	\caption{The Triangular and Hexagonal Lattices, and the Cayley graph for the free group
	$G$ generated by $\grs_1$ and $\grs_2$ (see Example \ref{example: locale} (4))}\label{fig: 1}
\end{figure}

We say that a locale is \textit{weakly transferable}, if for any ball $\Ba\subset X$,
the complement $X\setminus\Ba$ is a nonempty finite disjoint union of connected infinite graphs. 
By definition, a weakly transferable locale is an infinite locale.
We will also consider a stronger condition on the locale which we call \textit{transferable}
(see Definition \ref{def: transferable} for the precise definition).
Immediately from the definition, we see that if for any ball $\Ba\subset X$,
the complement $X\setminus\Ba$ is a connected infinite graph,
then $X$ is transferable.  The Euclidean lattice $\bbZ^d=(\bbZ^d,\bbE^d)$ for $d>1$, crystal lattices
such as the triangular and hexagonal lattices, as well as the Cayley graph
for a finitely generated free group generated by $d>1$ elements give examples of transferable locales
(see Remark \ref{rem: Cayley}).
The Euclidean lattice $\bbZ=(\bbZ,\bbE)$ for $d=1$, which is also 
the Cayley graph for a free group generated by one element, gives an example of a weakly transferable locale
which is not transferable.
See Example \ref{example: locale} in \S \ref{sec: configuration} for other examples of locales.

{We call the triplet $(X,S,\phi)$ as above a \textit{topological interacting system},
or simply a \emph{system} for short.}
For the system $(X,S,\phi)$, we define the configuration space $\State$ by 
\[
	\State\coloneqq\prod_{x\in X} S.
\]
We call any element $\state\in\State$ a \textit{configuration},
and we denote by $\star$ the \textit{base configuration}, defined to be the
configuration whose components are all at base state.
Next, for any $\state\in\State$ and $e=(x_1,x_2)\in E$, we let $(\eta'_{x_1}, \eta'_{x_2}) =\phi(\eta_{x_1},\eta_{x_2})$.
We define $\state^e=(\eta_x^e)\in\State$ by
\begin{equation}\label{eq: se}
	\eta_x^e\coloneqq
	\begin{cases}
		\eta_x  &   x\neq x_1,x_2\\
		\eta'_x & x=x_1,x_2.\\
	\end{cases}
\end{equation}
The transition structure on $\State$,
expressing all possible change of the configuration at a single instant,
is defined as
$\Phi\coloneqq\{(\state,\state^e)\mid\state\in\State, e\in E\}\subset\State\times\State$.
Then $(\State,\Phi)$ is a symmetric directed graph (see Lemma \ref{lem: SDG} for a proof).
Again by abuse of notation, we will often denote the graph $(\State,\Phi)$ by $\State$ 
(see also Remark \ref{rem: abuse}).  We remark that generally, 
$\State$ on an infinite locale  is not connected,
simple nor locally finite as a graph, and the set of vertices is not countable.


Next, we introduce the \textit{conserved quantity} for the interaction $\phi$,
which we define to be any function
$\xi\colon S\rightarrow\bbR$ satisfying $\xi(*)=0$ and
\begin{equation}\label{eq: conserve}
	\xi(s_1)+\xi(s_2)=\xi(s'_1)+\xi(s'_2)
\end{equation}
for any $(s_1,s_2)\in S\times S$, where $(s'_1,s'_2)\coloneqq\phi(s_1,s_2)$.
This function gives values reflecting the state, 
such as the number of particles or energy depending on the physical context, whose sums are 
conserved by the interaction.
We denote by $\Consv^\phi(S)$ the $\bbR$-linear space of all conserved quantities
for the interaction $\phi$.  

Denote by
$\State_*$ the subset of $\State$ consisting of elements $\state=(\eta_x)$ such that $\eta_x=*$ for all but finite $x\in X$.
Then $\State_*$ also has a structure of a graph induced from that of $\State$.
Note that If $X$ is finite, then $S^X_*=S^X$.
Any conserved quantity $\xi\in\Consv^\phi(S)$ defines a function $\xi_X\colon\State_*\rightarrow\bbR$
by $\xi_X(\state)\coloneqq\sum_{x\in X}\xi(\eta_x)$ for any $\state\in\State_*$.
Note that the sum is a finite sum since $\state\in\State_*$.
We call the value $\xi_X(\state)$ a conserved quantity of the configuration $\state$.

Throughout this article,
we let $\bbN=\{0,1,\ldots,\}$ denote the set of natural numbers.
We consider the following properties of an interaction, which will play an important
role in our main theorem.

\begin{definition}\label{def: conditions}
	For an interaction $\phi$ on $S$, let $c_\phi\coloneqq\dim_\bbR\Consv^\phi(S)$.
	\begin{enumerate}
	\item
	We say that the interaction $\phi$ is \textit{irreducibly quantified},
	if for any finite locale $X$, if the configurations
	$\state,\state'\in\State$ have
	the same conserved quantities, i.e.\ if $\xi_X(\state)=\xi_X(\state')$ for any 
	$\xi\in\Consv^\phi(S)$, then there exists a finite path (see \S\ref{subsec: configuration})
	from $\state$ to $\state'$ in $\State$.
	\item 
	We say that the interaction $\phi$ is \textit{simple}, if $c_\phi=1$,
	and for any nonzero conserved quantity $\xi\in\Consv^\phi(S)$,
	the monoid generated by $\xi(S)$ via addition in $\bbR$ 
	is isomorphic to $\bbN$ or $\bbZ$.
	\end{enumerate}
\end{definition}

A \textit{monoid} is defined to be a set with a binary operation 
that is associative and has an identity element, the first examples being
$\bbN$, $\bbZ$ or $\bbR$ with the operation being the usual addition and with identity element $0$.
Provided $c_\phi=1$, the second condition in Definition \ref{def: conditions} (2)
is satisfied for example if there exists $\xi\in\Consv^\phi(S)$ such that $\xi(S)\subset\bbN$ and $1\in\xi(S)$,
or  $\xi(S)\subset\bbZ$ and $\pm1\in\xi(S)$.

\begin{remark}
	The second condition of Definition \ref{def: conditions} (1) implies that
	any configurations with the same conserved quantities
	are in the same connected component of
	the configuration space $\State_*$.
	The configuration space on an infinite locale $X$ usually has an infinite number 
	of connected components.  If the above condition is satisfied,
	then we may prove that the connected components of $\State_*$ are
	characterized by its conserved quantities (see Remark \ref{rem: intrinsic}). 
	This condition is equivalent 
	to the condition that the associated stochastic process on the configurations 
	with fixed conserved quantities are \textit{irreducible}.
\end{remark}

The following are examples of interactions and corresponding conserved quantities.

\begin{example}\label{example: interactions}
	\begin{enumerate}
		\item The most basic situation is when $S=\{0,1\}$ with base state $*=0$.
		The map $\phi\colon S\times S\rightarrow S\times S$ defined by exchanging the 
		components of $S\times S$ is an interaction (see Figure \ref{fig: 2}). 
		The conserved quantity $\xi\colon S\rightarrow\bbN$ given by $\xi(s)=s$
		gives a basis of the one-dimensional $\bbR$-linear space $\Consv^\phi(S)$.
		This interaction is simple.
		The stochastic process induced from this interaction via a choice of a transition rate
		is called the \textit{exclusion process}.
		\begin{figure}[ht]
			\centering
			\includegraphics[width=13cm]{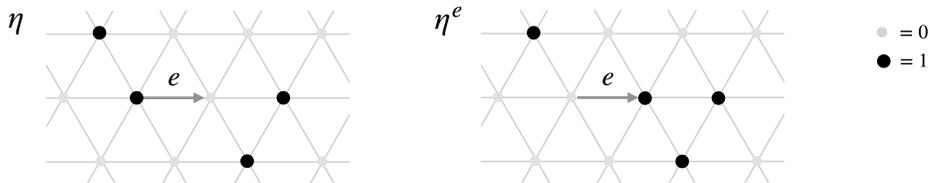}
			\caption{Exclusion Process}\label{fig: 2}
		\end{figure}
		\item Consider the case 
		$S=\{0,1,\ldots,\kappa\}$ with base state $*=0$ for some integer $\kappa>1$.
		The map $\phi\colon S\times S\rightarrow S\times S$ defined by exchanging the 
		components of $S\times S$ is an interaction (see Figure \ref{fig: 3}).  	
		For $\indi=1,\ldots,\kappa$, let $\xi^{(\indi)}$ be the conserved quantity given by
		$\xi^{(\indi)}(s)=1$ if $s=\indi$ and $\xi^{(\indi)}(s)=0$ otherwise.		
		Then $\xi^{(1)},\ldots,\xi^{(\kappa)}$ gives a basis of the $\bbR$-linear space $\Consv^\phi(S)$.
		The stochastic process induced from this interaction via a choice of a transition rate
		is called the \textit{multi-color exclusion process},
		or more generally, the \textit{multi-species exclusion process}.	
		\begin{figure}[ht]
			\centering
			\includegraphics[width=13cm]{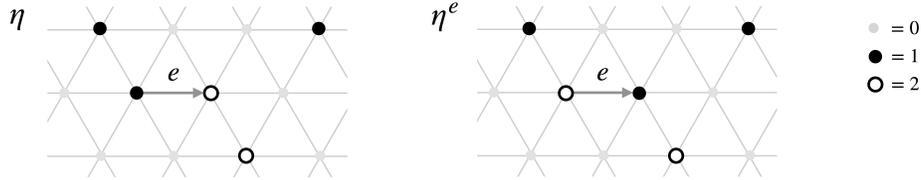}
			\caption{Multi-species Exclusion Process}\label{fig: 3}
		\end{figure}
		\end{enumerate}
		See Example \ref{example: interactions2} in \S\ref{subsec: CQ} for other examples of interactions
		covered by our theory.
		We will prove in Proposition \ref{prop: irreducibly quantified} that all of the interactions 
		in Examples \ref{example: interactions} and \ref{example: interactions2} are irreducibly quantified.
\end{example}

\begin{remark}\label{rem: interactions}
	For $s\in S$ in the interactions in Example \ref{example: interactions}, 
	$s=0$ describes the state where there are no particles
	at the vertex,
	and $s=i$ for an integer $i>0$ the state where there exists a particle of type labeled as $i$
	(referred to as color $i$ or species $i$)
	at the vertex.
	The \textit{exclusion} in the exclusion and the multi-species exclusion processes signify that 
	at most one particle is allowed to occupy each vertex.
	The conserved quantity $\xi^{(\indi)}$ returns $1$ if a particle of species $i$ occupies the vertex and $0$
	otherwise.  Then $\xi^{(\indi)}_{X}(\state)\coloneqq\sum_{x\in X}\xi^{(\indi)}(\eta_x)$ for a configuration $\state=(\eta_x)\in\State_*$
	on a locale $X$ expresses the total number of particles of species $k$ in the configuration.
\end{remark}

The hydrodynamic limits of the exclusion process for certain choices of transition rates of nongradient type
have been studied by Funaki, Uchiyama and Yau \cite{FUY96} and Varadhan and Yau \cite{VY97}
(see also Theorem \ref{thm: V} of \S\ref{subsec: main}).
For the multi-species exclusion process, a variant has been studied by
Quastel \cite{Qua92} and Erignoux \cite{Eri20}.
See Remark \ref{rem: hydrodynamic limit2} for other known cases
corresponding to the interactions given in Example \ref{example: interactions2}.
Up until now, all of the interacting systems of nongradient type whose hydrodynamic limits
have been proved are models over the Euclidean lattice.

Typical research in hydrodynamic limit investigates
the stochastic process of large scale interacting systems 
obtained from a specific interaction with a specific transition rate
on a specific locale.
One important purpose of this article is to construct a mathematical framework to 
study many types of models at once and to find specific conditions on the 
locale and interactions to allow for a suitable theory.  The notion of transferable locales
and irreducibly quantified interactions, which we believe are new and have not previously appeared in literature,
are steps in this direction.  The distinctive feature of our framework is the separation of the stochastic data
from the geometric data, as well as the separation of the set of states and the interaction from the underlying locale.
The theory works best when $S$ is discrete,
a case which already covers a wide variety of models.  In future research, we hope to generalize
our framework to include known models with more general $S$, such as $S=\bbR$ and $S=\bbR_{\geq0}$,
where a more subtle notion of uniform functions incorporating smoothness 
should be necessary for compatibility with existing models.

%
\subsection{Main Theorem and Relation to Varadhan's Decomposition}\label{subsec: main}
%

The goal of our article is to study 
the topological properties of the configuration space $\State$ with transition structure
via a newly defined class of uniform functions and forms.
In this subsection, we introduce Theorem \ref{thm: A}, 
which is the main theorem of this article,
giving a decomposition of shift-invariant closed uniform forms.
We will then discuss its relation to Varadhan's decomposition
of shift-invariant closed $L^2$-forms.

We first introduce notations concerning functions and forms on the configuration space with transition structure.
Consider the system $(X,S,\phi)$, and let $\State$ be the corresponding configuration space with transition structure.
For any set $A$, we let $C(A)\coloneqq\Map(A,\bbR)$ be the $\bbR$-linear space of functions from $A$ to $\bbR$.
We say that a function $f\in C(\State)$ is \textit{local}, if there exists a finite $\La\subset X$ such that
$f$ is in the image of $C(S^\La)$ with respect to the inclusion 
$C(S^\La)\hookrightarrow C(\State)$ induced 
from the projection $\State\rightarrow S^\La$. 
In other words, a function is local if it only depends on a configuration through a finite sublocale.
Any local function may be 
regarded as a function in $\CS$ via the map induced from the projection 
$\State_*\rightarrow S^\La$.
We denote by $C_\loc(\State)$ the space of local functions on $\State$,
which is a subspace of both $C(\State)$ and $\CS$.
We define the space of \textit{uniform functions} (see Definition \ref{def: uniform} for the precise definition) to be a certain 
$\bbR$-linear subspace $C_\unif(\State)$ of $\CS$ 
containing the space of local functions $C_\loc(\State)$,
and we let $C^0_\unif(\State)$ be the subspace of $C_\unif(\State)$
consisting of function $f$ satisfying $f(\star)=0$.
We define the space of closed uniform forms $Z_\unif^1(\State)$  (see Definitions \ref{def: closed} and \ref{def: unif form} for the precise definition) to be a certain 
$\bbR$-linear subspace of $\prod_{e\in E}C_\loc(\State)$, and
we define the differential $\partial\colon C^0_\unif(\State)\rightarrow Z_\unif^1(\State)$ by
$
	\partial f\coloneqq (\nabe f),
$
where $\nabe f$ for any $e\in E$
is the function defined by 
\begin{equation}\label{eq: differential}	
	\nabe f(\state)\coloneqq f(\state^e)-f(\state)
\end{equation}
for any $f\in C^0_\unif(\State)$ and $\state\in\State_*$.  
The differential $\partial$ is induced from the differential
of the standard cochain complex associated with the graph $(\State,\Phi)$ (see \S\ref{subsec: naive} and Appendix \ref{sec: A}).

For our main theorem,
we consider a locale with a free action of a group.
Let $\Group$ be a group, and we assume that the locale $X$ has a free action of a group $\Group$.
This induces actions of $\Group$ on various functions and forms.
Both the Euclidean lattice and crystal lattices such as the triangular and hexagonal lattices
of dimension $d$
have natural free actions of $\Group=\bbZ^d$.
For any $\bbR$-linear space with an action of $\Group$, we denote by $U^\Group$ the $\Group$-invariant 
subspace of $U$.  We will often say \textit{shift-invariant} to mean $G$-invariant
if the group $G$ is understood.
We denote by $\cC\coloneqq (Z^1_\unif(\State))^\Group$ 
the space of shift-invariant closed uniform forms, and by
$\cE\coloneqq \partial\bigl(C^0_\unif(\State)^\Group\bigr)$
the image by $\partial$ of the space of shift-invariant uniform functions.

\begin{remark}
	The existence of a free action of $\Group$ ensures
	that the locale $X$ is homogenous.
	We understand the quotient $\cC/\cE$ to philosophically represent the first uniform 
	cohomology of the quotient space
	$\State/\Group$, a topological space which can be interpreted as a model of an infinitesimal neighborhood
	of a macroscopic point.
\end{remark}

We denote by $\Hom(G,\Consv^\phi(S))$ the space of group homomorphisms from
$G$ to $\Consv^\phi(S)$.  In other words, $\psi\in\Hom(G,\Consv^\phi(S))$
is any map $\psi\colon G\rightarrow\Consv^\phi(S)$ satisfying $\psi(\sigma\tau)=\psi(\sigma)+\psi(\tau)$
for any $\sigma,\tau\in G$.
Our decomposition theorem giving the uniform analogue of Varadhan's decomposition
is as follows.

\begin{thmint}[=Theorem \ref{thm: main}]\label{thm: A}
	For the system $(X,S,\phi)$, assume that the interaction $\phi$
	is irreducibly quantified, and that $X$ has a free action of a group
	$\Group$. If $X$ is transferable, or if the interaction $\phi$ is simple and 
	$X$ is weakly transferable, then we have a canonical isomorphism	
	\begin{equation}\label{eq: delta0}
		 \cC/\cE\cong\Hom(G,\Consv^\phi(S)).
	\end{equation}
	Moreover, a choice of a fundamental domain for the action of $G$ on $X$
	gives a natural decomposition
	\[
		 \cC\cong \cE\oplus \Hom(G,\Consv^\phi(S))
	\]
	of $\bbR$-linear spaces.
\end{thmint}

If the rank of the maximal abelian quotient $\Group^\ab$ of $\Group$ is finite,
then we have the following.

\begin{corint}[=Corollary \ref{cor: main}]\label{cor: A}
	Let the assumptions be as in Theorem \ref{thm: A}.  Moreover,  suppose that 
	$G^\ab$ is of finite rank $d$.
	If we fix a generator of the free part of $G^\ab$,
	then we have an isomorphism
	$
		\Hom(G,\Consv^\phi(S))
		 \cong \bigoplus_{\indj=1}^d\Consv^\phi(S).
	$
	A choice of a fundamental domain  of $X$
	for the action of $\Group$
	gives a decomposition
	\begin{equation}\label{eq: V0}
		 \cC\cong \cE\oplus
		\bigoplus_{\indj=1}^d\Consv^\phi(S).
	\end{equation}
\end{corint}
The decomposition \eqref{eq: V0} decomposes any shift-invariant closed uniform form in $\cC$
as a unique sum of a form in $\cE$, closed forms whose potential are shift-invariant uniform functions,
and a form obtained as the image with respect to the isomorphism \eqref{eq: V0}
of elements in $\bigoplus_{\indj=1}^d\Consv^\phi(S)$.
The space $\cE$ corresponds to the part which averages out to \textit{zero} when taking 
a proper space-time scaling limit and so does not appear in the hydrodynamic limit. 
Hence the decomposition theorem implies that the
macroscopic property of our model may be completely expressed
in terms of forms arising from the space $\bigoplus_{\indj=1}^d\Consv^\phi(S)$, 
which 
are related to the flow of
conserved quantities in each of the directions induced by the action of the group $\Group$.

In addition to the geometric data $(X,S,\phi)$, if we fix a suitable transition rate,
then this gives a shift-invariant equilibrium measure on the configuration space
and a compatible inner product on the space of forms.
If we consider the case when the locale is the Euclidean lattice 
$X=(\bbZ^d,\bbE^d)$ with standard translation by the group $G=\bbZ^d$, and if 
$(S,\phi)$ is the exclusion process of
Example \ref{example: interactions} (1),
a typical choice of a transition rate gives rise to the 
product measure $\nu=\mu_p^{\otimes \bbZ^d}$ on $\State=\{0,1\}^{\bbZ^d}$, where $\mu_p$ is the probability measure on $S=\{0,1\}$ given by
\[
	\mu_{p}(s=1)=p, \qquad \mu_{p}(s=0)= 1-p
\]
for some real number $0<p<1$.  
We denote by $L^2(\nu)$ the usual 
$L^2$-space of square integrable functions 
on $\{0,1\}^{\bbZ^d}$ with respect to the measure $\nu$.
The space of local functions $C_\loc(\{0,1\}^{\bbZ^d})$ is known to be a dense subspace of $L^2(\nu)$.
We let $\xi\colon \{0,1\}\rightarrow\bbN$ be the conserved quantity given by $\xi(s)=s$,
which gives a basis of $\Consv^\phi(\{0,1\})$.
For any $x\in X$, we let $\xi_x\colon\{0,1\}^{\bbZ^d}\rightarrow\bbR$ be the function defined by 
$\xi_x(\state)\coloneqq \eta_x$ for any $\state=(\eta_x)\in\{0,1\}^{\bbZ^d}$.
For any $x=(x_j)\in\bbZ^d$, denote by $\tau_x$ the translation of $(\bbZ^d,\bbE^d)$ by $x$.
In this case, Varadhan's decomposition of shift-invariant 
closed $L^2$-forms proved by Funaki, Uchiyama and Yau
is the following.

\begin{thmint}[\cite{FUY96}*{Theorem 4.1}]\label{thm: V}
	Let $\omega =(\omega_{e})\in \prod_{e\in E} L^2(\nu)$ be a shift-invariant closed $L^2$-form.
	Then there exists a set of constants $a_1,\ldots,a_d\in\bbR$ and a
	series of local functions $(f_n)_{n\in\bbN}$ in $C_\loc(\{0,1\}^{\bbZ^d})$
	such that
	\[
		\omega_e=\lim_{n\rightarrow\infty}\nabla_e\bigg(\sum_{x\in\bbZ^d}\tau_x(f_n)
		+\sum_{j=1}^da_j\sum_{x\in\bbZ^d} x_j\xi_x\bigg)
	\]
	in $L^2(\nu)$ for any $e\in\bbE$.
\end{thmint}

The same statement for certain transition rates giving non-product measures on $\{0,1\}^{\bbZ^d}$
was proved by Varadhan and Yau \cite{VY97}, requiring different spectral gap estimates.
The uniform version of Theorem \ref{thm: V}, obtained by applying Corollary \ref{cor: A}
to the above model for the fundamental domain $\La_0=\{(0,\ldots,0)\}$ of $X=\bbZ^d$ for the action of 
$\Group=\bbZ^d$, is given as follows.

\begin{thmint}\label{thm: A1}
	Let $\omega =(\omega_{e})\in \prod_{e\in\bbE}C_\loc(\{0,1\}^{\bbZ^d})$ be a shift-invariant 
	closed form.
	Then there exists a set of constants $a_1,\ldots,a_d\in\bbR$ and a
	local function $f$ in $ C_\loc(\{0,1\}^{\bbZ^d})$ satisfying $f(\star)=0$
	such that
	\[
		\omega_e=\nabla_e\biggl(\sum_{x\in\bbZ^d}\tau_x(f)
		+\sum_{j=1}^da_j\sum_{x\in\bbZ^d} x_j\xi_x\biggr)
	\]
	in $C_\loc(\{0,1\}^{\bbZ^d})$ for any $e\in\bbE$.
\end{thmint}

Our proof does not require a choice of the transition rate, thus completely independent of the measure.
{We remark that the sums in the brackets on the right hand side of Theorem \ref{thm: V} 
and Theorem \ref{thm: A1} are uniform functions in $C^0_\unif(\{0,1\}^{\bbZ^d})$.}

Next, consider a general system $(X,S,\phi)$
satisfying the assumptions of Theorem \ref{thm: A}, and suppose
that $X$ has a free action of $\Group=\bbZ^d$.  We fix the generator of $\Group$
to be the standard basis of $\bbZ^d$, and we denote an element of $\Group=\bbZ^d$ by $\tau=(\tau_j)\in\bbZ^d$
instead of $x=(x_j)$ since the locale $X$ in general does not coincide with $\Group$. 
The disassociation of the group $\Group$ from the locale 
$X$ is another distinctive feature of our framework.
If we fix a fundamental domain $\La_0$ of
$X$ for the action of $\Group$, then 
Corollary \ref{cor: A} in this case gives the following.

\begin{thmint}\label{thm: A2}
	Assume for simplicity that $c_\phi$ is finite, and
	fix a basis $\xi^{(1)},\ldots,\xi^{(c_\phi)}$ of $\Consv^\phi(S)$.
	Let $\omega =(\omega_{e})\in\prod_{e\in E}C_\loc(\State)$ be a shift-invariant closed uniform form.
	In other words, let $\omega\in\cC$.
	Then there exists $a_{ij}\in\bbR$ for $i=1,\ldots,c_\phi$ and $j=1,\ldots,d$
	and a shift-invariant uniform function $F$ in $C^0_\unif(\State)$
	such that
	\begin{equation}\label{eq: d}
		\omega=\partial\biggl(F
		+\sum_{\indi=1}^{c_\phi}\sum_{\indj=1}^da_{ij}\biggl(\sum_{\grt\in\Group}\grt_\indj\xi^{(\indi)}_{\grt(\La_0)}\biggr)\biggr),
	\end{equation}
	where we let $\xi_W$ be the function in $C^0_\unif(\State)$
	defined as $\xi_W\coloneqq\sum_{x\in W}\xi_x$ for any conserved quantity
	$\xi\in\Consv^\phi(S)$ and $W\subset X$.
\end{thmint}
In Theorem \ref{thm: A2}, we remark that $\partial F\in \cE$, and
by Remark \ref{rem: explicit},
\begin{equation}\label{eq: omega}
	\omega_{\psi}\coloneqq\partial
	\biggl(\sum_{i=1}^{c_\phi}\sum_{\indj=1}^da_{ij}\biggl(\sum_{\grt\in\Group}\grt_\indj\xi^{(\indi)}_{\grt(\La_0)}\biggr)\biggr)\in 
		\cC
\end{equation}
is the image of 
$
	\psi=\bigl(\sum_{i=1}^{c_\phi}a_{i1}\xi^{(\indi)},\ldots,\sum_{i=1}^{c_\phi}a_{id}\xi^{(\indi)}\bigr)
	\in	\bigoplus_{\indj=1}^d\Consv^\phi(S)
$
through the isomorphism \eqref{eq: V0} for the choice of $\La_0$ in Theorem \ref{thm: A2}.
The equality $\omega=\partial F + \omega_\psi$ of \eqref{eq: d} 
is precisely the decomposition given by \eqref{eq: V0}.

In fact, Theorem \ref{thm: A1} is a special case of Theorem \ref{thm: A2}, as follows.
If $\La_0$ is finite, then we may see from the definition that any 
$\omega\in\prod_{e\in E}C_\loc(\State)$ which is closed and shift-invariant 
is uniform.  In addition, again if $\La_0$ is finite, any shift-invariant
uniform function $F\in C^0_\unif(\State)$ 
is of the form $F=\sum_{\tau\in\bbZ^d}\tau(f)$ for some local function $f\in C_\loc(\State)$
satisfying $f(\star)=0$
(see Lemma \ref{lem: OK}). Here, $\tau(f)$ denotes the image of $f$ with 
respect to the action of $\tau\in\bbZ^d$.
For the case $X=\bbZ^d$ with the action of $\Group=\bbZ^d$ given by the standard translation, 
if we let $\La_0=\{(0,\ldots,0)\}$, then we have $\tau_{x}(\La_0)=\{x\}$ for any $x\in\bbZ^d$.
From these observations and the definition of the differential $\partial$, 
we see that Theorem \ref{thm: A1} follows from Theorem \ref{thm: A2}.

In the general setting of Theorem \ref{thm: A},
the choice of a transition rate satisfying certain conditions 
gives an inner product compatible with the norm
on the space of $L^2$-forms.  
The existence of Varadhan's decomposition amounts to the following question.
See \cite{BS21}*{Conjecture 5.5} for a precise formulation of this question as a conjecture.

\begin{question*}
	Assume that the fundamental domain of the action of $\Group$ on the vertices of $X$
	is finite.
	For a suitable definition of closed $L^2$-forms,
	if $\omega$ is a shift-invariant closed $L^2$-form, then
	does there exist $\omega_n\in\cE$ for $n\in\bbN$ and $\psi\in \Hom(G,\Consv^\phi(S))$ such that
	\[
		\lim_{n\rightarrow\infty}(\omega_n + \omega_\psi) = \omega\quad?
	\]
	Here, we let $\omega_\psi$ be the element in $\cC$ corresponding to $\psi$
	in the decomposition \eqref{eq: V0} of Corollary \ref{cor: A}
	given for a choice of a fundamental domain of $X$ for the action of $\Group$.
\end{question*}

The question is answered affirmatively for the cases that Varadhan's decomposition are shown.
Although our local forms construct the core of the $L^2$-space, and local closed forms in our sense are
closed forms in the sense of the $L^2$-space, it is currently not  generally known whether
our local closed forms form \textit{a core of closed forms} in the sense of $L^2$-spaces.
Nevertheless, in subsequent research, we prove Varadhan's decomposition for certain locales
using Theorem \ref{thm: A} of this article, when $S$ is finite and $\mu$ is a product measure,
assuming a certain spectral gap estimate pertaining to the interaction (cf.\ \cite{BS21L2}).
Through this process, we hope to understand the role played by the sharp spectral gap estimates
in the proof of hydrodynamic limits for nongradient systems,
a question which 
has been an important open question for the past thirty years
(see for example \cite{KL99}*{Preface}).

Let $\cC_{L^2}$ and $\cE_{L^2}$ be the shift-invariant closed and exact forms for the $L^2$-space.
The inner product on the $L^2$-space defines an orthogonal decomposition  
\begin{equation}\label{eq: od}
	\cC_{L^2}\cong\cE_{L^2}\oplus \Hom(G,\Consv^\phi(S))
\end{equation}
which is different from \eqref{eq: V0}.
By reinterpreting the method in hydrodynamic limits for obtaining the
macroscopic deterministic partial differential equation from the microscopic system, 
we  have come to understand that the diffusion matrix 
associated with the macroscopic partial differential equation is
given precisely by the matrix relating the two decompositions \eqref{eq: V0} and \eqref{eq: od}.
One critical observation from this fact is that the size of the diffusion matrix of our system should be $c_\phi d$,
the dimension of $\bigoplus_{\indj=1}^d\Consv^\phi(S)$.

Through our investigation, we have come to see the stochastic data consisting of the measure and compatible
inner product as a certain analogy of differential geometric data on Riemannian manifolds
-- the volume form and the metric.  Through this analogy, the orthogonal decomposition
\eqref{eq: od} may be regarded as a differential geometric decomposition given as a certain analogue of the Hodge-Kodaira decomposition in Riemannian geometry, whereas the decomposition
\eqref{eq: V0} is viewed as a more topological decomposition.
In light of this analogy, it would be interesting to
interpret the diffusion matrix relating the topological and measure theoretic structures 
of $\cC_{L^2}$ as an analogy of the period matrix in Hodge theory
comparing the topological and differential geometric structures of the manifold.
Such ideas will be explored in future research.

%
\subsection{Overview}\label{subsec: overview}
%

In this subsection, we give an overview of the proof of Theorem \ref{thm: A}.
The key result for the proof is Theorem \ref{thm: B} below
concerning the property of uniform cohomology.
Consider the system $(X,S,\phi)$.
The uniform cohomology is defined for a configuration space with transition structure as follows.

\begin{definition}\label{def: uc}
	We define the \textit{uniform cohomology} $H^m_\unif(\State)$ for $m\in\bbN$ 
	of the configuration space
	$\State$ with transition structure to be the cohomology of the cochain complex
	\[
		C^0_\unif(\State)\xrightarrow{\partial} Z_\unif^1(\State)
	\]
	which is zero in degrees $m\neq0,1$.
	Concretely, we have $H^0_\unif(\State)\coloneqq\Ker\partial$,
	$H^1_\unif(\State)\coloneqq Z^1_\unif(\State)/\Im\partial$,
	and $H^m_\unif(\State)=\{0\}$ in degrees $m\neq0,1$.
	The uniform cohomology is philosophically the \textit{reduced cohomology}
	in the sense of topology of the pointed space consisting of the configuration space 
	$\State$ and base configuration $\star\in\State$.
\end{definition}

\begin{thmint}[=Theorem \ref{thm: cohomology}]\label{thm: B}
	For the system $(X,S,\phi)$, assume that the interaction $\phi$ is irreducibly quantified.
	If $X$ is transferable, or if the interaction $\phi$ is simple and $X$ is weakly transferable,
	then we have
	\[
		H^m_\unif(\State)\cong
		\begin{cases}
			\Consv^\phi(S)  &  m=0\\
			\{0\}  &  m\neq0.
		\end{cases}
	\]
\end{thmint}
The configuration space $\State$ with transition structure viewed geometrically as a graph
generally has an infinite number of connected components. 
Hence it may be surprising that 
$H^0_\unif(\State)$ is finite dimensional.
This calculation very beautifully reflects the fact
that assuming the conditions of the theorem,
the connected components of the graph $\State_*$ 
can be determined from the values of its conserved quantities (see Remark \ref{rem: intrinsic}).
By the definition of $H^m_\unif(\State)$, Theorem \ref{thm: B} is equivalent
to the existence of a short exact sequence
\begin{equation}\label{eq: SES}
	\xymatrix{
	0\ar[r]&\Consv^\phi(S)\ar[r]^i&C^0_\unif(\State)\ar[r]^<<<<{\partial}&Z^1_\unif(\State)\ar[r]&0.
}
\end{equation}
In other words, \eqref{eq: SES} is a sequence of $\bbR$-linear maps such that
$i$ is injective, $\partial$ is surjective, and $\Im i=\Ker\partial$.
A large portion of our article is dedicated to the construction of \eqref{eq: SES},
especially the proof that the differential $\partial$ is surjective.

First, in \S\ref{subsec: configuration}, we introduce our model and the associated configuration space.  
In \S\ref{subsec: CQ}, we  introduce the notion of a conserved quantity.
Then in \S \ref{subsec: naive}, we define the usual cohomology for the 
configuration space with transition structure and the notion of closed forms.
We then give the relation between the conserved quantities and the 
$H^0$ of the configuration space.

In \S\ref{subsec: uniform}, we introduce the notion of local functions with exact support and prove that any function $f\in C(\State_*)$ may be expanded uniquely as a possibly infinite sum of local functions with  
exact support.  We say that $f\in C(\State_*)$ is \textit{uniform} if the diameter 
of support in the expansion of $f$ is uniformly bounded.
In \S\ref{subsec: horizontal}, we first note that the function 
$\xi_X=\sum_{x\in X}\xi_x$ for any conserved quantity $\xi\in\Consv^\phi(S)$
is uniform.
Since $\xi_X(\star)=0$, the correspondence $\xi\mapsto \xi_X$ gives a natural inclusion
$i\colon\Consv^\phi(S)\hookrightarrow  C^0_\unif(\State)$.
Assume now that the interaction is irreducibly quantified.
We prove in Theorem \ref{thm: 2} that this inclusion gives an isomorphism
\[
	\Consv^\phi(S)\cong\Ker\partial.
\]
It remains to prove that $\partial$ is surjective. 
In Definition \ref{def: unif form}, we define the notion of uniform forms.
For the remainder of \S\ref{sec: local}, 
we assume in addition that $X$ is \textit{strongly transferable}, that is $X\setminus\Ba$ is
an infinite connected graph for any ball $\Ba$ in $X$.
We consider a function 
$f\in C(\State_*)$ such that $\partial f$ is uniform, and we construct in 
Proposition \ref{prop: h_f} of \S\ref{subsec: pairing} 
a \textit{symmetric} pairing $h_f\colon\Val\times\Val\rightarrow\bbR$
on a certain additive submonoid $\Val\subset\bbR^{c_\phi}$ 
(see Definition \ref{def: U})
satisfying the cocycle condition
\[
	h_f(\gra,\grb)+h_f(\gra+\grb,\grc)=h_f(\grb,\grc)+h_f(\gra,\grb+\grc)	
\]
for any $\gra,\grb,\grc\in\Val$.  We then prove in 
Proposition \ref{prop: important} of \S\ref{subsec: criterion}
that if $h_f\equiv 0$, then $f\in C^0_\unif(\State)$.

In \S\ref{sec: weakly}, we consider the case when $X$ is weakly transferable.
This section is technical and can be skipped if the reader is only interested in
the strongly transferable case.
In \S\ref{subsec: pairing wr}, we again construct a pairing $h_f$ for any 
$f\in C(\State_*)$ such that $\partial f$ is uniform,
and prove in Proposition \ref{prop: cocycle} the cocycle condition 
for $h_f$.
We note that in general, the pairing $h_f$ may not be symmetric.
We prove in Proposition \ref{prop: important wr} of \S \ref{subsec: criterion weak}
a weakly transferable version of Proposition \ref{prop: important}.
Finally, we prove in \S\ref{subsec: resilience}
that the pairing $h_f$ is symmetric if the locale $X$ is transferable.

In \S\ref{sec: cohomology}, we complete the proof that $\partial$ is surjective (see Theorem \ref{thm: 1} for details).
The method of proof is as follows.
For any closed uniform form $\omega\in Z^1_\unif(\State)$, since $\omega$
is closed, by Lemma \ref{lem: exact}, there exists $f\in C(\State_*)$ 
such that $\partial f=\omega$.  Note that $f$ may not necessarily be a uniform function.
By the previous argument, there exists a pairing $h_f\colon\Val\times\Val\rightarrow\bbR$
satisfying the cocycle condition.
We prove in Lemma \ref{lem: splitting1} that if the pairing $h_f$ is symmetric, 
or in Lemma \ref{lem: splitting1} if the interaction is simple, 
there exists a function $h\colon\Val\rightarrow\bbR$
such that 
\[
	h_f(\gra,\grb)=h(\gra)+h(\grb)-h(\gra+\grb)
\] 
for any $\gra,\grb\in\Val$.
We modify the function $f$ using $h$
to obtain a function with the same $\partial f$
but satisfies $h_f\equiv0$.  By Proposition \ref{prop: important} or Proposition \ref{prop: important wr},
we see that $f\in C^0_\unif(\State)$.
This proves that the differential $\partial$ in \eqref{eq: SES} is surjective, completing the proof
of Theorem \ref{thm: B}.

Now for the proof of Theorem \ref{thm: A}, suppose that the locale $X$ has a free action of a group $\Group$.
This gives a natural action of $\Group$ on various spaces of functions and forms.
Noting that $G$ acts trivially on $\Consv^\phi(S)$, the boundary
homomorphism of the
long exact sequence for group cohomology \eqref{eq: LES2} associated with the short exact sequence 
\eqref{eq: SES} immediately gives an inclusion
\[
	\xymatrix{
		\cC/\cE\,\,\ar@{^{(}->}[r]^<<<<<\delta&\Hom(G,\Consv^\phi(S)).
	}
\]
For a fixed fundamental domain $\La_0$ of $X$ for the action of $\Group$,
we let
\[
	\omega_\psi\coloneqq\partial\Bigl(\sum_{\grt\in\Group}\psi(\grt)_{\grt(\La_0)} \Bigr)\in\cC
\]
for any $\psi\in\Hom(\Group,\Consv^\phi(S))$, noting that $\psi(\tau)\in\Consv^\phi(S)$ for any $\tau\in\Group$.
The relation of $\omega_\psi$ to the form in \eqref{eq: omega} is explained in Remark \ref{rem: explicit}.
By explicit calculation, we see in Proposition \ref{prop: calculation} that 
$\delta(\omega_\psi)=\psi$, hence $\delta$ is surjective.  The $\bbR$-linear map 
$\omega\mapsto(\omega-\omega_\psi,\psi)$ for $\psi\coloneqq\delta(\omega)$ 
gives a decomposition of $\bbR$-linear spaces
\[
	\cC\cong\cE\oplus\Hom(G,\Consv^\phi(S)),
\]
completing the proof of Theorem \ref{thm: A}.
In Appendix \ref{sec: A}, we review well-known results concerning cohomology of graphs.
In Appendix \ref{sec: B}, we give some examples.
In Appendix \ref{subsec: B1}, we describe the objects appearing in our article 
for the exclusion process.  Finally in
Appendix \ref{subsec: B2}, we let $X=\bbZ$ and
consider the multi-color exclusion process for $S=\{0,1,2\}$.
Then $X$ is \textit{not} transferable and $c_\phi=2$.
We prove that $\partial$ of Definition \ref{def: uc} is \textit{not} surjective in this case.

%
%
%
\section{Configuration Space and Conserved Quantities}\label{sec: configuration}
%
%
%

In this section, we will introduce the configuration space 
of a large scale interacting system and the notion of a \textit{conserved quantity}.
We then define and investigate its \textit{cohomology}.

%
\subsection{Configuration Space and Transition Structure}\label{subsec: configuration}
%

In this subsection, we will give a graph structure which we call the 
transition structure on the configuration space of states on a locale.
We first review some terminology related to graphs.

A directed graph $(X,E)$, or simply a graph, 
is a pair consisting of a set $X$, which we call the set of \textit{vertices}, 
and a subset $E\subset X\times X$, which we call the set of \textit{directed edges},
or simply \textit{edges}.
For any $e\in E = X\times X$, we denote by $o(e)$ and $t(e)$ the first and second components of $e$,
which we call the \textit{origin} and \textit{target} of $e$, so that
$e = (o(e), t(e))\in X\times X$.  For any $e\in E$, we let $\bar e\coloneqq(t(e),o(e))$, which we call the \textit{opposite}
of $e$.  We say that a directed graph $(X,E)$ is \textit{symmetric} if $\bar e \in E$ for any $e\in E$,
and \textit{simple} if $(x,x)\not\in E$ for any $x\in X$.
For any $e=(o(e),t(e))\in E$, we will often use $e$ to denote the set $e=\{o(e),t(e)\}$.
We say that $(X,E)$ is locally finite, if for any $x\in X$, the set $\{ e \in E \mid  x \in e\}$ is finite.
In this article, by abuse of notation (see Remark \ref{rem: abuse}), 
we will often simply denote the graph $(X,E)$ by its set of vertices $X$.

We define a \textit{finite path} on the graph $X$ to be a finite sequence 
$\vec p\coloneqq(e^1,e^2,\ldots,e^N)$ of edges in $E$
such that $t(e^i)=o(e^{i+1})$ for any integer $0<i<N$.
We denote by $\len(\vec p)$ the number of
elements $N$ in $\vec p$, which we call the length of $\vec p$.
We let $o(\vec p)\coloneqq o(e^1)$ and $t(\vec p)\coloneqq t(e^N)$, and
we say that $\vec p$ is a path from $o(\vec p)$ to $t(\vec p)$.
For paths $\vec p_1,\vec p_2$ such that $t(\vec p_1)=o(\vec p_2)$,
we denote by $\vec p_1\vec p_2$ the path from $o(\vec  p_1)$ to  $t(\vec p_2)$ 
obtained as the composition of the two paths.
If a path $\vec p$ satisfies $o(\vec p)=t(\vec p)$, then we say that $\vec p$ is a \textit{closed path}.
For any $x,x'\in X$, we denote by $P(x,x')$ the set of paths from $x$ to $x'$.
We define the graph distance $d_X(x,x')$ between $x$ and $x'$ by 
\[
	d_X(x,x')\coloneqq\inf_{\vec p\in P(x,x')}\len(\vec p)
\]
if $P(x,x')\neq\emptyset$, and $d_X(x,x')\coloneqq\infty$ otherwise.
We say that any subset $Y\subset X$ is \textit{connected}, if $d_Y(x,x')<\infty$ for any $x,x'\in Y$,
where $d_Y$ is the graph distance on the graph $(Y,E_Y)$ given by $E_Y\coloneqq E\cap(Y\times Y)$.

\begin{definition}\label{def: locale}
	We define the \textit{locale} to be a locally finite simple symmetric directed graph $X=(X,E)$
	which is connected.  If the set of vertices of $X$ is an
	infinite set, then we say that $X$ is an \textit{infinite locale}.
\end{definition}

We will use the terminology \textit{locale} to express the discrete object that models the space
where the dynamics under question takes place. We understand the connectedness to be 
an important feature of the locale.

\begin{example}\label{example: locale}
	\begin{enumerate}
	\item The most typical example of an infinite locale is the \textit{Euclidean lattice} 
	$\bbZ^d=(\bbZ^d,\bbE^d)$ for integers $d\geq 1$, where
	$\bbZ^d$ is the $d$-fold product of $\bbZ$, and
	\[
		\bbE^d\coloneqq \bigl\{  (\bsx,\bsy) \in \bbZ^d\times\bbZ^d \,\,\big|\,\,  |\bsx-\bsy|=1\bigr\}.
	\]
	Here, we let $|\bsx-\bsy|\coloneqq\sum_{\indj=1}^d|x_\indj-y_\indj|$ for any 
	$\bsx=(x_1,\ldots,x_d)$, $\bsy=(y_1,\ldots,y_d)$ in $\bbZ^d$.
	\item For integers $d\geq 1$ and $n>0$, a variant of the Euclidean lattice is given by the \textit{Euclidean lattice with nearest $n$-neighbor} 
	$\bbZ^d_n=(\bbZ^d,\bbE^d_n)$, where
	\[
		\bbE^d_n\coloneqq \bigl\{  (\bsx,\bsy) \in \bbZ^d\times\bbZ^d \,\,\big|\,\,  0<|\bsx-\bsy|\leq n\bigr\}.
	\]
	\begin{figure}[ht]
			\centering
			\includegraphics[width=14cm]{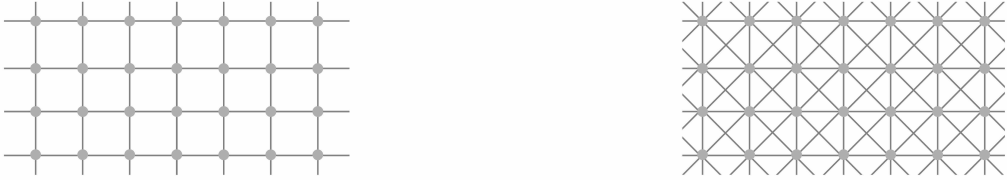}
			\caption{The Euclidean Lattices $\bbZ^2$ and with nearest $2$-neighbor $\bbZ^2_2$}\label{fig: 6}
	\end{figure}
	\item Many types of crystal lattices such as the triangular,
	hexagonal, and diamond lattices are infinite locales 
	(see for example \cite{Sun13}*{Example 3.4, Example 8.3} and \cite{KS01}*{\S5}).
	\item Let $G$ be a finitely generated group and let 
	$\cS\subset G$ be a minimal set of generators.
	Then the associated \textit{Cayley graph} $(G,E_\cS)$ given by 
	$E_\cS\coloneqq\{(\grt,\grt\grs), (\grt,\grt\grs^{-1}) \mid \grt\in G,\grs\in \cS\}$ is a locale.
	If $G$ is infinite, then the associated Cayley graph is an infinite locale.
		\item Let $X_1=(X_1,E_1)$ and $X_2=(X_2,E_2)$ be locales.
		If $X\coloneqq X_1 \times X_2$ and
		\begin{align*}
			E \coloneqq \{ & ((o(e_1), o(e_2)), (t(e_1), o(e_2)) \mid e_1\in E_1, e_2\in E_2\}\\& \cup
			\{  ((o(e_1), o(e_2)), (o(e_1), t(e_2)) \mid e_1\in E_1, e_2\in E_2\} \subset (X_1\times X_2)\times(X_1\times X_2),
		\end{align*}
		then $(X,E)$ is a locale, which we denote 
		$X_1\times X_2$.  We say that $X$ is a \textit{product} of $X_1$ and $X_2$.
		Note that $(\bbZ^d,\bbE^d)$ coincides with the 
		$d$-fold product $(\bbZ,\bbE) \times\cdots\times(\bbZ,\bbE)$.
		\item Suppose $X=(X,E)$ is a locale, and let $Y\subset  X$ be a connected subset.
		If we let $E_Y\coloneqq E\cap (Y\times Y)\subset X\times X$, then $Y=(Y,E_Y)$ gives a graph
		which is a locale. 
		We call $Y$ a \textit{sublocale} of $X$.
		\end{enumerate}
\end{example}

Next, we introduce the set of states, which is a nonempty set $S$ expressing
the possible states the model may take at each vertex, and the 
configuration space $\State$ for $S$ on $X$.

\begin{definition}
	We define the set of \textit{states} to be a nonempty set $S$.
	We call any element of $S$ a \textit{state}.
	We will designate an element $*\in S$ which we call the \textit{base state}.
	If $S\subset\bbR$ and $0\in S$, then we will often take the base state $*$ to be $0$.
	For any locale $X=(X,E)$, we define the \textit{configuration space} for $S$ on $X$ by
	\[
		\State\coloneqq \prod_{x\in X}S.
	\]	
	We call an element $\state=(\eta_x)$ of $\State$ a \textit{configuration}.
	We denote by $\star$ the \textit{base configuration},
	which is
	the configuration in $\State$ whose components are all at base state.
\end{definition}
	 
Now, we introduce the symmetric binary interaction, which 
expresses the interaction between states on adjoining vertices.

\begin{definition}\label{def: interaction}
	A \textit{symmetric binary interaction}, which we simply call an \textit{interaction} on $S$, is a map 
	$ 
		\phi\colon S\times S\rightarrow S\times S
	$ 
	such that 
	\begin{equation}\label{eq: reverse}
		\ic\circ\phi\circ\ic\circ\phi(s_1,s_2)=(s_1,s_2)
	\end{equation}
	for any $(s_1,s_2)\in S\times S$ satisfying $\phi(s_1,s_2)\neq(s_1,s_2)$, 
	where $\ic\colon S\times S\rightarrow S\times S$
	is the bijection obtained by exchanging the components of $S\times S$.   
\end{definition}

Examples of the set of states $S$ and interactions $\phi$ are given in 
Example \ref{example: interactions2} of \S\ref{subsec: CQ}.
Throughout this article, a system $(X,S,\phi)$ indicates that $X$ is a locale, $S$ is a set of states, 
and $\phi$ is an interaction.  We will next define the configuration space with transition structure
associated with such a system.  We first prepare a lemma.

\begin{lemma}\label{lem: SDG}
	For a locale $X=(X,E)$ and the set of states $S$, 
	let $\State$ be the configuration space for $S$ on $X$.
	Let $\phi\colon S\times S\rightarrow S\times S$ be an interaction on $S$.
	For any $e=(o(e),t(e))\in E\subset X\times X$, 
	we define the map $\phi_e\colon \State\rightarrow\State$ by
	$\phi_e(\state)\coloneqq\state^e$, where $\state^e=(\eta^e_x)\in\State$
	is defined as in \eqref{eq: se}.
	In other words, $\phi_e(\state)$ is obtained by applying $\phi$ to the $o(e)$ and 
	$t(e)$ components of $\state$. 
	If we denote by $\Phi$ the image of the map
	\begin{equation}\label{eq: transition}
		E\times\State\rightarrow\State\times\State, \qquad (e,\state) \mapsto (\state,\phi_e(\state)),
	\end{equation}
	then the pair $(\State, \Phi)$ is a symmetric directed graph.
\end{lemma}

\begin{proof}
	It is sufficient to prove that the directed graph $(\State, \Phi)$ is symmetric.
	By the definition of an interaction, we have $\ic\circ\phi\circ\ic\circ\phi(s_1,s_2)=(s_1,s_2)$
	for any $(s_1,s_2)\in S\times S$ such that $\phi(s_1,s_2)\neq(s_1,s_2)$.
	This shows that for any $\state\in\State$, if $\phi_e(\state)\neq\state$,
	then we have $\phi_{\bar e}\circ\phi_e(\state)=\state$. 
	Consider the element $(\state, \phi_e(\state)) \in\Phi$.
	If $\phi_e(\state)=\state$, then $(\phi_e(\state), \state)=(\state,\state)\in\Phi$.
	If  $\phi_e(\state)\neq\state$,
	then $(\phi_e(\state), \state) = ( \phi_e(\state), \phi_{\bar e} \circ \phi_e(\state))$, 
	which is an element in $\Phi$
	since it is the image of $(\bar e, \phi_e(\state))\in E\times\State$ by the map \eqref{eq: transition}.
\end{proof}

\begin{remark}\label{rem: sufficient}
	In fact, the condition \eqref{eq: reverse} that we impose on the interaction is simply
	a \textit{sufficient condition}
	and \textit{not a necessary condition} for our theory.  The property that we actually use 
	is that $(\State, \Phi)$ is a symmetric directed graph.
\end{remark}

\begin{definition}
	For a locale $X$ and a set of states $S$, if we fix an interaction 
	$\phi\colon S\times S\rightarrow S\times S$, then Lemma \ref{lem: SDG} implies that
	$\Phi$ gives a structure of a symmetric directed graph 
	\[
		\State=(\State, \Phi)
	\] 
	on the configuration space $\State$.
	We call this structure the \textit{transition structure}, and
	we call any element $\varphi \in \Phi$ a \textit{transition}.
	In particular, we say that $\varphi = (\state, \phi_e(\state))\in\Phi$ is a transition of $\state$ by $e$.
	Following the convention in literature, 
	we will often denote $\phi_e(\state)$ by $\state^e$.
\end{definition}


\begin{remark}\label{rem: abuse}
	Our convention of denoting the locale $(X,E)$ by $X$ and the configuration space
	with transition structure $(\State,\Phi)$ by $\State$
	follows similar convention as that of topological spaces, where the topological
	space and its underlying set is denoted by the same symbol.
	We are interpreting the set of edges of a graph as giving a geometric structure to
	the set of vertices.
\end{remark}

For the configuration space $\State$ with transition structure,
the edges $\Phi$ expresses all the possible 
transitions on the configuration space with respect to the interaction $\phi$.
For any element $\state=(\eta_x)\in S^X$, we define the support of $\state$ to be the set 
$\Supp(\state)\coloneqq\{ x\in X\mid \eta_x\neq*\}\subset X$.  
The subset 
\[
	\State_*\coloneqq\big\{\state=(\eta_x)\in S^X\mid\,|\Supp(\state)|<\infty\big\}
	\subset S^X
\]
of the configuration space will play an important role in our theory.
If we let $\Phi_*\coloneqq \Phi\cap(\State_*\times\State_*)$, then $\State_*=(\State_*,\Phi_*)$ is 
again a symmetric directed graph, 
which we refer to again as a configuration space with transition structure.

For any set $A$, we let $C(A)\coloneqq\Map(A,\bbR)$ be the $\bbR$-linear space of functions from $A$ to $\bbR$.
As in \S\ref{subsec: model}, for any finite $\La\subset X$,
we let $S^\La\coloneqq\prod_{x\in\La}S$.
Following standard convention, we let $S^\emptyset=\{\star\}$ if $\La=\emptyset$.
The natural projection $\State_*\subset\State\rightarrow S^\La$
given by mapping $\state=(\eta_x)_{x\in X}\in\State$ to $\state|_{\La}\coloneqq (\eta_x)_{x\in\La}$
induces a natural injection
$C(S^\La)\hookrightarrow \CS$.  
From now on, we will identify $C(S^\La)$ with its image in $\CS$.  
Note that any $f\in \CS$ is a function in $C(S^\La)$
if and only if $f$ as a function for $\state\in\State_*$ depends 
only on $\state|_{\La}\in S^\La$.
We call any such function a \textit{local function}.
We denote by $C_\loc(\State)$ the space of local functions, 
which is a subspace of both $C(\State)$ and $C(\State_*)$.
If the set of vertices of $X$ is finite, then we simply have $C_\loc(\State)=C(\State)=\CS$.
Our methods are of interest predominantly for the case when $X$ is an infinite locale.


%
\subsection{Conserved Quantities and Irreducibly Quantified Interactions}\label{subsec: CQ}
%

In this subsection, we introduce the notion of a \textit{conserved quantity}, which
are certain invariants of states preserved by the interaction.
Using the conserved quantities,  we will 
introduce the important notion for an interaction to be \textit{irreducibly quantified}.
In what follows, let $S$ be a set of states with base state $*\in S$,
and we fix an interaction $\phi\colon S\times S\rightarrow S\times S$ on $S$.

\begin{definition}	
	A \textit{conserved quantity} for the interaction
	$\phi$ is a function $\xi\colon S\rightarrow \bbR$ satisfying $\xi(*)=0$ and
	\begin{equation}\label{eq: conserve2}
		\xi(s_1) + \xi(s_2) = \xi(s_1') + \xi(s_2')
	\end{equation}
	for any $(s_1,s_2)\in S\times S$ and $(s_1',s_2')\coloneqq\phi(s_1,s_2)$.
	We denote by $\Consv^\phi(S)$ the $\bbR$-linear subspace of $\Map(S,\bbR)$ 
	consisting of the conserved quantities for the interaction $\phi$.
	We let $c_\phi\coloneqq\dim_\bbR\Consv^\phi(S)$.
\end{definition}

Examples of interactions and corresponding conserved quantities are given as follows.

\begin{example}\label{example: interactions2}
	\begin{enumerate}
		\item
		Let $S=\{0,1,\ldots,\kappa\}$ with base state $*=0$
		for some integer $\kappa>0$.
		For the \textit{multi-species exclusion process} of
		Example \ref{example: interactions} (2), 
		we have $c_\phi=\kappa$, and
		the conserved
		quantities given by
		$\xi^{(\indi)}(s)=1$ if $s=\indi$ and $\xi^{(\indi)}(s)=0$ otherwise
		for $\indi=1,\ldots,\kappa$
		give a basis of the $\bbR$-linear space $\Consv^\phi(S)$.
		\item Let $S=\bbN$
		or $S=\{0,\ldots,\kappa\}\subset\bbN$ for some integer $\kappa>0$.
		We let $*=0$ be the base state. 
		The map $\phi\colon S\times S\rightarrow S\times S$ defined by
		\[   
			\phi(s_1,s_2)\coloneqq
			\begin{cases}
				  (s_1-1, s_2+1)  &  s_1-1, s_2+1\in S\\
				 (s_1,s_2)  &  \text{otherwise}
			\end{cases}
		\]
		is an interaction.
		The stochastic process induced from this interaction when $S$ is finite
		is the \textit{generalized exclusion process}.
		Note that for any $n>1$ in $S$, 
		we have $\phi(n,0)=(n-1,1)$, hence for any conserved quantity $\xi$,
		equation \eqref{eq: conserve} inductively gives
		\[
			\xi(n) = \xi(n) + \xi(0) = \xi(n-1) + \xi(1) = \xi(n-2) + 2\xi(1) = \cdots = n\xi(1).
		\]
		This shows that $c_\phi=1$, and
		the conserved quantity $\xi\colon S\rightarrow\bbN$ given by $\xi(s)=s$
		gives a basis of the one-dimensional $\bbR$-linear space $\Consv^\phi(S)$.
		This interaction is simple in the sense of Definition \ref{def: conditions} (2).
		\begin{figure}[ht]
			\centering
			\includegraphics[width=13cm]{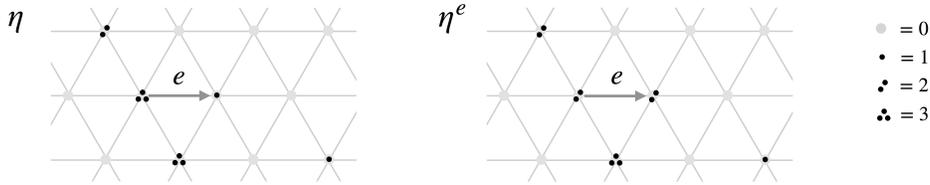}
			\caption{Generalized Exclusion Process}
		\end{figure}
		\item Let $S=\bbN$, or let 
		$S=\{0,\ldots,\kappa\}\subset\bbN$ for some integer $\kappa>1$,
		with base state $*=0$. 
		The map $\phi\colon S\times S\rightarrow S\times S$ defined by
		\[   
			\phi(s_1,s_2)\coloneqq
			\begin{cases}
				  (s_2, s_1)  & s_1>0, s_2=0\\
				  (s_1-1, s_2+1)  & s_1>1, s_2>0, s_2+1\in S\\
				 (s_1,s_2)  &  \text{otherwise}
			\end{cases}
		\]
		is an interaction.  We have $c_\phi=2$, and the functions $\xi^{(1)}(s)=s$ and 
		\[
		\xi^{(2)}(s)=
			\begin{cases}
				1  &  s>0\\
				0  &  s=0
			\end{cases}
		\]
		gives a basis of the $\bbR$-linear space $\Consv^\phi(S)$.
		This interaction is the \textit{lattice gas with energy}.
		Intuitively, $S$ represents the amount of energy on the vertex,
		with $s=0$ representing the fact that there are no particles.
		\item For $S=\{-1,0,1\}$ with base state $*=0$, the map 
		$\phi\colon S\times S\rightarrow S\times S$ defined by
		\begin{align*}
			\phi((0,0)) &= (-1,1), &\phi((-1,1)) &= (1,-1), &\phi((1,-1)) &= (0,0),&
		\end{align*}
		and $\phi((s_1,s_2))=(s_2,s_1)$ for $s_1,s_2\in S$ such that $s_1+s_2\neq0$
		is an interaction.
		We have $c_\phi=1$, and
		the conserved quantity $\xi\colon S\rightarrow\bbZ$ given by $\xi(s)=s$
		gives a basis of the $\bbR$-linear space $\Consv^\phi(S)$.
		This interaction is simple in the sense of Definition \ref{def: conditions} (2).
		\item
		For $S=\{0,1\}$ with base state $*=0$, the map 
		$\phi\colon S\times S\rightarrow S\times S$ defined by
		\begin{align*}
			\phi((0,s)) &= (1,s), &\phi((1,s)) &= (0,s)
		\end{align*}
		for any $s\in S$ does \textit{not} satisfy \eqref{eq: reverse}.  However, $(S^X,\Phi)$
		is a symmetric directed graph for any locale $X$ and our theory applies also to this case
		(see also Remark \ref{rem: sufficient}).
		This case is known as the \textit{Glauber model}.
		Since $\phi(0,0)=(1,0)$, equation \eqref{eq: conserve2} gives 
		$\xi(1)=\xi(0)=0$ for any conserved quantity
		$\xi$, hence we have $c_\phi=\dim_\bbR\Consv^\phi(S)=0$.
	\end{enumerate}
	The exclusion process of Example \ref{example: interactions} (1) is a special case of both
	the multi-species exclusion process and the generalized exclusion process, with the set
	of states given by $S=\{0,1\}$.
\end{example}

Let $X$ be a locale, and let $\State_*$ be the configuration space with transition structure
associated with our interaction.
Note that if $\xi$ is a conserved quantity, then $\xi$ defines a function
$\xi_X\colon\State_*\rightarrow\bbR$ given by
\begin{equation}\label{eq: sum}
	\xi_X(\state)\coloneqq\sum_{x\in X}\xi(\eta_x)
\end{equation}
for any $\state=(\eta_x)\in\State_*$.  The sum is well-defined since by definition,
$\eta_x=*$ outside a finite number of $x\in X$.

\begin{remark}\label{rem: interactions2}
	The interactions in Example \ref{example: interactions2} have the following interpretations.
	See Remark \ref{rem: interactions} for the case of the multi-species exclusion process.
	\begin{enumerate}
	\item[(2)]
	For $s\in S$ in the process in Example  \ref{example: interactions2} (2), 
	$s=0$ describes the state where there are no particles
	at the vertex, and $s=k$ for an integer $k>0$ the state where there are $k$ 
	indistinguishable particles at the vertex.
	Then for the conserved quantity given by $\xi(s)=s$ for any $s\in S$,
	the value
	$\xi_X(\state)$ for a configuration $\state\in\State_*$ expresses the 
	total number of particles in the configuration.
	If $S=\{0,\cdots,\kappa\}$ for some integer $\kappa>0$, then this process is called the generalized
	\textit{exclusion} process, since at most $\kappa$ particles are allowed to occupy the vertex.
	If $S=\bbN$, then the process includes the trivial case where the particles evolve as independent 
	random walks with no interactions between the particles, as described in \cite{KL99}*{Chapter 1}.
	\item[(3)]
	For $s\in S$ the process in Example  \ref{example: interactions2} (3), 
	$s=0$ describes the state where there are no particles
	at the vertex, and $s=k$ for an integer $k>0$ the state where there is a particle with energy $k$ at the vertex.
	Then $\xi_X^{(1)}(\state)$ for a configuration $\state\in\State_*$ expresses the 
	total energy of the particles in the configuration,
	and $\xi_X^{(2)}(\state)$ expresses the 
	total number of particles in the configuration. 
	\item[(4)]
	For the process in Example \ref{example: interactions2} (4), 
	$s$ describes the spin of the particle at the vertex.
	Then for the conserved quantity given by $\xi(s)=s$ for any $s\in S$,
	the value
	$\xi_X(\state)$ for a configuration $\state\in\State_*$ expresses the 
	total spin of the configuration.
	\item[(5)]
	For $s\in S$ in the Glauber model in Example \ref{example: interactions2} (5), 
	one interpretation is that
	$s=0$ describes the state where there are no particles at the vertex, 
	and $s=1$ describes the state where there is a single particle at the vertex.
	Since the interaction allows for the creation and annihilation of particles,
	there are no nontrivial conserved quantities.  Another interpretation is that
	$s$ describes the spin of the particle at each vertex.	
	\end{enumerate}	
\end{remark}

\begin{remark}\label{rem: hydrodynamic limit2}
	The hydrodynamic limits 
	for 
	the generalized exclusion process has been studied by Kipnis-Landim-Olla \cite{KLO94}.
	See also \cite{KL99}*{Chapter 7}.
	The case of lattice gas with energy has been studied by Nagahata \cite{N03},
	and the stochastic process induced from the interaction in Example \ref{example: interactions2} (4)
	was studied by Sasada \cite{Sas10}.
	In all of the known cases, the underlying locale is taken to be the Euclidean lattice.
\end{remark}

The important notion of irreducibly quantified is defined as follows.

\begin{definition}\label{def: irreducibly quantified}
	We say that an interaction $\phi\colon S\times S\rightarrow S\times S$
	is \textit{irreducibly quantified}, if for any finite locale
	$(X,E)$ and configurations $\state,\state'\in\State$ satisfying
	\[
		\xi_X(\state) = \xi_X(\state')
	\]
	for all conserved quantities $\xi\in\Consv^\phi(S)$, 
	there exists a finite path $\vec\gamma$ from $\state$ to $\state'$ in $\State$.
\end{definition}

\begin{remark}\label{rem: general}
	If we consider $S=\{0,1\}$ with base state $*=0$ and the interaction
	$\phi$ given by
	\begin{align*}
		\phi((0,0))&=(1,1), &
		\phi((1,0))&=(0,1), &
		\phi((0,1))&=(1,0), &
		\phi((1,1))&=(0,0),
	\end{align*}
	then a function $\xi\colon S\rightarrow\bbR$ satisfying 
	$\xi(0)=0$ and \eqref{eq: conserve2} would imply that $\xi(1)=0$,
	hence $\xi\equiv 0$.  
	In fact, this interaction is not irreducibly quantified.
	In order to deal with such a model, it may be necessary
	to consider conserved quantities with values in $\bbZ/2\bbZ$.
\end{remark}

The remainder of this section is devoted to the proof of Proposition \ref{prop: irreducibly quantified},
which asserts that the interactions given in Example \ref{example: interactions2} are all irreducibly quantified.
For the proof, we first introduce the notion of exchangeability.

\begin{definition}\label{def: exchangeable}
	We say that an interaction $\phi\colon S\times S\rightarrow S\times S$
	is \textit{exchangeable},  if for any $(s_1,s_2)\in S\times S$, 
	a suitable composition of maps $\phi$ and $\bar\phi$
	such that $(s_1,s_2)$ maps to $(s_2,s_1)\in S\times S$.
	Here, if we write $\phi(s_1,s_2)=(\phi_1(s_1,s_2),\phi_2(s_1,s_2))\in S\times S$, then we let
	 $\bar\phi(s_1,s_2)\coloneqq (\phi_2(s_2,s_1), \phi_1(s_2,s_1))$
	 for any $(s_1,s_2)\in S\times S$.
	 In other words, $\bar\phi\coloneqq\ic\circ\phi\circ\ic$, where $\ic\colon S\times S\rightarrow S\times S$
	is the bijection obtained by exchanging the components of $S\times S$.
\end{definition}

It is straightforward to check that the interactions in Example \ref{example: interactions2} are all
exchangeable.    Since $\bar\phi\circ\phi(s_1,s_2)=(s_1,s_2)$ if $\phi(s_1,s_2)\neq(s_1,s_2)$,
the composition of Definition \ref{def: exchangeable} is necessarily of the form $\phi^n$ or $\bar\phi^n$ for some integer $n\geq 0$.
In what follows, consider a system $(X,S,\phi)$ and the associated configuration space $\State_*$
with transition structure.  We first prepare some notations.

\begin{definition}
	Let $\state$ be a configuration in $\State_*$, and let $x,y$ be vertices in $X$.
	\begin{enumerate}
		\item
		We define $\state^{x,y}\in\State_*$
to be the configuration whose component outside $x,y\in X$ coincides
with that of $\state$, and the $x$ and $y$ components are given by $\eta^{x,y}_x\coloneqq \eta_y$ and $\eta^{x,y}_y\coloneqq \eta_x$.
		\item
		We define $\state^{x\arr y}$ to be the configuration whose component outside $x,y\in X$ coincides
	with that of $\state$, and the $x$ and $y$ components are given by 
	$\eta^{x\arr y}_x\coloneqq \phi_1(\eta_x,\eta_y)$ and 
	$\eta^{x\arr y}_y\coloneqq\phi_2(\eta_x,\eta_y)$,
	where $\phi(\eta_x,\eta_y)=(\phi_1(\eta_x,\eta_y),\phi_2(\eta_x,\eta_y))\in S\times S$ for the interaction $\phi$ on $S$.
	\end{enumerate}
\end{definition}

A path in $\State_*$ is a sequence of transitions, and a transition is induced by an interaction
on some edge of the locale.
Hence for a sequence of edges $\bse=(e^1,\ldots,e^N)$ in $E$
and $\state\in\State_*$, if we let $\state^0\coloneqq\state$ and 
$\state^i\coloneqq(\state^{i-1})^{e^i}$ for $i=1,\ldots,N$,
then $\varphi^i=(\state^i,\state^{i+1})$
is a transition and $\vec\gamma^\bse_\state\coloneqq(\varphi^1,\ldots,\varphi^N)$ gives a path from 
$\state$ to $\state^\bse\coloneqq\state^N$ in $\State_*$.  We call this path $\vec\gamma_\state^\bse$  
the \textit{path with origin $\state$ induced by $\bse$},
or a path obtained by applying the edges $\bse$ to $\state$.

We prove some existences of paths between certain configurations, 
when the interaction is exchangeable.

\begin{lemma}\label{lem: first}
	If an interaction $\phi\colon S\times S\rightarrow S\times S$
	is \textit{exchangeable}, then for any configuration $\state\in\State_*$ and
	vertices $x,y\in X$,
	there exists a path $\vec\gamma$ from $\state$ to $\state^{x,y}$ in $\State_*$.
\end{lemma}

\begin{proof}
	Since $X$ is connected, there exists a path $\vec p=(e^1,\ldots,e^N)$ from $x$ to $y$.
	Note that by definition, $x=o(e^1)$ and $y=t(e^N)$.
	Since $\phi$ is exchangeable, applying $e^1$ or $\bar e^1$ sufficiently many times to $\state$,
	we obtain a path from $\state$ to $\state^{o(e^1),t(e^1)}$.
	Repeating this process for $e^2,\ldots,e^{N}$,
	we obtain a path $\vec\gamma_1$ from  $\state$ to $\state'$, where $\state'$ is such that
	the components of $\state'$ coincides with that of $\state$ outside the vertices appearing in 
	the edges $e^1,\ldots,e^{N}$, $\eta'_{o(e^i)}=\eta_{t(e^i)}$ for $i=1,\ldots, N-1$, and $\eta'_{y}=\eta_x$.
	Then, reversing the above process,
	first by applying $e^{N-1}$ or $\bar e^{N-1}$ sufficiently many times to $\state'$,
	we obtain a path from $\state'$ to $(\state')^{o(e^{N-1}),t(e^{N-1})}$.
	Repeating this process for $e^{N-2},\ldots,e^{1}$,
	we see that we obtain a path $\vec\gamma_2$ from $\state'$ to $\state^{x.y}$.
	Then the composition $\vec\gamma\coloneqq\vec\gamma_1\vec\gamma_2$
	gives a path from $\state$ to $\state^{x,y}$ in $\State_*$ as desired.
\end{proof}

\begin{lemma}\label{lem: second}
	If an interaction $\phi\colon S\times S\rightarrow S\times S$
	is \textit{exchangeable}, then for any configuration $\state\in\State_*$ and
	vertices $x,y\in X$,
	there exists a path $\vec\gamma$ from $\state$ to $\state^{x\arr y}$ in $\State_*$.
\end{lemma}

\begin{proof}
	Again, since $X$ is connected, there exists a path $\vec p=(e^1,\ldots,e^N)$ from $x=o(e^1)$ to $y=t(e^N)$.
	If we let $x'=o(e^N)$, then $x'$ is a vertex connected to $y$ by the edge $e^N$.
	Then by construction, if we let $\state'\coloneqq(\state^{x,x'})^{e_N}$, then $\state'$ is a configuration
	 whose component outside $x,x',y\in X$ coincides
with that of $\state$, and the $x, x'$ and $y$ components are given by $\eta'_x=\eta_{x'}$, 
$\eta'_{x'}=\phi_1(\eta_x,\eta_y)$, and $\eta'_y = \phi_2(\eta_x,\eta_y)$.
Thus we have $(\state')^{x,x'}=\state^{x\arr y}$.
By Lemma \ref{lem: first}, there exists a path $\vec\gamma_1$ from $\state$ to $\state^{x,x'}$
and a path $\vec\gamma_2$ from $\state'$ to  $\state^{x\arr y}$ in $\State_*$.
If we denote by $\vec\varphi$ the path given by a single transition $\varphi=(\state^{x,x'},\state')$,
then the composition $\vec\gamma\coloneqq\vec\gamma_1\vec\varphi\vec\gamma_2$ gives a path from
$\state$ to $\state^{x\arr y}$ in $\State_*$ as desired.
\end{proof}

Using the above results, we may prove that various interactions are irreducibly quantified.

\begin{proposition}\label{prop: irreducibly quantified}
	The interactions of Example \ref{example: interactions2} are all irreducibly quantified.
\end{proposition}

\begin{proof}
	By Example \ref{example: interactions2},
	we see that $\Consv^\phi(S)$ are finite dimensional.
	Let $X$ be a locale and let $\state,\state'\in \State_*$
	such that $\xi_X(\state)=\xi_X(\state')$ for any conserved quantity $\xi\in \Consv^\phi(S)$.
	We prove the existence of a path $\vec\gamma$ from $\state$ to $\state'$ 
	by induction on the cardinality of the set 
	$\Delta(\state,\state')\coloneqq\{x\in X\mid \eta_x\neq \eta'_x\}$.
	Note that $\Delta(\state,\state')$ is finite since the supports of $\state$ and $\state'$ are finite.
	If $|\Delta(\state,\state')|=0$, then $\state=\state'$ and there is nothing to prove.
	Suppose $|\Delta(\state,\state')|>0$ and that the assertion is proved for $\state'',\state'$,
	where $\state''$ is any configuration in $\State_*$
	such that $|\Delta(\state'',\state')|<|\Delta(\state,\state')|$
	and $\xi_X(\state'')=\xi_X(\state')$ for any conserved quantity $\xi\in \Consv^\phi(S)$.
	\begin{enumerate}
	\item Consider the case of the multi-species exclusion process of Example \ref{example: interactions2} (1)
	with conserved quantity $\xi^{(1)},\ldots,\xi^{(c_\phi)}$.
	Let $x\in\Delta(\state,\state')$.
	If we let $\eta_x=\indi\in S$, then since 
	$\xi^{(\indi)}_X(\state)=\xi^{(\indi)}_X(\state')$, there exists
	$y\in \Delta(\state,\state')$ such that $\eta_y'=\indi$.  
	By Lemma \ref{lem: first},
	there exists a path $\vec\gamma_1$ from $\state$ to $\state^{x,y}$.
	Note that $\state^{x,y}$ coincides with $\state$ outside $x,y$, and we have 
	$\state^{x,y}_y=\eta_x=\indi=\eta_y'$, which implies that 
	$\xi_X(\state)=\xi_X(\state^{x,y})=\xi_X(\state')$ for any conserved quantity and
	$|\Delta(\state^{x,y},\state')|<|\Delta(\state,\state')|$.
	Hence by the induction hypothesis, there exists a path $\vec\gamma_2$ 
	from $\state^{x,y}$ to $\state'$ in $\State_*$.  Our assertion is proved by taking
	$\vec\gamma\coloneqq\vec\gamma_1\vec\gamma_2$.
	\item
	Consider the case of the generalized exclusion process of Example \ref{example: interactions2} (2)
	with conserved quantity $\xi$ given by $\xi(s)=s$.
	Again let $x\in\Delta(\state,\state')$ such that $\eta_x>\eta'_x$.
	Since  $\xi_X(\state)=\xi_X(\state')$, there exists $y\in\Delta(\state,\state')$
	such  that $\eta_y<\eta'_y$.
	Let $M\coloneqq\min(\eta_x-\eta'_x, \eta'_y-\eta_y)$,
	and let $\state^0=\state$ and $\state^i=(\state^{i-1})^{x\arr y}$ for $i=1,\ldots,M$.
	By Lemma \ref{lem: second}, there exists a path $\vec\gamma_i$
	from $\state^{i-1}$ to $\state^{i}$ in $\State_*$ for $1<i<M$.
	We have $\xi_X(\state^M)=\xi_X(\state)=\xi_X(\state')$, and since
	$\eta^M_x=\eta_x-M$ and $\eta^M_y=\eta_y+M$, we have either 
	$\eta^M_x=\eta'_x$ or $\eta^M_y=\eta'_y$.
	This shows that $|\Delta(\state^{M},\state')|<|\Delta(\state,\state')|$.
	Hence by the induction hypothesis, there exists a path $\vec\gamma'$
	from $\state^{M}$ to $\state'$ in $\State_*$.
	Our assertion is proved by taking
	$\vec\gamma\coloneqq\vec\gamma_1\cdots\vec\gamma_M\vec\gamma'$.
	\item
	Consider the case of the lattice gas with energy of Example \ref{example: interactions2} (3)
	with conserved quantity $\xi^{(1)}$ and $\xi^{(2)}$.
	Suppose there exists $x\in\Delta(\state,\state')$ such that
	$\eta'_x=0$. Then, since $\xi_X^{(2)}(\state)=\xi_X^{(2)}(\state')$,
	there exists $y\in\Delta(\state,\state')$ such that 
	$\eta_y=0$.  Then $\state^{x,y}$ satisfies $\xi_X(\state^{x,y})=\xi_X(\state')$
	for any conserved quantity $\xi$, and
	$|\Delta(\state^{x,y},\state')|<|\Delta(\state,\state')|$.
	Hence by the induction hypothesis, there exists a path $\vec\gamma_2$ 
	from $\state^{x,y}$ to $\state'$ in $\State_*$.
	If we let $\vec\gamma_1$ be the path from $\state$ to $\state^{x,y}$ given in Lemma
	\ref{lem: first}, then $\vec\gamma\coloneqq\vec\gamma_1\vec\gamma_2$
	satisfies the desired property.
	Otherwise, we have $\eta_x\neq0$ and $\eta'_x\neq0$ for any $x\in\Delta(\state,\state')$.
	Since $\xi_X^{(1)}(\state)=\xi_X^{(1)}(\state')$, there exists $x,y\in\Delta(\state,\state')$
	such that $\eta_x>\eta'_x$ and $\eta_y<\eta'_y$.  As in (2), let
	$M\coloneqq\min(\eta_x-\eta'_x, \eta'_y-\eta_y)$,
	and let $\state^0=\state$ and $\state^i=(\state^{i-1})^{x\arr y}$ for $i=1,\ldots,M$.
	By Lemma \ref{lem: second}, there exists a path $\vec\gamma_i$
	from $\state^{i-1}$ to $\state^{i}$ in $\State_*$ for $1<i<M$.
	We have $\xi_X(\state^M)=\xi_X(\state)=\xi_X(\state')$
	and either $s^M_x=\eta'_x$ or $s^M_y=\eta'_y$, which
	shows that $|\Delta(\state^{M},\state')|<|\Delta(\state,\state')|$.
	Hence by the induction hypothesis, there exists a path $\vec\gamma'$ 
	from $\state^{M}$ to $\state'$ in $\State_*$.
	Our assertion is proved by taking
	$\vec\gamma\coloneqq\vec\gamma_1\cdots\vec\gamma_M\vec\gamma'$.	
	\item The case of Example \ref{example: interactions2} (4) is proved in a similar fashion as that of (1),
	but by first using the interaction $\phi((0,0))=(1,-1)$ and  $\phi((-1,1))=(0,0)$ to equalize the number of 
	vertices whose states are at $+1$ and $-1$.
	\item Consider the Glauber model of Example \ref{example: interactions2} (5).
	Since in this case, the only conserved quantity is the zero map, the condition for 
	the conserved quantity is always satisfied.
	Let $x\in\Delta(\state,\state')$.  If we let $e\in E$ be any edge such that $o(e)=x$,
	then $\state^e$ coincides with $\state$ outside $x$ and we have $\eta^e_x=\eta'_x$.
	This shows we have $|\Delta(\state^e,\state')|<|\Delta(\state,\state')|$.
	Hence by the induction hypothesis, there exists a path $\vec\gamma'$ 
	from $\state^e$ to $\state'$ in $\State_*$.
	Our assertion is proved by taking
	$\vec\gamma\coloneqq\vec\varphi\vec\gamma'$, where $\vec\varphi$
	is the path given by the transition $\varphi=(\state,\state^e)$.\qedhere
\end{enumerate}
\end{proof}

%
\subsection{The Cohomology of the Configuration Space}\label{subsec: naive}
%

In this subsection, we consider the cohomology as graphs of a configuration space with transition structure.
See Appendix \ref{sec: A} for generalities concerning the homology and cohomology of graphs.
Let $(X,E)$ be a locale, and 
we let $\State_*=(\State_*,\Phi_*)$ be the configuration space 
with transition structure associated to a system $(X,S,\phi)$.
For any $\varphi=(o(\varphi),t(\varphi))\in \Phi_*$,
we let $\bar\varphi\coloneqq(t(\varphi),o(\varphi))\in \Phi_*$, 
which we call the opposite transition.

The cohomology defined in 
Definition \ref{def: coh graph} of a configuration space with transition structure
is given as follows.

\begin{definition}
	Let $(X,S,\phi)$ be a system.
	For the graph $(\State_*,\Phi_*)$, let
	\begin{align}\label{eq: cohomology}
		\CS&\coloneqq\Map(\State_*, \bbR),&
		C^1(\State_*) &\coloneqq\Map^\alt(\Phi_*, \bbR),
	\end{align}
	where $\Map^{\alt}(\Phi_*, \bbR)
	\coloneqq\{\omega\colon\Phi_*\rightarrow \bbR 
	\mid \forall\varphi\in\Phi_*\,\,\omega(\bar\varphi)=-\omega(\varphi) \}$.
	We define the differential
	\begin{equation}\label{eq: diff}
		\partial\colon\CS\rightarrow C^1(\State_*),\qquad f\mapsto\partial f
	\end{equation}
	by  $\partial f(\varphi)\coloneqq f(\state^e) - f(\state)$ for any $\varphi=(\state,\state^e)\in\Phi_*$.	
	The \textit{cohomology} of $(S^X_*,\Phi_*)$ is given by
	\begin{align*}
		H^0(S^X_*)&=\Ker\partial, &  H^1(S^X_*)&=C^1(S^X_*)/\partial C(S^X_*),
	\end{align*}
	and $H^m\bigl(S^X_*)=\{0\}$
	for any $m\in\bbN$ such that $m\neq0,1$.
\end{definition}

We will often call an element in $C^1(\State_*)$ a \textit{form}.
We next show that the differential \eqref{eq: diff} coincides with the differential given in \S\ref{subsec: model}.
By the definition of $\Phi_*$ in \eqref{eq: transition}, we have a surjection $E\times\State_*\rightarrow\Phi_*$.
This shows that we have natural injections
\[
	\xymatrix{
		C^1(\State_*) \quad\ar@{^{(}->}[r]& \quad\Map(\Phi_*, \bbR)\quad\ar@{^{(}->}[r] &\quad\Map(E\times\State_*,\bbR).
	}
\]
Hence we may view a form $\omega\in C^1(\State_*)$ 
as a family of functions $\omega = (\omega_e)_{e\in E}$
through the identification
\[
	\Map(E\times\State_*,\bbR) = \prod_{e\in E}C(\State_*),
\]
where $\omega_e\colon\State_*\rightarrow\bbR$ is the function defined as
\[
	\omega_e(\state)\coloneqq \omega(\varphi),   \qquad \varphi = (\state,\state^e) \in \Phi_*\subset\State_*\times\State_*.
\]
Conversely, any family of functions $(\omega_e) \in\prod_{e\in E} C(\State_*)$ comes from an element $\omega$ in $C^1(\State_*)$ if and only if 
$\omega_e(\state)=\omega_{e'}(\state)$ if $\state^e=\state^{e'}$,
$\omega_e(\state)=0$ if $\state^e=\state$,
and $\omega_e(\state)=-\omega_{\bar e}(\state^e)$ for any $(e,\state)\in E\times\State_*$.
For any $f\in \CS$, if we view $\partial f$
as an element $\partial f = ((\partial f)_e)_{e\in E}$ in $\prod_{e\in E}C(\State_*)$,
then we see that $(\partial f)_e$ is a function satisfying $(\partial f)_e(\state)= f(\state^e) - f(\state)$
for any $\state\in\State_*$.  If we define the function $\nabla_e f\in\Map(\State_*,\bbR)$
for any $e\in E$ by
\[
	\nabla_e(f)(\state)\coloneqq f(\state^e)-f(\state)
\]
for any $\eta\in\State_*$, then we have $\partial f = (\nabla_e f)_{e\in E}$ by construction.
Hence our differential coincides with the differential $\partial$ of \S\ref{subsec: model}.
In what follows, we will often identify a form $\omega\in C^1(\State_*)$ 
with its representation $\omega = (\omega_e)_{e\in E}$ in $\prod_{e\in E}C(\State_*)$.

We next use the conserved quantities of \S \ref{subsec: CQ}
to investigate the cohomology $H^0(\State_*)$.
We say that a function $f\in \CS$ is \textit{horizontal}, if $\partial f=0$.
We have the following.

\begin{lemma}\label{lem: special}
	Suppose $f\in \CS$ is horizontal.
	Then if $\state,\state'\in\State_*$ are in the same connected components of $S^X_*$, 
	then we have $f(\state)=f(\state')$.
	In particular, the function $f$ is constant on the connected components of $S^X_*$.
\end{lemma}

\begin{proof}
	This follows from Lemma \ref{lem: horizontal A}, applied to the graph $(S^X_*,\Phi_*)$.
\end{proof}

Let $\xi$ be a conserved quantity in $\Consv^\phi(S)$.
If we associate to $\xi$
the function $\xi_X\colon\State_*\rightarrow\bbR$ of
\eqref{eq: sum}, then this induces a
homomorphism of $\bbR$-linear spaces $\Consv^\phi(S)\rightarrow \CS$.
This homomorphism is injective since for a fixed $x\in X$, if we let $\state\in\State_*$ 
be the configuration with $s$ in the $x$-component and
at base state in the other components, then we have $\xi(s)=\xi_X(\state)$, hence $\xi_X=\xi'_X$
implies that $\xi=\xi '$ for any $\xi,\xi'\in\Consv^\phi(S)$.

\begin{lemma}\label{lem: horizontal}
	Let $\xi$ be a conserved quantity for the interaction $\phi$.
	Then the function $\xi_X\in \CS$ defined by \eqref{eq: sum} is horizontal.
	In particular, $\xi_X$ defines a class in $H^0(\State_*)$.
\end{lemma}

\begin{proof}
	For any $\varphi=(\state,\state^e)\in\Phi_*$, we have
	\[
		\partial\xi_X(\varphi)\coloneqq\xi_X(\state^e)-\xi_X(\state)
		=\sum_{x\in X}\xi(\eta^e_x) - \sum_{x\in X}\xi(\eta_x).
	\]
	If we let $e=(x_1, x_2)\in E\subset X\times X$, then by definition of $\phi_e$ given in 
	Lemma \ref{lem: SDG},
	we have $(\eta^e_{x_1}, \eta^e_{x_2})=\phi(\eta_{x_1}, \eta_{x_2})$
	and $\eta^e_x=\eta_x$ for $x\neq x_1,x_2$.
	This shows that 
	\[
		\partial\xi_X(\varphi) 
		=\sum_{x\in X}\xi(\eta^e_x) - \sum_{x\in X}\xi(\eta_x)
		=\bigl(\xi\bigl(\eta^e_{x_1}\bigr) + \xi\bigl(\eta^e_{x_2}\bigr)\bigr)-\bigl(\xi(\eta_{x_1}) + \xi(\eta_{x_2})\bigr)=0
	\]
	as desired,
	where the last equality follows from \eqref{eq: conserve}.
\end{proof}

Lemma \ref{lem: horizontal} shows that \eqref{eq: sum} induces an injective
homomorphism of $\bbR$-linear spaces 
\begin{equation}\label{eq: inclusion}
	\Consv^\phi(S)\hookrightarrow H^0(\State_*).
\end{equation}

\begin{remark}
	We remark that $H^0(\State_*)$ is in general very large, containing the $\bbR$-algebra generated by the image of
	$\Consv^\phi(S)$.  For example, the function $(\xi_X)^2$ on $\State_*$ for any conserved quantity
	$\xi\in\Consv^\phi(S)$ also defines an element of $H^0(\State_*)$.
	The property \eqref{eq: conserve2} ensures that $\xi_X$ is an \textit{extensive quantity},
	i.e.\ additive with respect to the size of the system. 
	We will prove in Theorem \ref{thm: cohomology} that under suitable conditions,
	\eqref{eq: inclusion} gives an isomorphism between $\Consv^\phi(S)$
	and the $0$-th \emph{uniform} cohomology of the configuration space with transition structure.
\end{remark}
Let $\state=(\eta_x)$ and $\state'=(\eta'_x)$ be configurations in $\State_*$, and let
$\vec\gamma$ be a finite path from $\state$ to $\state'$.
If $\xi$ is a conserved quantity for the interaction $\phi$, then Lemma 
\ref{lem: horizontal} and Lemma \ref{lem: special} give
the equality
\[
	\xi_X(\state) = \xi_X(\state').
\]
This shows that $\xi_X$ is constant on each of the connected components of $\State_*$.
From now until the end of this subsection, we assume that $X$ is an infinite locale.
We define a \textit{monoid} to be any set with a binary operation that is associative and has an identity element.  
We say that a monoid is commutative if the operation is commutative.

\begin{definition}\label{def: U}
	Assume for simplicity that $\Consv^\phi(S)$ is finite dimensional,
	and fix an $\bbR$-linear basis $\xi^{(1)},\ldots,\xi^{(c_\phi)}$ of $\Consv^\phi(S)$. 
	We define the map
	\[
		\bsxi_{\!X}\colon\State_*\rightarrow\bbR^{c_\phi}
	\]
	by $\bsxi_{\!X}(\state)\coloneqq(\xi_X^{(1)}(\state),\ldots,\xi_X^{c_\phi}(\state))$ 
	for any $\state=(\eta_x)\in\State_*$,
	where $\xi_X^{(\indi)}(\state)\coloneqq\sum_{x\in X}\xi^{(\indi)}(\eta_x)$ for any $\indi=1,\ldots,c_\phi$. 
	Assuming that $X$ is an infinite locale, 
	we let $\Val\coloneqq\bsxi_{\!X}(\State_*)$, which we view as a commutative monoid
	with operation induced from the addition on $\bbR^{c_\phi}$.
	Then the monoid $\Val$ is determined independently up to a natural isomorphism
	of the choice of the basis $(\xi^{(1)},\ldots,\xi^{(c_\phi)})$.
\end{definition}

\begin{remark}\label{rem: intrinsic}
	We may define the commutative monoid $\Val$ of Definition \ref{def: U} 
	intrinsically as an additive submonoid 
	of $\Hom_\bbR(\Consv^\phi(S),\bbR)$ as follows.  We define a map
	\[
		\bsxi_{\!X}^\univ\colon\State_*\rightarrow \Hom_\bbR(\Consv^\phi(S),\bbR)
	\]
	by $\state\mapsto (\xi \mapsto \xi_X(\state))$, and we let
	$\Val\coloneqq\bsxi_{\!X}^\univ(\State_*)$.   	
	In fact, the monoid $\Val$ is defined independently of the choice of the infinite locale $X$.
	In what follows, we will denote the map 
	$
		\bsxi_{\!X}^\univ\colon\State_*\rightarrow \cM
	$
	simply by $\bsxi_{\!X}$.
	If $\Consv^\phi(S)$ is finite dimensional,
	a choice of a basis $\xi^{(1)},\ldots,\xi^{(c_\phi)}$ of $\Consv^\phi(S)$ gives an
	isomorphism $\Hom_\bbR(\Consv^\phi(S),\bbR)\cong\bbR^{c_\phi}$,
	and the monoid $\Val$ maps to the $\Val$ of Definition \ref{def: U} 
	through this isomorphism.
	
	By Lemma \ref{lem: horizontal} and Lemma \ref{lem: special},
	if the configurations $\state, \state'\in\State_*$ are connected by a path, then 
	it is in the same fiber of the map $\bsxi_{\!X}$.	
	On the other hand, the condition that the interaction is irreducibly quantified
	implies that the fibers of the map $\bsxi_{\!X}$ are connected.
	Thus in this case, the connected components of $\State_*$ corresponds 
	bijectively with the elements of $\Val$.
	The authors thank Hiroyuki Ochiai for suggesting this formulation.
\end{remark}

By Remark \ref{rem: intrinsic}, if the interaction is irreducibly quantified,
then the connected components of $\State_*$ correspond bijectively
with the elements of the image $\Val$ of $\bsxi_X$.  
In particular, if $c_\phi=0$,
then $\State_*$ is connected.  However,
if $c_\phi>0$, then we have
$\dim_\bbR H^0(\State_*)=\infty$. 
Although $H^m(\State_*)$ is the standard cohomology 
of $\State_*$ and reflects the topological structure of the graph $(\State_*,\Phi_*)$,
it is not so useful in the sense that it is in general infinite dimensional over infinite locales.
In the \S\ref{sec: local}, we will define the notion of uniform functions
which gives a certain subspace $C_\unif(S^X)\subset C(S^X_*)$,
and will define in \S\ref{sec: cohomology}
the uniform cohomology $H^m_\unif(S^X)$ by replacing the functions and forms of \eqref{eq: cohomology}
with uniform functions and forms.
For uniform cohomology, the inclusion of \eqref{eq: inclusion} induces an isomorphism $\Consv^\phi(S)\cong H^0_\unif(S^X)$, hence $H^0_\unif(S^X)$ is finite dimensional if $c_\phi$ is finite.

For the first cohomology,
the group $H^1(\State_*)$ is in general also large, since there are many linearly independent
forms which are
not exact.  
For example, $H^1$ may be infinite dimensional as follows.

\begin{remark}
	Let $X=(\bbZ^d,\bbE^d)$ and $S=\{0,1,\ldots,\kappa\}$ for an integer $d>1$ and some natural number $\kappa>0$.  
	If we let $\State_*$ be the configuration space for $S$ on $X$ with transition structure 
	for the multi-species exclusion process 
	of Example \ref{example: interactions}, 
	then we have $\dim_\bbR H^1(\State_*)=\infty$.
	
	This may be seen as follows.
	For any $x\in X$, let $1_x\in\State_*$ be the configuration with $1$ in the $x$ 
	component and $0$ in the other components. 
	For any edge $e =(x,x')\in \bbE$, the configuration $1_x^e$ is the 
	configuration $1_{x'}$ with $1$ in the $x'$ component and $0$ in the other components,
	hence $1_e\coloneqq(1_x,1_{x'})$ is a transition of $\State_*$.
	If we let $\omega\in C^1(\State_*)=
	\Map^\alt(\Phi_*,\bbR)$ be the form given by $\omega(1_e)
	=-\omega(1_{\bar e})=1$ and $\omega(\varphi)=0$ for $\varphi\neq 1_e, 1_{\bar e}$, 
	then we may prove that $\omega$ gives a nonzero element in $H^1(\State_*)$.
	In addition, we may prove that such $\omega$ for a finite set of edges in $E$ 
	which do not share common vertices give linearly independent 
	elements of $H^1(\State_*)$, giving our assertion.
\end{remark}

In considering uniform cohomology,
we will  consider a class of forms called \textit{closed forms}, 
which are in fact always exact.
We recall that a finite path on a graph $\State_*$ 
is a finite sequence $(\varphi^1,\ldots,\varphi^N)$
of \textit{transitions} of $S^X_*$ such that $t(\varphi^i)=o(\varphi^{i+1})$ for any integer $0<i<N$.
As in \eqref{eq: integral A}, for a form $\omega\in C^1(\State_*)$ on $\State_*$,
we define the
\textit{integral} of $\omega$ with respect to the  path $\vec\gamma=(\varphi^1,\ldots,\varphi^N)$ by
\[
	\int_{\vec\gamma}\omega\coloneqq\sum_{j=1}^N\omega(\varphi^j).
\]
As in Definition \ref{def: closed A}, we define a closed form as follows.

\begin{definition}\label{def: closed}
	We say that a form $\omega\in C^1(\State_*)=\Map^\alt(\Phi_*,\bbR)$ is \textit{closed},
	if for any closed path $\vec\gamma$ in $\State_*$, we have
	\[
		\int_{\vec\gamma}\omega=0.
	\]
\end{definition}

We say that a form $\omega\in C^1(\State_*)$ is \textit{exact}, 
if there exists a function $f\in \CS$ such that  $\partial f = \omega$.
By Lemma \ref{lem: equivalent A}, a form is exact if and only if it is closed.

\begin{lemma}\label{lem: exact}
 	A form $\omega\in C^1(S^X_*)$ is exact if and only if it is closed.
\end{lemma}

\begin{proof}
	This follows from Lemma \ref{lem: equivalent A} applied to the graph $(S^X_*,\Phi_*)$.
\end{proof}

We will denote by $Z^1(S^X_*)$ the space of closed forms on $(S^X_*,\Phi_*)$.
Closed forms will play a role in the definition of uniform cohomology.

%
%
%
\section{Uniform Functions and Uniform Forms}\label{sec: local}
%
%
%

In this section, we will define the notion of \textit{uniform functions} and 
\textit{uniform forms}, which are functions and forms which reflect the geometry of the underlying
locale.  We will then investigate its properties, including a criterion for a function to be 
uniform.

%
\subsection{Uniform Functions on the Configuration Space}\label{subsec: uniform}
%

For any system $(X,S,\phi)$, let $\State = \prod_{x\in X}S$
be the configuration space for $S$ on $X$. 
We let $\State_* \subset\State$ be the subset consisting of configurations with 
finite support.
 In this subsection, we will prove the existence
of a canonical expansion of functions in $C(\State_*)$ in terms
of local functions with exact support (see Definition \ref{def: exact support}), 
and we will introduce the notion
of a \textit{uniform function}, which are functions which reflect the geometry
of the underlying locale.

For a finite $\La\subset X$, there exists a natural  inclusion
$
	\iota^\La\colon S^\La  \hookrightarrow S^\La \times S^{X\setminus\La}_* = \State_*
$
given by $\eta_\La \mapsto (\eta_\La,\star)$ for any $\eta_\La\in S^\La$, 
where $\star\in S^{X\setminus\La}$ is the element whose components are all at base state.
By abuse of notation, we will often denote $\iota^\La(\state|_\La)$ by $\state|_\La$.
This inclusion induces
a homomorphism
\begin{equation}\label{eq: operator}
	\iota^\La\colon \CS\rightarrow C(S^\La)\subset \CS,
\end{equation}
which may be regarded as an $\bbR$-linear operator on the set of functions $\CS$.
Note that we have $\iota^\La f(\state)=f(\state|_\La)$ for any $f\in \CS$ and $\state\in\State$.
By definition, if $f\in C(S^\La)$, then we have $\iota^{\La} f=f$.
For any $\La,\La'\subset X$, we have $\iota^{\La}\iota^{\La'} = \iota^{\La'}\iota^{\La}
=\iota^{\La\cap\La'}$.

\begin{definition}\label{def: exact support}
	For any finite $\La\subset X$, we let
	\[
		C_\La(\State)\coloneqq\big\{ f\in C(S^\La)
		\mid f(\state)=0\text{ if $\exists x\in \La$ such that $\eta_x=*$}\big\}.
	\]
	We call any function $f\in C_\La(\State)$ a local function with \textit{exact support} $\La$.
\end{definition}

Any function in $\CS$ has a unique expansion in terms of local functions with exact support,
as will be shown in Proposition \ref{prop: expansion}.
We first start with the following lemma.
\begin{lemma}\label{lem: sum}
	Let $(f_\La)$ be a set of functions such that $f_\La\in C_\La(\State)$ for any finite $\La\subset X$.
	Then the sum
	\[
		f\coloneqq\sum_{\La\subset X, |\La|<\infty}f_\La
	\]
	defines a function in $\CS$.
\end{lemma}

\begin{proof}
	By definition, for any $\state=(\eta_x)\in\State_*$, 
	the support $\Supp(\state)\subset X$ satisfies $|\Supp(\state)|<\infty$.
	Then we have
	\[
		f(\state)\coloneqq\sum_{\La\subset X,|\La|<\infty}f_{\La}(\state)= 
		\sum_{\La\subset\Supp(\state)}f_{\La}(\state),
	\]
	where the last sum is defined since it is a finite sum.  We see that the sum defines a function
	$f\colon\State\rightarrow\bbR$ as desired.
\end{proof}

The expansion of functions in $\CS$ in terms of local functions with exact support
is given as follows.

\begin{proposition}\label{prop: expansion}
	For any $f\in \CS$, there exists a unique expansion
	\begin{equation}\label{eq: expansion}
		f=\sum_{\La\subset X, |\La|<\infty} f_\La,
	\end{equation}
	in terms of local functions with exact support $f_\La\in C_\La(\State)$ for finite $\La\subset X$.
\end{proposition}

\begin{proof} 
	We construct $f_\La$ by induction on the cardinality of $\La$.
	Suppose an expansion of the form \eqref{eq: expansion} exists. 
	Note that for any $\La, \La'\subset X$ such that 
	$\La'\not\subset\La$, we have $\iota^{\La} f_{\La'}=0$ 
	since $f_{\La'}\in C_{\La'}(\State)$, hence if we apply the $\bbR$-linear
	operator $\iota^{\La}$
	of \eqref{eq: operator} on \eqref{eq: expansion}, then we obtain the equality
	\[
		\iota^{\La} f=\sum_{\La'\subset\La} f_{\La'}.
	\]
	Hence this shows that assuming the existence of the expansion,
	 $f_\La$ is uniquely given inductively for the set $\La$ by
	\begin{equation}\label{eq: fL}
		f_\La=\iota^\La f - \sum_{\La'\subsetneq\La} f_{\La'}.
	\end{equation}
	We will prove by induction on the cardinality of $\La$
	that the function $f_\La$ inductively given by \eqref{eq: fL} 
	is a function in $C_\La(\State)$.
	We let $\star\in\State$ be the base state.
	For $\La=\emptyset$, equation \eqref{eq: fL} gives
	$
		f_\emptyset=\iota^\emptyset f,
	$
	which shows that $f_\emptyset$ is an element in $C_\emptyset(\State)$.
	Note that $f_\emptyset$ is the constant function given by 
	$f_\emptyset(\state)=f(\star)$ for any $\state\in\State$.
	For $|\La|>0$, suppose $f_{\La'}\in C_{\La'}(\State)$ 
	for any $\La'\subsetneq\La$.
	Then by \eqref{eq: fL}, the function $f_\La$ is a function in $C(S^\La)$.
	Next, we prove that $f\in C_\La(\State)$.
	Suppose $\state=(\eta_x)\in\State$ satisfies $\eta_x=*$ for
	some $x\in\La$. If $\La'\subsetneq\La$ and $x\in \La'$, then we have 
	$f_{\La'}(\state)=0$ since $f_{\La'}\in C_{\La'}(\State)$.  
	Hence by \eqref{eq: fL}, we have
	\[
		f_\La(\state)
		=\iota^\La f(\state) - \sum_{\La'\subset\La\setminus\{x\}} f_{\La'}(\state).
	\]
	Note that since $\eta_x=*$, we have $\iota^\La f(\state)
	=\iota^{\La\setminus\{x\}}f(\state)$ by definition of $\iota^\La$, 
	hence we have
	\begin{align*}
		f_\La(\state)
		&=\iota^{\La\setminus\{x\}}
		f(\state) - \sum_{\La'\subset\La\setminus\{x\}} f_{\La'}(\state)
		=\biggl(\iota^{\La\setminus\{x\}}
		f(\state) - \sum_{\La'\subsetneq\La\setminus\{x\}} f_{\La'}(\state)\biggr)
		- f_{\La\setminus\{x\}}(\state)\\
		&=  f_{\La\setminus\{x\}}(\state)- f_{\La\setminus\{x\}}(\state)=0.
	\end{align*}
	This proves that $f_\La\in C_\La(\State)$ as desired.
	The sum \eqref{eq: expansion} gives the function $f$ in $\CS$, 
	since  for any $\state\in\State$, we have 
	$f(\state)=\iota^{\Supp(\state)}f(\state)=\sum_{\La\subset\Supp(\state)}
	f_{\La}(\state)$.
\end{proof}

\begin{corollary}\label{cor: expansion}
	If $f\in C(S^\La)$ for some finite $\La\subset X$, 
	then we have a unique expansion
	\[
		f=\sum_{\La'\subset\La} f_{\La'},
	\]
	in terms of local functions with exact support $f_{\La'}\in C_{\La'}(\State)$ for $\La'\subset\La$.
\end{corollary}

\begin{proof}
	Our statement follows by applying the $\bbR$-linear operator $\iota^\La$ to
	the expansion \eqref{eq: expansion} of
	Proposition \ref{prop: expansion}, noting that $\iota^\La f=f$
	and $\iota^\La f_{\La'}=0$ if $\La'\not\subset\La$.
\end{proof}

For any $\La\subset X$, we define the diameter $\diam{\La}$
of $\La$ by
\[
	\diam{\La}\coloneqq\sup_{x,x'\in\La} d_X(x,x').
\]
Since $X$ is connected, if $\La\subset X$ is finite, then we have $\diam{\La}<\infty$.
Uniform functions are defined as follows.

\begin{definition}\label{def: uniform}
	We say that a function $f\in \CS$ is \textit{uniform}, if there exists $R>0$
	such that the canonical expansion of \eqref{eq: expansion} is given by
	\[
		f=\sum_{\substack{\La\subset X\\\diam{\La}\leq R}} f_\La.
	\]
	In other words, $f_\La=0$ in the expansion \eqref{eq: expansion}
	if $\diam{\La}>R$.
	We denote by $C_\unif(S^X)$ the $\bbR$-linear subspace of $\CS$ consisting of
	uniform functions, and by $\CuS$ the subspace of $C_\unif(S^X)$
	consisting of functions satisfying $f(\star)=0$.
\end{definition}

For any $x\in X$ and $\La\subset X$, we let
$
	d_X(x,\La)\coloneqq\inf_{x'\in \La}d_X(x,x'),
$
and for $\La,\La'\subset X$, we let
$
	d_X(\La,\La')\coloneqq\inf_{(x,x')\in \La\times \La'}d_X(x,x').
$

\begin{remark}\label{rem: converse}
	Suppose $f\in C^0_\unif(\State)$ so that there exists some $R>0$
	such that 
	 $f_\La\equiv0$ in the expansion \eqref{eq: expansion}
	if $\diam{\La}>R$.
	Let $\La$ and $\La'$ be subsets of $X$ such that 
	$d_X(\La,\La')>R$. By Corollary \ref{cor: expansion}, we have
	$\iota^{\La\cup\La'} f =\sum_{\La''\subset\La\cup\La'}f_{\La''}$,
	$\iota^{\La} f =\sum_{\La''\subset\La}f_{\La''}$ and $\iota^{\La'} f =\sum_{\La''\subset\La'}f_{\La''}$.
	Note that if $\La''\subset\La\cup\La'$ satisfies 
	$\La''\cap\La\neq\emptyset$ and $\La''\cap\La'\neq\emptyset$,
	then we have $\diam{\La''}>R$, hence $f_{\La''}=0$ from our choice of $R$. 
	This shows that we have
	$
		\iota^{\La\cup\La'}f=\iota^{\La} f+\iota^{\La'}f,
	$
	where we have used the fact that $f(\star)=\iota^\emptyset f=0$.
\end{remark}

%
\subsection{Horizontal Uniform Functions}\label{subsec: horizontal}
%

Consider the system $(X,S,\phi)$.
From now until the end of \S\ref{sec: local}, 
we assume that the interaction $\phi$ is irreducibly quantified in the
sense of Definition \ref{def: irreducibly quantified}.
The purpose of this subsection is to prove the following theorem.

\begin{theorem}\label{thm: 2}
	For the system $(X,S,\phi$), assume that $X$ is an infinite locale 
	and the interaction $\phi$ is irreducibly quantified.
	Let $f$ be a uniform function in $C^0_\unif(\State)$.
	If $f$ is horizontal, i.e.\ if $\partial f=0$, 
	then there exists a conserved quantity $\xi\colon S\rightarrow\bbR$
	such that
	\[
		f(\state)=\sum_{x\in X}\xi(\eta_x)
	\]
	for any $\state=(\eta_x)\in\State$.
\end{theorem}

Theorem \ref{thm: 2} implies that assuming that $X$ is infinite and $\phi$ is irreducibly quantified,
any uniform function which is constant on the connected components of $\State_*$
coincides with $\xi_X$ for some conserved quantity $\xi\in\Consv^\phi(S)$.
We will give the proof of Theorem \ref{thm: 2} at the end of this subsection.
We first prove the following lemma.

\begin{lemma}\label{lem: three}
	Suppose $f\in \CS$, and for any $x\in X$, let
	$f_{\{x\}}$ be the function with exact support $\La=\{x\}$ 
	in the canonical expansion \eqref{eq: expansion}.
	If $f$ is horizontal, then the functions 
	\[
		\iota^{\{x\}}  f_{\{x\}} \colon S\rightarrow \bbR
	\]
	are all equal as $x$ varies over $X$.
\end{lemma}

\begin{proof}
	Consider the expansion
	\[
		f=\sum_{\La\subset X, |\La|<\infty} f_\La
	\] 
	of \eqref{eq: expansion}. 
	We let $s$ be any element in $S$.
	For $x,x'\in X$, we let
	$\state\coloneqq\iota^{\{x\}}(s)$ and $\state'\coloneqq\iota^{\{x'\}}(s)$
	be the configuration in $\State_*$ with $s$ respectively in the $x$ and $x'$ components,
	and base state $*$ in the other components.  This implies that
	\[
		\sum_{x\in X}\xi(\eta_x) = \sum_{x\in X}\xi(\eta'_x) = \xi(s),
	\]
	in other words that $\xi_X(\state)=\xi_X(\state')$
	for any conserved quantity $\xi\in\Consv^\phi(S)$.
	Hence from the fact that
	the interaction is irreducibly quantified, there exists a finite path $\vec\gamma$ from $\state$ to $\state'$ in $\State_*$.
	By Lemma \ref{lem: special}, noting that $\partial f=0$, we have
	\[
		\iota^{\{x\}} f_{\{x\}}(s) = f(\state) = f(\state') =\iota^{\{x'\}}  f_{\{x'\}}(s),
	\]
	which shows that $\iota^{\{x\}}f_{\{x\}}\colon S\rightarrow\bbR$ 
	is independent of the choice of $x\in X$ as desired.
\end{proof}

Next, we prove the following lemma.

\begin{lemma}\label{lem: four}
	Assume that $X$ is an infinite locale, and
	suppose $f$ is a uniform function in $C^0_\unif(\State)$.
	If $f$ is horizontal, then we have
	\[
		f= \sum_{x\in X} f_{\{x\}},
	\]
	where $f_{\{x\}}$ is the function with exact support $\La=\{x\}$ 
	in the canonical expansion \eqref{eq: expansion}.
\end{lemma}

\begin{proof}
	Since $f$ is uniform, there exists an $R>0$ such that
	$f_\La\equiv0$ if $\diam{\La}>R$.
	By Proposition \ref{prop: expansion},
	it is sufficient to prove that for any finite $\La\subset X$ satisfying $|\La|>1$,
	we have $f_\La\equiv0$ in the expansion of \eqref{eq: expansion}.
	We will prove this by induction on the cardinality of $\La\subset X$.
	We consider a finite $\La\subset X$ such that $n\coloneqq|\La|>1$, and
	assume that $f_{\La'}\equiv 0$ for any $\La'\subset X$
	such that  $1<|\La'|<n$.  
	Note that this condition is trivially true for $n=2$. 
	We let  $\La_n\subset X$ be a subset of $X$
	with $n$ elements such that $\diam{\La_n}>R$.
	Such $\La_n$ exists since $X$ is a locally finite infinite graph that is connected.
	Then by construction, we have $f_{\La_n}\equiv 0$.
	We fix a bijection between the set $\{1,\cdots,n\}$ and the sets $\La$ and $\La_n$,
	which induces bijections $S^n \cong S^\La$ and $S^n\cong  S^{\La_n}$.
	For any $(\eta_\indi)	\in S^n$, which  we view as an element in $S^\La$
	and $S^{\La_n}$, we let $\state\coloneqq\iota^{\La}((\eta_\indi))$ and 
	$\state'\coloneqq\iota^{\La_n}((\eta_\indi))$.
	Then for any conserved quantity $\xi\colon S\rightarrow\bbR$, we have
	\[
		\sum_{x\in X}\xi(\eta_x)=\sum_{x\in X}\xi(\eta'_x)=\sum_{\indi=1}^n \xi(\eta_\indi),
	\]
	in other words, $\xi_X(\state)=\xi_X(\state')$.
	Since the interaction is irreducibly quantified, there exists a finite path 
	$\vec\gamma$ from $\state$ to $\state'$ in $\State_*$.
	By Lemma \ref{lem: special} and our condition that $\partial f=0$, we have $f(\state) = f(\state')$.
	Note that by Lemma \ref{lem: three},
	if we let $\zeta\coloneqq\iota^{\{x\}}f_{\{x\}}
	\colon S\rightarrow\bbR$, then $\zeta$ is independent of the choice of $x\in X$.
	Corollary \ref{cor: expansion} implies that we have
	\begin{align*}
		 f(\state) &= \iota^{\La}  f(\state) = f_{\La}(\state) +
		 \sum_{\La'\subsetneq\La} f_{\La'}(\state)
		 = f_{\La}(\state) +
		 \sum_{x\in\La} f_{\{x\}}(\state)= f_{\La}(\state) 
		+ \sum_{\indi=1}^n \zeta(\eta_\indi)\\
		 f(\state') &= \iota^{\La_n}  f(\state') = f_{\La_n}(\state') +
		 \sum_{\La'\subsetneq\La_n} f_{\La'}(\state')
		 = 
		 \sum_{x\in\La_n}f_{\{x\}}(\state')= 
		 \sum_{\indi=1}^n \zeta(\eta_\indi),
	\end{align*}
	hence we have $f_{\La}(\state)=0$.
	Since this was true  for any $(\eta_\indi)\in S^n$, we have $f_\La\equiv 0$ as desired.
	Our assertion now follows by induction on $n$.
\end{proof}

We may now prove Theorem \ref{thm: 2}.

\begin{proof}[Proof of Theorem \ref{thm: 2}]
	Let $\xi\coloneqq\iota^{\{x\}}f_{\{x\}}\colon S\rightarrow\bbR$,
	which by Lemma \ref{lem: three} is independent of the choice of $x\in X$.
	Then by Lemma \ref{lem: four} and the definition of $\iota^{\{x\}}f_{\{x\}}$,
	 we have
	\[
		f(\state)=\sum_{x\in X}\xi(\eta_x)
	\]
	for any $\state=(\eta_x)\in\State_*$.
	It is sufficient to show that $\xi$ is a conserved quantity.
	First, note that we have $\xi(*)=f_{\{x\}}(*)=0$ for any $x\in X$
	since $f_{\{x\}}$ has exact support $\{x\}$.
	Consider an edge $e=(o(e),t(e))\in E\subset X\times X$, and for $(s_1,s_2)\in S\times S$,
	let $\state\in\State_*$ be the configuration with $s_1$ in the $o(e)$ component, $s_2$ in the $t(e)$
	component, and base states $*$ at the other components.
	Then $\state^e$ is the configuration with $s'_1$ in the $o(e)$ component, $s'_2$ in the $t(e)$
	component, and base states at the other components, where $(s_1',s_2')=\phi(s_1,s_2)$.
	If we let $\varphi=(\state,\state^e)$ be the transition from $\state$ to $\state^e$,
	then $\partial f=0$ implies that
	\[
		\partial f(\varphi) = f(\state^e) - f(\state) = (\xi(s_1')+\xi(s_2'))-(\xi(s_1)+\xi(s_2))
		=0.
	\]
	This shows that $\xi$ satisfies \eqref{eq: conserve2}, hence it is a conserved quantity as desired.
\end{proof}

%
\subsection{Pairings for Functions with Uniform Differentials}\label{subsec: pairing}
%

Let $X$ be an infinite locale and assume that the interaction is irreducibly quantified.
In this subsection, we first define the notion of \textit{uniform forms}.
Next, 
we will prove Proposition \ref{prop: h_f}, which associates 
a certain pairing $h_f\colon\Val\times\Val\rightarrow\bbR$
to any function $f\in C(\State_*)$ whose differential $\partial f\in C^1(\State_*)$ is a 
uniform form.
where $\Val$ is the commutative monoid given in
Definition \ref{def: U} (see also Remark \ref{rem: intrinsic}).

A \textit{ball} in $X$ is a set of the form
$
	 B(x,R)\coloneqq\{ x'\in X\mid d_X(x,x')\leq R\}
$
for some $x\in X$ and constant $R>0$.
We say that $x$ is the \textit{center} and $R>0$ is the \textit{radius} of $ B(x,R)$.
If $\Ba$ is a ball in $X$, then we denote by $r(\Ba)$ the radius of $\Ba$.
For any $\La\subset X$, we let $ B(\La,R)\coloneqq\bigcup_{x\in\La}  B(x,R)$, which we call the
\textit{$R$-thickening} of $\La$.
In particular, for any edge $e=(o(e),t(e))\in E \subset X\times X$, we let 
$ B(e,R)\coloneqq  B(o(e),R)\cup  B(t(e),R)$.
For any $R>0$,
we define the set of \textit{$R$-uniform forms} 
$C^1_R(\State) \subset C^1(\State_*)\subset\prod_{e\in E}\Map(\State_*,\bbR)$
by 
\[
	C^1_R(\State)\coloneqq C^1(\State_*)\cap\prod_{e\in E}C\bigl(S^{ B(e,R)}\bigr).
\]

\begin{definition}\label{def: unif form}
	We define the space of \textit{uniform forms} on $\State$ to be the $\bbR$-linear space
	\[	
		C^1_\unif(\State)\coloneqq\bigcup_{R>0}C^1_R(\State).
	\]
	We define a \textit{closed uniform form} to be a uniform form which is closed in the
	sense of Definition \ref{def: closed}.
	We will denote by $Z^1_\unif(\State)$ the space of closed uniform forms.
\end{definition}

For any subset $Y\subset X$, we denote by $\pr_{\subl}$ the map of sets
$\pr_{\subl}\colon\State_*\rightarrow S^\subl_*$
induced from the natural projection $\pr_{\subl}\colon S^X\rightarrow S^\subl$.
For any $\state\in\State_*$,
we will often denote $\pr_\subl(\state)$ 
by $\state|_\subl$.
By abuse of notation, we often write $\state|_\subl$ for the configuration $\iota_\subl(\state|_\subl)$
in $\State_*$.
We say that the configurations $\state,\state'\in\State_*$ 
\textit{coincide outside $\subl$}, if $\state|_{\sublc}=\state'|_{\sublc}$
for $\sublc\coloneqq X\setminus\subl$.

For any conserved quantity $\xi\colon S\rightarrow\bbR$ and $W\subset X$,
we define the function $\xi_W\colon\State_*\rightarrow\bbR$ by 
$\xi_W(\state)\coloneqq\sum_{x\in W}\xi(\eta_x)$ for any $\state\in\State_*$.
The following result concerns the values of functions whose differential are uniform.

\begin{lemma}\label{lem: difference}
	Let $f\in \CS$ and assume that  $\partial f\in C^1_R(\State)$
	for some $R>0$.  Let $Y\subset X$ be a sublocale, and suppose $\state,\state'\in\State_*$
	are configurations which coincide outside $Y$
	and satisfy $\xi_Y(\state)=\xi_Y(\state')$ for any conserved quantity
	$\xi\in\Consv^\phi(S)$.
	Suppose that the interaction is irreducibly quantified.
	If $\wt Y$ is any subset of $X$ such that 
	$ B(Y,R)\subset\wt Y$, then we have
	\[
		f(\state')-f(\state)=f(\state'|_{\wt Y})-f(\state|_{\wt Y}).
	\]
\end{lemma}

\begin{proof}
	Since $\partial f\in C^1_R(\State)$, we have $\nabe f\in C\bigl(S^{ B(e,R)}\bigr)$.
	Hence since $ B(Y,R)\subset \wt Y$, we have
	\begin{equation}\label{eq: W}
		\nabe f(\state)= \nabe f(\state|_{\wt Y})
	\end{equation}
	for any $e\subset Y$.
	The condition $\xi_Y(\state)=\xi_Y(\state')$ implies that $\xi_Y(\state|_Y)=\xi_Y(\state'|_Y)$
	for any conserved quantity $\xi$.
	Since the interaction is irreducibly quantified, there exists a path $\vec\gamma|_Y$
	from $\state|_Y$ to $\state'|_Y$ in $S^Y_*$.
	If we let $\bse=(e^1,\ldots,e^N)$ be the sequence of edges in $Y$
	such that $\vec\gamma|_Y=\vec\gamma_{\state|_Y}^\bse$,
	then since $\state$ and $\state'$ coincide outside $Y$,
	the path $\vec\gamma\coloneqq\vec\gamma_{\state}^\bse$
	gives a path from $\state$ to $\state'$ and
	the path $\vec\gamma|_{\wt Y}\coloneqq\vec\gamma_{\state|_{\wt Y}}^\bse$
	gives a path from $\state|_{\wt Y}$ to $\state'|_{\wt Y}$.
	Note that for any $e\in E$, by definition, 
	$\nabe f(\state)= f(\state^e)-f(\state)$.  Hence
	\begin{align*}
		f(\state')&-f(\state)=\sum_{i=1}^N \nabla_{\!e^i} f(\state^{{i-1}}),&
		f(\state'|_{\wt Y})&-f(\state|_{\wt Y})
		=\sum_{i=1}^N \nabla_{\!e^i} f(\state^{{i-1}}|_{\wt Y}),
	\end{align*}
	where we let $\state^0\coloneqq\state$ and $\state^i\coloneqq(\state^{i-1})^{e^i}$
	for $i=1,\ldots, N$.
	Our assertion now follows from \eqref{eq: W}.
\end{proof}

For the remainder of this article, we consider the map
\[
	\bsxi_{\!X}\colon C(S^X_*)\rightarrow\cM
\]
given in Remark \ref{rem: intrinsic}. 
For finite $\La$, we denote by $\Val_{|\La|}\coloneqq\bsxi_\La(\State_*)\subset\cM$ 
the image of $\State_*$ with respect to the map $\bsxi_\La$.  
As the notation suggests, the set $\Val_{|\La|}$ as a subset of $\Val$ 
depends only on the cardinality of $\La$.
We have the following.

\begin{proposition}\label{prop: h}
	Let $f\in \CS$ and suppose $\partial f\in C^1_R(\State)$ for some $R>0$.
	Then for any finite connected $\La,\La'\subset X$ such that $d_X(\La,\La')>R$,
	there exists a function $h_f^{\La,\La'}\colon\Val_{|\La|}\times\Val_{|\La'|}\rightarrow\bbR$
	such that
	\[
		\iota^{\La\cup\La'}f(\state)
		-\iota^{\La}f(\state)
		-\iota^{\La'}f(\state)
		=h_f^{\La,\La'}\bigl(\bsxi_{\!\La}(\state),\bsxi_{\!\La'}(\state)\bigr)
	\]
	for all $\state\in\State_*$.
\end{proposition}

\begin{proof}
	Let $\La,\La'$ be finite connected subsets of $X$ such that $d_X(\La,\La')>R$.
	Noting that $\iota^W f(\state)=f(\state|_W)$ for any $W\subset X$,
	it is sufficient to prove the statement for $\state\in\State_*$ with support in $\La\cup\La'$.
	Consider configurations $\state,\state'\in\State_*$ 
	such that $\Supp(\state),\Supp(\state')\subset \La\cup\La'$,
	and satisfying $\bsxi_\La(\state)=\bsxi_\La(\state')$ and  $\state|_{\La'}=\state'|_{\La'}$.
	Then by construction, we have
	$\bsxi_{\!X}(\state)=\bsxi_{\!X}(\state')$. 
	Since the interaction is irreducibly quantified and the configurations $\state$ and $\state'$ 
	coincide outside $\La$, by Lemma \ref{lem: difference} 
	applied to $Y\coloneqq\La$, which is a locale, and $\wt Y\coloneqq B(\La,R)$, 
	we have
	\[
		f(\state')-f(\state)=f(\state'|_{ B(\La,R)})-f(\state|_{ B(\La,R)}).
	\]
	Since $\state'=\state'|_{\La\cup\La'}$ and $\state=\state|_{\La\cup\La'}$, 
	noting that $(\La\cup\La') \cap  B(\La,R) =\La$,
	we have
	\[
		f(\state'|_{\La\cup\La'})-f(\state'|_{\La})=f(\state|_{\La\cup\La'})-f(\state|_{\La}).
	\]
	This shows that the function
	$
		\iota^{\La\cup\La'}f(\state)
		-\iota^{\La}f(\state)
	$
	depends only on $\bsxi_{\!\La}(\state)$ and $\state|_{\La'}$ if 
	$d_X(\La,\La')>R$.
	Since $\iota^{\La'}f(\state)$ also depends only on $\state|_{\La'}$,
	we see that 
	\begin{equation}\label{eq: symmetry}
		\iota^{\La\cup\La'}f(\state)
		-\iota^{\La}f(\state) -\iota^{\La'}f(\state)
	\end{equation}
	is a function which depends only on $\bsxi_{\!\La}(\state)$ and $\state|_{\La'}$.
	Due to symmetry, we can also see that \eqref{eq: symmetry} depends only on 
	$\bsxi_{\!\La'}(\state)$ and $\state|_{\La}$.  This implies that \eqref{eq: symmetry}
	only depends on $\bsxi_{\!\La}(\state)$ and $\bsxi_{\!\La'}(\state)$,
	hence there exists a function $h_f^{\La,\La'}\colon\Val_{\!|\La|}\times\Val_{\!|\La'|}\rightarrow\bbR$
	such that
	\[
			\iota^{\La\cup\La'}f(\state)
		-\iota^{\La}f(\state) -\iota^{\La'}f(\state)
		= h_f^{\La,\La'}\bigl(\bsxi_{\!\La}(\state),\bsxi_{\!\La'}(\state)\bigr)
	\]
	for any $\state\in\State_*$ as desired.
\end{proof}

\begin{lemma}\label{lem: independence}
	Let $f\in \CS$ and $\partial f\in C^1_R(\State)$ for some $R>0$.
	For finite connected subsets $\La,\La',\La''$ in $X$,
	suppose $\La'$ and $\La''$ are in the same connected component of $X\setminus  B(\La,R)$.
	Then the functions $h_f^{\La,\La'}$ and $h_f^{\La,\La''}$ of Proposition \ref{prop: h} satisfy
	\[
		h_f^{\La,\La'}(\gra,\grb)=h_f^{\La,\La''}(\gra,\grb)
	\]
	for any $\gra\in\Val_{|\La|}$ and $\grb\in\Val_{|\La'|}\cap\Val_{|\La''|}$.
\end{lemma}

\begin{proof}
	Let $\gra\in\Val_{|\La|}$ and $\grb\in\Val_{|\La'|}\cap\Val_{|\La''|}$.
	Then since $\La\cap\La'=\La\cap\La''=\emptyset$,
	there exists states $\state'$ and $\state''\in\State$ 
	at base state outside $\La\cup\La'$ and  $\La\cup\La''$ respectively
	satisfying
	$\bsxi_{\!\La}(\state')=\bsxi_{\!\La}(\state'')=\gra$ and $\bsxi_{\!\La'}(\state')=
	\bsxi_{\!\La''}(\state'')=\grb$.  Moreover, we may choose $\state'$ and
	$\state''$ so that they coincide on $\La$.
	Let $Y$ be a sublocale of $X$ containing $\La'\cup\La''$ such that $d_X(Y,\La)>R$.
	By our choice of $\state'$ and $\state''$, we have 
	$\bsxi_{\!X}(\state')=\bsxi_{\!X}(\state'')=\gra+\grb$.  
	Since the interaction is irreducibly quantified,
	by Lemma \ref{lem: difference} applied to $Y$ and $\wt Y= B(Y,R)$, we have
	\begin{equation}\label{eq: before}
		f(\state'')-f(\state')=f(\state''|_{ B(Y,R)})-f(\state'|_{ B(Y,R)}).
	\end{equation}
	Since $\state'$ and $\state''$ are at base state outside $\La\cup\La'$ and  $\La\cup\La''$
	and $\La\cap  B(Y,R)=\emptyset$,
	we have $\state'=\state'|_{\La\cup\La'}$, $\state''=\state''|_{\La\cup\La''}$,
	$\state'|_{ B(Y,R)}=\state'|_{\La'}$, and
	$\state''|_{ B(Y,R)}=\state''|_{\La''}$.  Noting that $\iota^W f(\state)=f(\state|_W)$ for any 
	$W\subset X$, equation \eqref{eq: before} gives
	\[
		\iota^{\La\cup\La''}f(\state'')-\iota^{\La''}f(\state'')=\iota^{\La\cup\La'}f(\state')-\iota^{\La'}f(\state').
	\]
	Noting also that $\iota^{\La}f(\state')=\iota^{\La}f(\state'')$ since $\state'$ and $\state''$ coincide on 
	$\La$, by the definition of 
	$h_f^{\La,\La'}$ and $h_f^{\La,\La''}$, we  have
	\begin{align*}
		h_f^{\La,\La'}(\gra,\grb)&\coloneqq
		\iota^{\La\cup\La'}f(\state')
		-\iota^{\La}f(\state') -\iota^{\La'}f(\state')\\
		&= \iota^{\La\cup\La''}f(\state'')
		-\iota^{\La}f(\state'') -\iota^{\La''}f(\state'')
		=h_f^{\La,\La''}(\gra,\grb)
	\end{align*}
	as desired.
\end{proof}

Let $\Val$ be the commutative monoid defined in Definition \ref{def: U}.
We will next construct a well-defined pairing $h_f\colon\Val\times\Val\rightarrow\bbR$.
We first consider some conditions on the locale.

\begin{definition}\label{def: we}
	 We say that a locale $X$ is \textit{weakly transferable}, 
	if for any ball $\Ba\subset X$, 
	the complement $X\setminus\Ba$ is a nonempty finite disjoint union of connected infinite graphs.
	In particular, if $X\setminus\Ba$ is a connected infinite graph for any ball $\Ba\subset X$, 
	then we say that $X$ is \textit{strongly transferable}.
\end{definition}

Immediate from the definition, weakly transferable locales are infinite locales.

\begin{remark}
	Consider the Euclidean lattice given in Example \ref{example: locale} (1).
	Then $\bbZ^d=(\bbZ^d,\bbE^d)$ for $d>1$ is  strongly transferable.
	The Euclidean lattice $\bbZ=(\bbZ,\bbE)$ is weakly transferable, 
	but not strongly transferable.
\end{remark}

For any $R>0$, consider the set
\[
	\FC_R\coloneqq\{(\La,\La')\mid \text{$\La,\La'$: finite nonempty
	connected $\subset X$}, d_X(\La,\La')>R\}.
\]
For any $(\La_1,\La'_1),(\La_2,\La_2')\in\FC_R$, we denote $(\La_1,\La'_1)\leftrightarrow(\La_2,\La_2')$
if $\La_1=\La_2$ and  $\La'_1$ and $\La_2'$ are in the same connected component of 
$X\setminus  B(\La_1,R)=X\setminus  B(\La_2,R)$, or 
$\La'_1=\La'_2$ and  $\La_1$ and $\La_2$ are in the same connected component of 
$X\setminus  B(\La'_1,R)=X\setminus  B(\La'_2,R)$.
Note that we have $(\La_1,\La'_1)\leftrightarrow(\La_2,\La_2')$ if and only if
$(\La'_1,\La_1)\leftrightarrow(\La'_2,\La_2)$.

\begin{definition}\label{def: AR}
	We denote by $\sA_R$ the subset of $\FC_R$ consisting of pairs $(\La,\La')$ 
	such that at least one of $\La$ and $\La'$ are balls. 
\end{definition}

From now until the end of this subsection,
we assume that $X$ is strongly transferable.
In this case, we will prove that the pairing 
$h_f^{\La,\La'}\colon\Val_{|\La|}\times\Val_{|\La'|}\rightarrow\bbR$ of Proposition \ref{prop: h} 
associated with $f$ is independent of the choice of $(\La,\La')\in\sA_R$,
and defines a well-defined pairing
$h_f\colon\Val\times\Val\rightarrow\bbR$
satisfying a certain cocycle condition.
We will address the weakly transferable case in \S \ref{sec: weakly}.

\begin{lemma}\label{lem: independence two}
	Suppose $X$ is strongly transferable.
	Then for any  $(\La_1,\La'_1), (\La_2,\La'_2)\in\sA_R$, we have
	\[
			h_f^{\La_1,\La'_1}(\gra,\grb)=h_f^{\La_2,\La'_2}(\gra,\grb)
	\]
	for any $\gra\in\Val_{|\La_1|}\cap\Val_{|\La_2|}$ and $\grb\in\Val_{|\La'_1|}\cap\Val_{|\La'_2|}$.
\end{lemma}

\begin{proof}
	We first consider the case when $\La_1,\La'_1, \La_2,\La'_2$ are all balls.
	Note that for a ball $\Ba$, the $R$-thickening $ B(\Ba,R)$ is also a ball.
	Let $\Ba$ be a ball such that $|\Ba|\geq |\La_i|$ and $d_X(\La_i,\Ba)>R$ for $i=1,2$.
	Then since $X$ is strongly transferable, the sets $X\setminus  B(\La_1,R)$, $X\setminus  B(\La_2,R)$
	and $X\setminus  B(\Ba,R)$
	are all sublocales, hence we have
	\[
		(\La_1,\La'_1)\leftrightarrow(\La_1,B)\leftrightarrow(\La_2,B)\leftrightarrow(\La_2,\La'_2).
	\]
	Our assertion follows from Lemma \ref{lem: independence}.
	Now, consider the case for general $(\La_i,\La'_i)\in\sA_R$. 
	By replacing the component which is not a ball with a ball
	of sufficiently large size in the complement, 
	we see that there exists a pair of balls $(\Ba_i,\Ba'_i)\in\sA_R$
	such that $|\Ba_i|\geq|\La_i|$, $|\Ba'_i|\geq|\La'_i|$ for $i=1,2$
	satisfying $(\La_i,\La'_i)\leftrightarrow(\Ba_i,\Ba'_i)$.
	Again by Lemma \ref{lem: independence}, we see that
	$h_f^{\La_i,\La'_i}(\gra,\grb)=h_f^{\Ba_i,\Ba'_i}(\gra,\grb)$
	for any $\gra\in\Val_{|\La_i|}$ and $\grb\in\Val_{|\La'_i|}$.
	Our assertion now follows from our assertion for balls.
\end{proof}

\begin{proposition}\label{prop: h_f}
	For the system $(X,S,\phi$),
	assume that $X$ is strongly transferable and that the interaction $\phi$
	is irreducibly quantified.
	Let $f\in C(\State_*)$ be a function such that $\partial f\in C^1_R(\State)$
	for some constant $R>0$.
	Then there exists a pairing
	\[
		h_f\colon\Val\times\Val\rightarrow\bbR
	\]
	such that for any $(\La,\La')\in\sA_R$, we have
	\begin{equation}\label{eq: partition}
		\iota^{\La\cup\La'} f(\state)-\iota^{\La} f(\state)-\iota^{\La'}f(\state)
		=h_f(\xi_\La(\state),\xi_{\La'}(\state))
	\end{equation}
	for any $\state\in\State$.
	Moreover, the pairing $h_f$ is symmetric, in other words 
	$h_f(\gra,\grb)=h_f(\grb,\gra)$
	for any $\gra,\grb\in\Val$,
	and satisfies the cocycle condition
	\begin{equation}\label{eq: cocycle}
		h_f(\gra,\grb)+h_f(\gra+\grb,\grc)=h_f(\grb,\grc)+h_f(\gra,\grb+\grc)	
	\end{equation}
	for any $\gra,\grb,\grc\in\Val$.
\end{proposition}

\begin{proof}
	Note that for any $\gra,\grb\in\Val$, there exists $k\in\bbN$ such that 
	$\gra,\grb\in\Val_k$.
	By Lemma \ref{lem: independence two}, 
	the pairing $h_f^{\La,\La'}\colon\Val_{|\La|}\times\Val_{|\La'|}\rightarrow\bbR$ 
	of Proposition \ref{prop: h} associated with $f$ 
	is independent of the choice of $(\La,\La')\in\sA_R$,
	hence we have a pairing $h_f\colon\Val\times\Val\rightarrow\bbR$.
	The equality of \eqref{eq: partition} follows from Proposition \ref{prop: h}.
	In addition, for $(\La,\La')\in\sA_R$, we have $(\La',\La)\in\sA_R$,
	which by \eqref{eq:  partition} implies that $h_f$ is symmetric.
	In order to prove the cocycle condition,
	fix an arbitrary $\gra,\grb,\grc\in\Val$, and let $k\in\bbN$ be  
	such that $\gra,\grb,\grc\in\Val_k$.
	Let $\Ba_1$ be any ball with $|\Ba_1|>k$.  
	Since $X$ is strongly transferable, $X\setminus  B(\Ba_1,R)$ is a locale.
 	Take any ball $\Ba_2$  with $|\Ba_2|>k$ in $X\setminus  B(\Ba_1,R)$.
	Let $\Ba$ be a ball sufficiently large containing both $\Ba_1$ and $\Ba_2$.
	Then again since $X$ is strongly transferable, $X\setminus B(\Ba,R)$
	is a locale.	
	Take any ball $\Ba_3$  with $|\Ba_3|>k$ in $X\setminus  B(\Ba,R)$.
	Then the balls $\Ba_1,\Ba_2,\Ba_3\subset X$ satisfy
	$|\Ba_i|\geq k$ and $d_X(\Ba_i,\Ba_j)>R$ for $i\neq j$.
	Since $\gra,\grb,\grc\in\Val_k$,
	there exists $\state\in\State$ such that $\xi^{\Ba_1}(\state)=\gra$, $\xi^{\Ba_2}(\state)=\grb$,
	$\xi^{\Ba_3}(\state)=\grc$ and is at base state outside $\Ba_1\cup  \Ba_2\cup  \Ba_3$.
	We let $\La \subset X\setminus  B(\Ba_3,R)$ be a finite connected subset of $X$
	such that $\Ba_1\cup  \Ba_2\subset \La$, and let
	$\La' \subset X\setminus  B(\Ba_1,R)$ be a finite connected subset of $X$ such that 
	$\Ba_2\cup  \Ba_3\subset \La'$.
	For example, we may take $\La$ to be the union of $\Ba_1,\Ba_2$ and points on a path
	in $X\setminus  B(\Ba_3,R)$ from the center of $\Ba_1$ to the center of $\Ba_2$,
	and similarly for $\La'$.
	Note by construction, we have $(\Ba_1,\La')$, $(\La,\Ba_3)\in\sA_R$.  Then 
	by Lemma \ref{lem: independence} and the definition of $h$, we have
	\begin{align*}
		h_f(\gra,\grb)&=h_f^{\Ba_1,\Ba_2}(\gra,\grb)=\iota^{\Ba_1\cup  \Ba_2}f(\state)
		-\iota^{\Ba_1}f(\state)-\iota^{\Ba_2}f(\state)\\
		h_f(\gra+\grb,\grc)&=h_f^{\La,\Ba_3}(\gra+\grb,\grc)
		=\iota^{\La\cup  \Ba_3}f(\state)
		-\iota^{\La}f(\state)-\iota^{\Ba_3}f(\state)\\
		h_f(\grb,\grc)&=h_f^{\Ba_2,\Ba_3}(\gra,\grb)=\iota^{\Ba_2\cup  \Ba_3}f(\state)
		-\iota^{\Ba_2}f(\state)-\iota^{\Ba_3}f(\state)\\
		h_f(\gra,\grb+\grc)&=h_f^{\Ba_1,\La'}(\gra,\grb+\grc)=\iota^{\Ba_1\cup\La'}f(\state)
		-\iota^{\Ba_1}f(\state)-\iota^{\La'}f(\state).
	\end{align*}
	Since $\state$ is at base state outside $\Ba_1\cup  \Ba_2\cup  \Ba_3$, we have
	\begin{align*}
		\iota^{\La\cup  \Ba_3}f(\state)&=\iota^{\Ba_1\cup\La'}
		f(\state)=\iota^{\Ba_1\cup  \Ba_2\cup  \Ba_3}f(\state), &
		\iota^{\La}f(\state)&=\iota^{\Ba_1\cup  \Ba_2}f(\state), &
		\iota^{\La'}f(\state)&=\iota^{\Ba_2\cup  \Ba_3}f(\state),
	\end{align*}
	which proves our assertion.
\end{proof}

%
\subsection{Criterion for Uniformity}\label{subsec: criterion}
%

In this subsection, we prove Proposition \ref{prop: important}, which gives a criterion
for a function $f\in \CS$ to be uniform when $X$ is strongly transferable.  
The weakly transferable case will be addressed in Proposition \ref{prop: important wr}.
This result will play an essential role in the proof of Theorem \ref{thm: 1}.
As in \S\ref{subsec: pairing}, we assume here that the interaction is irreducibly quantified.

\begin{proposition}\label{prop: important}
	For the system $(X,S,\phi)$, assume that $X$ is strongly transferable and that the interaction
	$\phi$
	is irreducibly quantified.
	Let $f\in C(\State_*)$ be a function such that $\partial f\in C^1_R(\State)$
	for some $R>0$, and	
	let $h_f\colon\Val\times\Val\rightarrow\bbR$ be the pairing given in Proposition \ref{prop: h_f}.
	If $h_f\equiv 0$, then we have  $f\in C^0_\unif(\State)$.
\end{proposition}

We first prove a lemma to characterize the functions in $\CuS$,
which does not require $X$ to be strongly transferable.
For any ball $ B(x,R)$, we denote by $\cB^*(x,R)\coloneqq   B(x,R)\setminus\{x\}$
the punctured ball.

\begin{lemma}\label{lemma: check}
	Suppose $f\in\CS$.  Then $f\in\CuS$ if and only if there exists $R>0$
	such that for any finite $\La\subset X$ and $x\in \La$, we have
	\begin{equation}\label{eq: formula}
		\iota^{\La} f-\iota^{\La\setminus\{x\}} f=\iota^{\La\cap  B(x,R)} f
		- \iota^{\La\cap \cB^*(x,R)} f.
	\end{equation}
\end{lemma}

\begin{proof}
	First observe that for any $f\in \CS$, we have
	\[
		\iota^{\La} f=\sum_{\La'\subset\La}f_{\La'}	
		=\sum_{\La'\subset\La\setminus\{x\}}f_{\La'}	
		+\sum_{\La'\subset\La, x\in\La'}f_{\La'}	
		=\iota^{\La\setminus\{x\}}f
		+\sum_{\La'\subset\La, x\in\La'}f_{\La'}	
	\]
	where $f_\La$ is the local function with exact support in the canonical expansion
	\eqref{eq: expansion}.  
	
	Suppose $f\in\CuS$.  Then there exists $R>0$ such that $f_{\La'}\equiv0$ for
	any finite $\La'\subset X$ safisfying $\diam{\La'}>R$.
	Then for any finite $\La\subset X$ and $x\in\La$, 
	we have
	\begin{align*}
		\sum_{\La'\subset\La,x\in\La'}f_{\La'}&=\sum_{\La'\subset\La\cap  B(x,R),x\in\La'}f_{\La'}
		=\sum_{\La'\subset\La\cap  B(x,R)}f_{\La'}
		-\sum_{\La'\subset\La\cap \cB^*(x,R)}f_{\La'}\\
		&=\iota^{\La\cap  B(x,R)}f-\iota^{\La\cap \cB^*(x,R)} f.
	\end{align*}
	This gives \eqref{eq: formula}.
	Next, we prove the converse.
	By the same argument, if \eqref{eq: formula} holds, 
	then we have
	\[
		\sum_{\La'\subset\La, x\in\La'}f_{\La'}	
		=\sum_{\La'\subset\La\cap  B(x,R), x\in\La'}f_{\La'}
	\]
	for any finite $\La\subset X$ and $x\in\La$.
	Suppose there exists finite $\La\subset X$ such that 
	$\diam{\La}>R$ and $f_\La\not\equiv0$.
	By iteratively replacing $\La'\subset\La$ by $\La$ if necessary,
	we may assume that for any $\La'\subsetneq\La$,
	we have $f_{\La'}\equiv0$ or  $\diam{\La'}\leq R$.
	However, for this $\La$ and any $x\in\La$ such that there
	exists $y\in\La$ with $d_X(x,y)>R$, we have
	\[
		\sum_{\La'\subset\La,x\in\La'}f_{\La'} =f_\La+\sum_{\La'\subsetneq\La,x\in\La'}f_{\La'}
		=f_\La+\sum_{\La'\subset\La\cap  B(x,R),x\in\La'}f_{\La'}.
	\]
	Comparing with the previous equality, we see that $f_\La\equiv0$.
	This contradicts our hypothesis that $f_\La\not\equiv0$, hence our assertion is proved.
\end{proof}

Before the proof of Proposition \ref{prop: important},
we prepare an additional lemma.
Here, we will assume that $X$ is strongly transferable.
The weakly transferable case will be treated in \S \ref{sec: weakly}.

\begin{lemma}\label{lem: final easy}
	For the system $(X,S,\phi)$,
	assume that $X$ is strongly transferable and that the interaction $\phi$ is irreducibly quantified.
	Suppose $f\in\CS$ and $\partial f\in C^1_R(\State)$ for some $R>0$, and 
	let $h_f\colon\Val\times\Val\rightarrow\bbR$ be the pairing given in Proposition \ref{prop: h_f}.
	If $h_f\equiv 0$, 
	then for any $x\in X$ and finite $\La\subset X$ such that $d_X(x,\La)>R$, 
	we have
	\[
		\iota^{\La\cup  B(x,R)}f-\iota^{\La\cup \cB^*(x,R)}f
		=\iota^{ B(x,R)}f-\iota^{\cB^*(x,R)}f.
	\]
\end{lemma}

\begin{proof}
	Note that by the definition of the pairing $h_f$ given in \eqref{eq: partition}, 
	we have
	\begin{equation}\label{eq: distribution two}
		\iota^{\La\cup\La'}f=\iota^{\La} f+\iota^{\La'}f
	\end{equation}
	for any $(\La,\La')\in\sA_R$.
	If $\La$ is a ball, then $(\La, B(x,R))$ and $(\La,\cB^*(x,R))$ are both pairs in $\sA_R$,
	hence our assertion immediately follows by applying \eqref{eq: distribution two}
	to the left hand side.
	Consider a general finite $\La\subset X$ such that $d_X(x,\La)>R$.
	It is sufficient to prove that
	\begin{equation}\label{eq: goal}
		f(\state|_{\La\cup  B(x,R)})-f(\state|_{\La\cup \cB^*(x,R)})
		=f(\state|_{ B(x,R)})-f(\state|_{\cB^*(x,R)})
	\end{equation}
	for any $\state\in\State_*$ such that $\Supp(\state)\subset\La\cup  B(x,R)$.
	Since $X$ is strongly transferable, $X\setminus  B(x,2R)$ is a locale.
	Let $\Ba\subset X\setminus  B(x,2R)$ be a ball whose cardinality is greater than that of 
	$\La$ and satisfying $d_X(\La,\Ba)>R$.  Such a set  $\Ba$ exists since $X\setminus  B(x,2R)$ 
	is infinite.
	Now choose a $\state'\in\State_*$ such that $\state'$ coincides with $\state$ outside 
	$\La\cup\Ba$, is at base state on $\La$, and $\bsxi_{\Ba}(\state')=\bsxi_{\La}(\state)$.
	This implies that $\bsxi_{X}(\state')=\bsxi_{X}(\state)$.
	Since $X$ is strongly transferable, the complement $Y\coloneqq X\setminus  B(x,R)$ is a locale.
	From our condition that the interaction is irreducibly quantified,
	by Lemma \ref{lem: difference}
	applied to $Y=X\setminus  B(x,R)$ and $\wt Y= X\setminus\{x\}$, we have
	\[
		f(\state')-f(\state)
		=f(\state'|_{X\setminus\{x\}})-f(\state|_{X\setminus\{x\}}).
	\]
	Noting that $\state'=\state'|_{\Ba\cup  B(x,R)}$ and $\state=\state|_{\La\cup  B(x,R)}$,
	we see that
	\begin{equation}\label{eq: progress}
		f(\state'|_{\Ba\cup  B(x,R)})-f(\state|_{\La\cup  B(x,R)})
		=f(\state'|_{\Ba\cup \cB^*(x,R)})-f(\state|_{\La\cup \cB^*(x,R)}).
	\end{equation}
	Since $\Ba\subset X\setminus B(x,2R)$, we have $d_X(\Ba, B(x,R))>R$, hence 	
	$(\Ba,  B(x,R))$ and $(\Ba,\cB^*(x,R))$ are both pairs in $\sA_R$.
	Our condition \eqref{eq: distribution two} on $f$ 
	applied to this pair implies that
	\begin{align*}
		f(\state'|_{\Ba\cup  B(x,R)})&=f(\state'|_{\Ba})+f(\state'|_{ B(x,R)}),&
		f(\state'|_{\Ba\cup \cB^*(x,R)})&=f(\state'|_{\Ba})+f(\state'|_{\cB^*(x,R)}).
	\end{align*}
	Hence \eqref{eq: progress} gives
	\[
		f(\state'|_{ B(x,R)})-f(\state|_{\La\cup  B(x,R)})
		=f(\state'|_{\cB^*(x,R)})-f(\state|_{\La\cup \cB^*(x,R)}).
	\]
	Our assertion \eqref{eq: goal} follows from the fact that
	$\state$ and $\state'$ coincides on $ B(x,R)$.
\end{proof}

We may now prove Proposition \ref{prop: important}.

\begin{proof}[Proof of Proposition \ref{prop: important}]
	Suppose $f\in \CS$ and $\partial f\in C^1_R(\State)$, and that 
	$
		h_f\equiv 0
	$
	for the pairing $h_f$ of Proposition \ref{prop: h_f}.
	By Lemma \ref{lemma: check}, it is sufficient to check that for any fixed finite subset
	$\La\subset X$ and $x\in\La$, we have
	\[
		\iota^{\La} f-\iota^{\La\setminus\{x\}} f=\iota^{\La\cap  B(x,R)} f
		- \iota^{\La\cap \cB^*(x,R)} f.
	\]
	We let $\La'\coloneqq\La\setminus (\La\cap  B(x,R))$.  Then the above equation may be written as
	\begin{equation}\label{eq: orange}
		\iota^{\La'\cup(\La\cap  B(x,R))} f
	-\iota^{\La'\cup(\La\cap \cB^*(x,R))} f
	=\iota^{\La\cap  B(x,R)} f
		- \iota^{\La\cap \cB^*(x,R)} f.
	\end{equation}
	In order to prove this, it is sufficient to prove that
	\begin{equation}\label{eq: lemon}
		\iota^{\La'\cup  B(x,R)} f-\iota^{\La'\cup \cB^*(x,R)} f
		=\iota^{ B(x,R)} f-\iota^{\cB^*(x,R)} f,
	\end{equation}
	since \eqref{eq: orange} may be obtained by applying 
	$\iota^{\La'\cup(\La\cap  B(x,R))}$ to both sides of \eqref{eq: lemon}.
	Hence it is sufficient to prove that for any finite $\La\subset X$ such that 
	$\La\cap  B(x,R)=\emptyset$, we have
	\[
		\iota^{\La\cup  B(x,R)} f-\iota^{\La\cup \cB^*(x,R)} f
		=\iota^{ B(x,R)} f-\iota^{\cB^*(x,R)} f,
	\]
	which is precisely Lemma \ref{lem: final easy}.
\end{proof}

%
%
%
\section{Pairing and Criterion in the Weakly Transferable Case}\label{sec: weakly}
%
%
%

In this section, we will extend Proposition \ref{prop: h_f} concerning
the construction of the pairing $h_f$ and Proposition \ref{prop: important} 
concerning the criterion for uniformity to the case when the locale $X$ is \textit{weakly transferable}.
The reader interested only in the strongly transferable case may skip to \S\ref{sec: cohomology}.
We first start with the classification of objects in $\sB_R$ which generalizes
the set of pairs $\sA_R$ given in Definition \ref{def: AR}

%
\subsection{Equivalence Relation for Pairs}\label{subsec: equivalence}
%

Let $R>0$.
In this subsection, we investigate a certain equivalence relation
for certain pairs $(\Ba,\Ba')$ in a subset of $\FC_R$.  The result of this subsection
is concerned only with the underlying graph structure of the locale $X$.
We define a variant of $\sA_R$ of Definition \ref{def: AR} as follows.

\begin{definition}\label{def: B}
	For any $r\in\bbN$,
	we denote by $\sB^r_R$ the subset of $\FC_R$ consisting of 
	pairs of balls $(\Ba,\Ba')$ such that the radii of $\Ba$ and $\Ba'$ are at least $r$.
	Furthermore, we let $\sB_R\coloneqq\sB^0_R$.
\end{definition}

We define the relation $\sim_{r}$ to be the 
equivalence relation in $\sB^r_R$ generated by the relations
$(\Ba,\Ba')\leftrightarrow(\Ba,\Ba'')$
and $(\Ba',\Ba)\leftrightarrow(\Ba'',\Ba)$.
In this subsection, we first study the equivalence relation $\sim_{r}$ on $\sB^r_R$.
In particular, we have the following.

\begin{proposition}\label{prop: B}
	For any $r\in\bbN$, there are at most two equivalence classes with respect to 
	the equivalence relation $\sim_{r}$ in $\sB^r_R$.  Moreover, there is only one equivalence class
	if $(\Ba,\Ba')\sim_{r}(\Ba',\Ba)$ for some $(\Ba,\Ba')\in\sB^r_R$.
\end{proposition}

The following observation will be used throughout this section.

\begin{remark}
	Let $X$ be a weakly transferable locale, and let $\Ba\subset X$ be a ball.
	\begin{enumerate}
		\item For any finite set $\La\subset X$ and $r,R>0$, there
		exists a ball $B'$ of radius at least $r$ such that
		$d_X(\La,\Ba')>R$ and $\Ba'\subset (X\setminus\Ba)$.
		\item For any finite connected set $\La\subset X\setminus\Ba$
		 and $r,R>0$,  there exists a ball $\Ba'$ of
		 radius at least $r$ such that
		$d_X(\Ba,\Ba')>R$ and $\Ba'$ and $\La$ are in the same connected 
		 component of $X\setminus \Ba$.
	\end{enumerate}
\end{remark}

In order to  prove Proposition \ref{prop: B},
we first prove Lemma \ref{lem: important} and Lemma \ref{lem: either}.

\begin{lemma}\label{lem: important}
	For any $r\in\bbN$, suppose there exists three balls $\Ba_1,\Ba_2,\Ba_3$ 
	of radii at least $r$ in $X$ such that $d_X(\Ba_i,\Ba_j)>2R$ for $i\neq j$.
	Then either $(\Ba_1,\Ba_2)\leftrightarrow(\Ba_1,\Ba_3)$ or
	$(\Ba_1,\Ba_2)\leftrightarrow(\Ba_3,\Ba_2)$ holds.
\end{lemma}

\begin{proof} 
	Suppose $\Ba_2$ and $\Ba_3$ are \textit{not} in the same connected component of 
	$X\setminus  B(\Ba_1,R)$.
	Since $X$ is connected, there exists a path $\vec p=(e^1,\ldots,e^N)$ such that 
	$o(\vec\gamma)\in\Ba_2$ and $t(\vec\gamma)\in\Ba_3$.
	We take the shortest among such paths, and let 
	$p^i=t(e^i)$ for $i=1,\ldots,N$.  
	Since $\Ba_2$ and $\Ba_3$ are not in the same connected 
	component of $X\setminus  B(\Ba_1,R)$, there exists $\indj$ such that $p^\indj\in  B(\Ba_1,R)$.
	Since $d_X(\Ba_1,\Ba_2)>2R$, we have $p^\indj\not\in  B(\Ba_2,R)$, hence $\indj>R$.
	Then for any $i\geq\indj$, we have $p^i\in X\setminus  B(\Ba_2,R)$, since we have taken $\vec p$ 
	to be the shortest path from a vertex in $\Ba_2$ to a vertex in $\Ba_3$.
	Then $\vec p'=(p^\indj,p^{\indj+1},\ldots,p^N)$ gives a path 
	in $X\setminus  B(\Ba_2,R)$ from $p^\indj$ to a vertex in $\Ba_3$.
	Since $p^\indj\in  B(\Ba_1,R)$, there exists a path in $X\setminus  B(\Ba_2,R)$
	from an element in $\Ba_1$ to the vertex $p^\indj$.
	The combination of this path with $\vec p'$ gives a path in $X\setminus  B(\Ba_2,R)$
	from a vertex in $\Ba_1$ to a vertex in $\Ba_3$, hence $\Ba_1$ and $\Ba_3$ 
	are in the same connected component of $X\setminus  B(\Ba_2,R)$ as desired.
\end{proof}

\begin{lemma}\label{lem: either}
	For any $r\in\bbN$, suppose there exists three balls $\Ba_1,\Ba_2,\Ba_3$ 
	of radii at least $r$ in $X$ such that $d_X(\Ba_i,\Ba_j)>2R$ for $i\neq j$.
	Then $(\Ba_1,\Ba_2)\sim_{r}(\Ba_1,\Ba_3)$ or $(\Ba_1,\Ba_2)\sim_{r} (\Ba_3,\Ba_1)$ holds.
\end{lemma}

\begin{proof}
	By Lemma \ref{lem: important}, at least one of 
	$(\Ba_1,\Ba_2)\sim_{r}(\Ba_1,\Ba_3)$ or $(\Ba_1,\Ba_2)\sim_{r}(\Ba_3,\Ba_2)$ holds.
	Also, by reversing the roles of $\Ba_2$ and $\Ba_3$, we see that at least one of
	$(\Ba_1,\Ba_3)\sim_{r}(\Ba_1,\Ba_2)$ or $(\Ba_1,\Ba_3)\sim_{r}(\Ba_2,\Ba_3)$ holds.
	Hence if $(\Ba_1,\Ba_2)\not\sim_{r}(\Ba_1,\Ba_3)$, then we have 
	$(\Ba_1,\Ba_2)\sim_{r}(\Ba_3,\Ba_2)$ and $(\Ba_1,\Ba_3)\sim_{r}(\Ba_2,\Ba_3)$,
	where the last equivalence implies that $(\Ba_3,\Ba_2)\sim_{r}(\Ba_3,\Ba_1)$
	by symmetry.  This implies that $(\Ba_1,\Ba_2)\sim_{r}(\Ba_3,\Ba_2)\sim_{r}(\Ba_3,\Ba_1)$
	as desired.
\end{proof}

We may now prove Proposition \ref{prop: B}.

\begin{proof}[Proof of Proposition \ref{prop: B}]
	We fix a pair $(\Ba_1,\Ba'_1)\in\sB^r_R$.	
	We let $\Ba''_1\subset X\setminus  B(\Ba_1,2R)$ be a ball of radius at least $r$
	which is in the same connected component as $B'_1$ in $X\setminus  B(B_1,R)$.
	Then we have $(\Ba_1,\Ba'_1)\sim_{r} (\Ba_1,\Ba''_1)$.
	Consider any $(\Ba_2,\Ba'_2)\in\sB^r_R$.
	We let $\Ba$ be a ball sufficiently large containing $B_1\cup B''_1$.
	Then $ B(\Ba,2R)$ is also a ball, and $X\setminus  B(\Ba,2R)$
	decomposes into a finite sum of locales.
	We let $\Ba_3 \subset X\setminus  B(\Ba,2R)$ be a close ball of radius at least $r$
	which is in the same connected component as $\Ba_2$ in $X\setminus  B(\Ba'_2,R)$.
	Then by definition, we have $(\Ba_2,\Ba'_2)\sim_{r}(\Ba_3,\Ba'_2)$.
	Finally, we let $\Ba'$ be the ball sufficiently large containing $B_1\cup B''_1\cup B_3$,
	and we let $\Ba''_3 \subset X\setminus  B(\Ba',2R)$ be a ball of radius at least $r$
	which is in the same connected component of $X\setminus  B(B_3,R)$ as $B'_2$.
	Then by construction, we have $(B_2,B'_2)\sim_{r} (B_3,B'_2)\sim_{r} (B_3,B''_3)$.

	By our construction of $B''_1$ and $B''_3$ and Lemma \ref{lem: either},
	either $(B_1,B''_1)\sim_{r}(B_1,B''_3)$ or $(B_1,B''_1)\sim_{r}(B''_3,B_1)$ holds.
	Furthermore, either $(B_1,B''_3)\sim_{r}(B_3,B''_3)$ or $(B''_3,B_1)\sim_{r}(B''_3,B_3)$.
	Combining the two, noting that $(B_2,B'_2)\sim_{r}  (B_3,B''_3)$,
	we obtain our assertion.
\end{proof}

As a consequence of Proposition \ref{prop: B}, we have the following.

\begin{proposition}\label{prop: cases}
	Let $X=(X,E)$ be a weakly transferable locale, and let $R>0$. 
	For any integer $r\in\bbN$, 
	we choose an equivalence $\sC^r_R$ in $\sB^r_R$ with respect
	to the equivalence relation $\sim_{r}$.  Then one of the following holds.
	\begin{enumerate}
		\item For any $r\in\bbN$, we have $\sB^r_R=\sC^r_R$.
		\item There exists $r_0\in\bbN$ such that for any $r<r_0$, we have $\sB^r_R=\sC^r_R$,
		and for $r\geq r_0$, we have
		\begin{equation}\label{eq: expansion B}
			\sB^r_R=\sC^r_R\cup\ol\sC^r_R,
		\end{equation}
		where 
		$\ol\sC^r_R\coloneqq\{ (\La ',\La)\mid (\La,\La')\in\sC^r_R\}$.
	\end{enumerate}
	Moreover, we may choose $\sC^r_R$ so that
	$\sC^r_R=\sC^{r_0}_R \cap\sB^r_R$ for any $r\geq r_0$,
	where we let $r_0=0$ when (1) holds.
\end{proposition}

\begin{proof}
	Suppose (1) does not hold.
	Then by Proposition \ref{prop: B}, there exists $r\in\bbN$ such that for any $(\Ba,\Ba')\in\sB^r_R$,
	we have $(\Ba,\Ba')\not\sim_r(\Ba',\Ba)$.
	Then for any $r'\geq r$ and $(\Ba,\Ba')\in\sB^{r'}_R$,  if $(\Ba,\Ba')\sim_{k'}(\Ba',\Ba)$
	then this would imply that $(\Ba,\Ba')\sim_r(\Ba',\Ba)$.
	Hence our assumption implies that $(\Ba,\Ba')\not\sim_{r'}(\Ba',\Ba)$.
	We take $r_0$ to be the minimum of  such $r$.
	Then \eqref{eq: expansion B} follows from Proposition \ref{prop: B}.
	
	Note that for any $r\geq r_0$ and $(\Ba_1,\Ba'_1)$ and $(\Ba_2,\Ba'_2)$ in $\sB^r_R$,
	we have $(\Ba_1,\Ba'_1)\sim_{r_0}(\Ba_2,\Ba'_2)$ if and only if 
	$(\Ba_1,\Ba'_1)\sim_{r}(\Ba_2,\Ba'_2)$.
	This may be proved as follows.
	It is immediate from the definition that if $\sim_{r}$ holds, then $\sim_{r_0}$ holds.
	On the other hand, if $(\Ba_1,\Ba'_1)\sim_{r_0}(\Ba_2,\Ba'_2)$ and if 
	$(\Ba_1,\Ba'_1)\not\sim_{r}(\Ba_2,\Ba'_2)$,
	then we would have $(\Ba_1,\Ba'_1)\sim_{r}(\Ba'_2,\Ba_2)$, which would imply that
	$(\Ba_1,\Ba'_1)\sim_{r_0}(\Ba'_2,\Ba_2)$.  Then 
	we would have $(\Ba_2,\Ba'_2)\sim_{r_0}(\Ba'_2,\Ba_2)$,
	contradicting our choice of $r_0$.
	
	If we choose an equivalence class $\sC^{r_0}_R$ of $\sB^{r_0}_R$ with respect
	to the equivalence $\sim_{r_0}$,  then we may take $\sC^r_R\coloneqq
	\sC^{r_0}_R\cap\sB^r_R$ for any $r\geq r_0$, 
	which gives an equivalence class of $\sB^r_R$ with respect
	to the equivalence $\sim_{r}$ satisfying the condition of our assertion.
\end{proof}

\begin{definition}\label{def: cases}
	Let $X=(X,E)$ be a weakly transferable locale, and let $R>0$. 
	If (1) of Proposition \ref{prop: cases} holds,  then we say that $\sB_R$ has a 
	\textit{unique class}, and we  define the integer $r_0$ to be \textit{one}.
	If (2) of Proposition \ref{prop: cases} holds,  then we say that $\sB_R$ is
	\textit{split},
	and we define $r_0$ to be the minimum integer satisfying \eqref{eq: expansion B}.
	Moreover, we will fix an equivalence class $\sC^r_R$ of $\sB^r_R$ with respect
	to the equivalence relation $\sim_{r}$ so that $\sC^r_R= \sC^{r_0}_R\cap\sB^r_R$ for 
	$r\geq r_0$.
\end{definition}

%
\subsection{Pairing for Functions in the Weakly Transferable Case}\label{subsec: pairing wr}
%

In this subsection, we will construct and prove the cocycle condition for
the pairing $h_f\colon\Val\times\Val\rightarrow\bbR$.
We let $X$ be a weakly transferable locale.  We let $\sB^r_R$ as in Definition \ref{def: B},
and we fix an equivalence class $\sC^r_R$ of $\sB^r_R$ with respect to the equivalence
relation $\sim_{r}$ as in Definition \ref{def: cases}.
We define the pairing $h_f$ as follows.

\begin{definition}\label{def: h case2}
	Let $f\in\CS$ such that 
	$\partial f \in C^1_R(\State)$ for some $R>0$. 
	We define the pairing $h_f\colon\Val\times\Val\rightarrow\bbR$ as follows.
	For any $\gra,\grb\in\Val$, we let $k\geq r_0$ such that $\gra,\grb\in\Val_k$.
	By taking an arbitrary $(\Ba,\Ba')\in\sC^k_R$, we let 
	\[
		h_f(\gra,\grb)\coloneqq h_f^{\Ba,\Ba'}(\gra,\grb),
	\]
	where $h_f^{\Ba,\Ba'}$ is the pairing $h_f^{\La,\La'}$ defined in Proposition \ref{prop: h}
	for $(\La,\La')=(\Ba,\Ba')$.
	Note that we have $|\Ba|\geq r(\Ba)$ and $|\Ba'|\geq r(\Ba')$, hence $|\Ba|, |\Ba'|\geq k$.
\end{definition}

For the remainder of this subsection, 
we let $R>0$, and we fix a $f\in\CS$ such that $\partial f\in C^1_R(\State)$.
If $(\Ba_1,\Ba'_1),(\Ba_2,\Ba'_2)\in\sC^k_R$,
then we have $(\Ba_1,\Ba'_1)\sim_{k}(\Ba_2,\Ba'_2)$, hence by 
Lemma \ref{lem: independence}, we see that
\[
	h_f^{\Ba_1,\Ba'_1}(\gra,\grb)=h_f^{\Ba_2,\Ba'_2}(\gra,\grb)
\]
for any $\gra,\grb\in\Val_k$.
Since $\sC^k_R=\sC^{r_0}_R\cap\sB^k_R$,
this shows that $h_f$ is independent of the choices of $k\geq r_0$ 
satisfying $\gra,\grb\in\Val_k$ and the pair $(\Ba,\Ba')\in\sC^k_R$.

When $X$ is weakly transferable, the pairing $h_f$ may not be symmetric.

\begin{lemma}\label{lem: commutative}
	If $\sB_R$ has a unique class, then the pairing $h_f$ is symmetric.
\end{lemma}

\begin{proof}
	If $\sB_R$ has a unique class, then $(\Ba,\Ba')\in\sC^k_R$ implies that 
	$(\Ba',\Ba)\in\sC^k_R$.   This shows that for any $k\geq r_0$ and
	$\gra,\grb\in\Val_k$, we have
	\[
		h_f(\gra,\grb)=h_f^{\Ba,\Ba '}(\gra,\grb)=h_f^{\Ba',\Ba}(\grb,\gra)=h_f(\grb,\gra)
	\]
	as desired.
\end{proof}

\begin{remark}
	If $\sB_R$ is split, then $h_f$ may not necessarily be symmetric.
\end{remark}

The pairing $h_f$ may be calculated as follows.

\begin{proposition}\label{prop: def h two}
	For any $(\La,\La')\in\sA_R$ and for any $\gra\in\Val_{|\La|}$ and $\grb\in\Val_{|\La'|}$,
	we have
	\[
		h_f(\gra,\grb)=h_f^{\La,\La'}(\gra,\grb)  \quad\text{or}\quad   h_f(\grb,\gra)=h_f^{\La,\La'}(\gra,\grb).
	\]
	In particular, if $h_f$ is symmetric, then we have
	\[
		h_f(\gra,\grb)=h_f^{\La,\La'}(\gra,\grb).
	\]
\end{proposition}

\begin{proof}
	Fix an arbitrary $(\La,\La')\in\sA_R$ and $\gra\in\Val_{|\La|}$, $\grb\in\Val_{|\La'|}$.
	Let $k\geq  r_0$ such that
	$k\geq\max\{|\La|,|\La'|\}$. Assume without loss of generality that $\La$
	is a ball.  Since $X$ is weakly transferable, the connected component
	of $X\setminus  B(\La,R)$ containing $\La'$ is infinite, hence contains a closed 
	ball $\Ba'$ such that $r(\Ba')\geq k$. Then $(\La,\La')\leftrightarrow (\La,\Ba')$.
	Again, the connected component of $X\setminus   B(\Ba',R)$ containing $\La$
	is infinite, hence contains a ball $\Ba$ such that $r(\Ba)\geq k$.
	Then $(\La,\Ba')\leftrightarrow (\Ba,\Ba')$.  By construction, $(\Ba,\Ba')\in\sB^k_R$.
	If $(\Ba,\Ba')\in\sC^k_R$, then by Lemma \ref{lem: independence}, 
	noting that $|\Ba|,|\Ba'|\geq k$,
	we have
	\[
		h_f^{\La,\La'}(\gra,\grb)=h_f^{\Ba,\Ba'}(\gra,\grb)=h_f(\gra,\grb).
	\]
	If $(\Ba,\Ba')\not\in\sC^k_R$, then $(\Ba',\Ba)\in\sC^k_R$,
	hence we have
	\[
		h_f^{\La,\La'}(\gra,\grb)=h_f^{\Ba,\Ba'}(\gra,\grb)=h_f^{\Ba',\Ba}(\grb,\gra)=h_f(\grb,\gra)
	\]
	as desired.
\end{proof}

In order to prove the cocycle condition of $h_f$,
we will prepare the following lemma concerning the existence of a triplet of balls
$B_1,B_2,B_3$ sufficiently large and sufficiently apart satisfying the condition
$(B_1,B_2)\leftrightarrow(B_1,B_3)\leftrightarrow(B_2,B_3)$.  
Note that if $X$ is strongly transferable, then this condition is automatically satisfied
if $d_X(B_i,B_j)>R$ for $i\neq  j$.

\begin{lemma}\label{lem: triple}
	For any $r\in\bbN$, there exists balls $\Ba_1,\Ba_2,\Ba_3$ 
	of radii at  least $r$
	in $X$ such that $d_X(\Ba_i,\Ba_j)>R$ for $i\neq j$ and
	\begin{equation}\label{eq: leftrightarrows}
		(\Ba_1,\Ba_2)\leftrightarrow(\Ba_1,\Ba_3)\leftrightarrow(\Ba_2,\Ba_3).
	\end{equation}
	Moreover, we may take $\Ba_1,\Ba_2,\Ba_3$ so that
	$(\Ba_1,\Ba_2)$, $(\Ba_1,\Ba_3)$, $(\Ba_2,\Ba_3)\in\sC^r_R$.
\end{lemma}

\begin{proof}
	Take $\Ba_1$ to be a ball in $X$ such that $r(\Ba_1)\geq r$.
	We then take $\Ba_2$ to be an arbitrary ball in $X\setminus  B(B_1,2R)$
	such that $r(\Ba_2)\geq r$.  Finally, we take a ball $\La$ containing $B_1$ and $B_2$,
	and let $\Ba_3$ be a ball in $X\setminus  B(\La,2R)$ with $r(\Ba_3)\geq r$.
	By construction, we have $d_X(B_i,B_j)>2R$ for $i\neq j$.
	By Lemma \ref{lem: either}, the following holds.
	\begin{enumerate}
		\item Either $(\Ba_1,\Ba_2)\leftrightarrow(B_1,B_3)$ or $(\Ba_1,\Ba_2)\leftrightarrow(B_3,B_2)$.
		\item Either $(\Ba_1,\Ba_3)\leftrightarrow(B_1,B_2)$ or $(\Ba_1,\Ba_3)\leftrightarrow(B_2,B_3)$.
		\item Either $(\Ba_2,\Ba_3)\leftrightarrow(B_2,B_1)$ or $(\Ba_2,\Ba_3)\leftrightarrow(B_1,B_3)$.
	\end{enumerate}
	From this observation, we see that at least one of the following holds.
	\begin{enumerate}
		\item[(a)] $(\Ba_1,\Ba_2)\leftrightarrow(B_1,B_3)$ and $(\Ba_1,\Ba_3)\leftrightarrow(B_2,B_3)$.
		\item[(b)] $(\Ba_1,\Ba_2)\leftrightarrow(B_3,B_2)$ and $(\Ba_1,\Ba_3)\leftrightarrow(B_2,B_3)$.
		\item[(c)] $(\Ba_1,\Ba_3)\leftrightarrow(B_1,B_2)$ and $(\Ba_2,\Ba_3)\leftrightarrow(B_2,B_1)$.
	\end{enumerate}
	In fact, if (a) does not hold, then either $(\Ba_1,\Ba_2)\not\leftrightarrow(B_1,B_3)$ or 
	$(\Ba_1,\Ba_3)\not\leftrightarrow(B_2,B_3)$ holds.
	If $(\Ba_1,\Ba_2)\not\leftrightarrow(B_1,B_3)$, then by (1), we have 
	$(\Ba_1,\Ba_2)\leftrightarrow(B_3,B_2)$ and by (2), we have
	$(\Ba_1,\Ba_3)\leftrightarrow(B_2,B_3)$, hence (b) holds.
	If $(\Ba_1,\Ba_3)\not\leftrightarrow(B_2,B_3)$, then by (2), we have 
	$(\Ba_1,\Ba_2)\leftrightarrow(B_3,B_2)$ and by (3), we have
	$(\Ba_2,\Ba_3)\leftrightarrow(B_2,B_1)$, hence (c) holds.
	
	If (a) holds, then we have \eqref{eq: leftrightarrows}.
	If (b) holds, then we obtain \eqref{eq: leftrightarrows}
	by taking $B_1$ to be $B_2$ and $B_2$ to be $B_1$, 
	If (c) holds, then we obtain \eqref{eq: leftrightarrows}
	by taking $B_2$ to be $B_3$ and $B_3$ to be $B_2$.
	Our last  assertion is proved if $(B_1,B_2)\in \sC^r_R$.
	If $(B_1,B_2)\not\in\sC^r_R$, our assertion is proved by taking 
	$B_1$ to be $B_3$ and $B_3$ to be $B_1$.
\end{proof}

We are now ready to prove the cocycle condition for our $h_f$.

\begin{proposition}\label{prop: cocycle}
	Assume that $X$ is a weakly transferable locale.
	Let $R>0$ and let $f\in\CS$ such that $\partial f\in C^1_R(\State)$.
	Then the pairing $h_f\colon\Val\times\Val\rightarrow\bbR$ satisfies
	\[
		h_f(\gra,\grb)+h_f(\gra+\grb,\grc)=h_f(\grb,\grc)+h_f(\gra,\grb+\grc).
	\]
	for any $\gra,\grb,\grc\in\Val$.
\end{proposition}

\begin{proof}
	For an arbitrary $\gra,\grb,\grc\in\Val$, there exists $k\geq r_0$ such that
	$\gra,\grb,\grc,\in\Val_k$.  Let $\Ba_1,\Ba_2,\Ba_3$ be balls in $X$
	satisfying the condition of Lemma \ref{lem: triple} for $r=k$.
	Then since $\gra,\grb,\grc\in\Val_k$,
	there exists $\state\in\State$ such that $\xi^{B_1}(\state)=\gra$, $\xi^{B_2}(\state)=\grb$,
	$\xi^{B_3}(\state)=\grc$ and is at base state outside $B_1\cup  B_2\cup  B_3$.
	Then since $(B_1,B_3)\leftrightarrow  (B_2,B_3)$,  there exists a finite
	connected subset $\La \subset X\setminus  B(B_3,R)$
	such that $B_1\cup  B_2\subset \La$,
	and since $(B_1,B_2)\leftrightarrow(B_1,B_3)$, 
	there exists a finite connected subset
	$\La' \subset X\setminus  B(B_1,R)$ such that $B_2\cup  B_3\subset \La'$.	
	Note by construction, we have $(B_1,\La')$, $(\La,B_3)\in\sA_R$,
	and $(B_1,B_2), (B_2,B_3)\in\sC^k_R$.
	Then 
	by Lemma \ref{lem: independence} and the definition of $h$, we have
	\begin{align*}
		h_f(\gra,\grb)&=h_f^{B_1,B_2}(\gra,\grb)=\iota^{B_1\cup  B_2}f(\state)
		-\iota^{B_1}f(\state)-\iota^{B_2}f(\state)\\
		h_f(\gra+\grb,\grc)&=h_f^{\La,B_3}(\gra+\grb,\grc)
		=\iota^{\La\cup  B_3}f(\state)
		-\iota^{\La}f(\state)-\iota^{B_3}f(\state)\\
		h_f(\grb,\grc)&=h_f^{B_2,B_3}(\gra,\grb)=\iota^{B_2\cup  B_3}f(\state)
		-\iota^{B_2}f(\state)-\iota^{B_3}f(\state)\\
		h_f(\gra,\grb+\grc)&=h_f^{B_1,\La'}(\gra,\grb+\grc)=\iota^{B_1\cup\La'}f(\state)
		-\iota^{B_1}f(\state)-\iota^{\La'}f(\state).
	\end{align*}
	Since $\state$ is at base state outside $B_1\cup  B_2\cup  B_3$, we have
	\begin{align*}
		\iota^{\La\cup  B_3}f(\state)&=\iota^{B_1\cup\La'}
		f(\state)=\iota^{B_1\cup  B_2\cup  B_3}f(\state), &
		\iota^{\La}f(\state)&=\iota^{B_1\cup  B_2}f(\state), &
		\iota^{\La'}f(\state)&=\iota^{B_2\cup  B_3}f(\state),
	\end{align*}
	which proves our assertion.
\end{proof}

%
\subsection{Criterion for Uniformity in the Weakly Transferable Case}\label{subsec: criterion weak}
%

In this subsection, we will prove Proposition \ref{prop: important wr},
which is a weakly transferable version of Proposition \ref{prop: important}.

\begin{proposition}\label{prop: important wr}
	Assume that the locale $X$ is weakly transferable.
 	Suppose $f\in\CS$ satisfies $\partial f\in C^1_R(\State)$ for $R>0$,
	and let $h_f\colon\Val\times\Val\rightarrow\bbR$ be the pairing defined in Definition \ref{def: h case2}.
	If $h_f\equiv0$, then we have  $f\in C^0_\unif(\State)$.
\end{proposition}

The proof of Proposition \ref{prop: important wr} is the same
as that of Proposition \ref{prop: important}, but by using Lemma \ref{lem: final hard}
below which is valid even when $X$ is weakly transferable, instead of 
Lemma \ref{lem: final easy}, which assumed that $X$ is strongly transferable.

\begin{lemma}\label{lem: final hard}
	Assume that $X$ is weakly transferable.
	Suppose $f\in\CS$ and $\partial f\in C^1_R(\State)$, and 
	let $h_f\colon\Val\times\Val\rightarrow\bbR$ be the pairing defined in Definition \ref{def: h case2}.
	If $h_f\equiv0$,
	then for any $x\in X$ and finite $\La\subset X$ such that $d_X(x,\La)>R$, 
	we have
	\[
		\iota^{\La\cup  B(x,R)}f-\iota^{\La\cup \cB^*(x,R)}f
		=\iota^{ B(x,R)}f-\iota^{\cB^*(x,R)}f,
	\]
	where $\cB^*(x,R)\coloneqq  B(x,R)\setminus\{x\}$ is the punctured ball.
\end{lemma}

\begin{proof}
	Since $h_f\equiv0$, by the definition of the pairing $h_f$ given in Definition \ref{def: h case2}
	and Proposition \ref{prop: h}, we have
	\begin{equation}\label{eq: distribution three}
		\iota^{\La\cup\La'}f=\iota^{\La} f+\iota^{\La'}f
	\end{equation}
	for any $(\La,\La')\in\sA_R$.
	Again, if $B$ is a ball, then $(B, B(x,R))$ and $(B,\cB^*(x,R))$ are both
	pairs in $\sA_R$, hence our assertion immediately follows by applying \eqref{eq: distribution three}
	to the left hand side.
	Consider a general finite $\La\subset X$ such that $d_X(x,\La)>R$.
	It is sufficient to prove that
	\begin{equation}\label{eq: goal two}
		f(\state|_{\La\cup  B(x,R)})-f(\state|_{\La\cup \cB^*(x,R)})
		=f(\state|_{ B(x,R)})-f(\state|_{\cB^*(x,R)})
	\end{equation}
	for any $\state\in\State_*$ such that $\Supp(\state)\subset\La\cup  B(x,R)$.
	Since $X$ is weakly transferable, $Y\coloneqq X\setminus  B(x,R)$ decomposes
	into a finite number of locales $Y=Y_1\cup\cdots\cup Y_K$.
	We may reorder the $Y_i$ so that there exists integer $k$ such that
	$Y_i\cap\La\neq\emptyset$ for $i\leq k$ and  $Y_i\cap\La=\emptyset$ for $i>k$.
	Let $\La_i\coloneqq\La\cap Y_i$.   Then we have $\La=\La_1\cup\cdots\cup \La_k$.

	Let $\Ba'_1$ be a ball such that $ B(\La,R)\cup  B(x,R)\subset \Ba'_1$.
	Let $\Ba_1\subset (X\setminus  B(\Ba'_1,R))\cap Y_1$, 
	$d_X(\La,\Ba_1)>R$ be a ball whose cardinality is greater than that of 
	$\La_1$.  Such a set  $\Ba_1$ exists since $Y_1$ is an infinite set.
	Now choose a $\state'\in\State_*$ such that $\state'$ coincides with $\state$ outside 
	$\La_1\cup\Ba_1$, is as base state on $\La_1$, and $\bsxi_{\Ba_1}(\state')=\bsxi_{\La_1}(\state)$.
	This implies that $\bsxi_{X}(\state')=\bsxi_{X}(\state)$.
	Since the transition is irreducibly quantified, by Lemma \ref{lem: difference}
	applied to $Y_1$ and $\wt Y_1\coloneqq X\setminus\{x\}$, we have
	\[
		f(\state')-f(\state)
		=f(\state'|_{X\setminus\{x\}})-f(\state|_{X\setminus\{x\}}).
	\]
	Noting that $\state'=\state'|_{\Ba_1\cup \Ba'_1}$ and $\state=\state|_{\Ba'_1}$,
	we see that
	\begin{equation}\label{eq: progress two}
		f(\state'|_{\Ba_1\cup \Ba'_1})-f(\state|_{\Ba'_1})
		=f(\state'|_{\Ba_1\cup (\Ba'_1\setminus\{x\})})-f(\state|_{\Ba'_1\setminus\{x\}}).
	\end{equation}
	By our choice of $\Ba_1$, we have $d_X(\Ba_1,\Ba'_1)>R$, hence 	
	$(\Ba_1, \Ba'_1)$ and $(\Ba_1,\Ba'_1\setminus\{x\})$ are both pairs in $\sA_R$.
	Our condition \eqref{eq: distribution three} on $f$ 
	applied to this pair implies that
	\begin{align*}
		f(\state'|_{\Ba_1\cup \Ba'_1})&=f(\state'|_{\Ba_1})+f(\state'|_{\Ba'_1}),&
		f(\state'|_{\Ba_1\cup(\Ba'_1\setminus\{x\})})&=f(\state'|_{\Ba_1})
		+f(\state'|_{\Ba'_1\setminus\{x\}}).
	\end{align*}
	Hence \eqref{eq: progress two} gives
	$
		f(\state'|_{\Ba'_1})-f(\state|_{\Ba'_1})
		=f(\state'|_{\Ba'_1\setminus\{x\}})-f(\state|_{\Ba'_1\setminus\{x\}}).
	$
	In particular, since $\state'$ is at base state outside 
	$\Ba_1\cup\La_2\cup\cdots\cup\La_k\cup  B(x,R)$ and $\state$ is at base state
	outside $\La\cup  B(x,R)$, we have
	\[
		f(\state'|_{\La_2\cup\cdots\cup\La_k\cup  B(x,R)})-f(\state|_{\La\cup \cB^*(x,R)})
		=f(\state'|_{\La_2\cup\cdots\cup\La_k\cup \cB^*(x,R)})-f(\state|_{\La\cup  B(x,R)}).
	\]
	Note that
	$\state'|_{\La_2\cup\cdots\cup\La_k\cup  B(x,R)}
	=\state|_{\La_2\cup\cdots\cup\La_k\cup  B(x,R)}$
	since $\state'$ and $\state$ coincides outside $\La_1\cup B''_1$.  Hence
	we see that
	\[
		f(\state|_{\La\cup  B(x,R)})-f(\state|_{\La\cup \cB^*(x,R)})
		=f(\state|_{\La_2\cup\cdots\cup\La_k\cup \cB^*(x,R)})
		-f(\state|_{\La_2\cup\cdots\cup\La_k\cup  B(x,R)}).
	\]
	By applying the same argument with $\La$ replaced by $\La_2\cup\cdots\cup\La_k$
	and $\La_1$ replaced by $\La_2$, we have
	\[
		f(\state|_{\La_2\cup\cdots\cup\La_k\cup  B(x,R)})
		-f(\state|_{\La_2\cup\cdots\cup\La_k\cup \cB^*(x,R)})
		=f(\state|_{\La_3\cup\cdots\cup\La_k\cup \cB^*(x,R)})
		-f(\state|_{\La_3\cup\cdots\cup\La_k\cup  B(x,R)}).
	\]
	By repeating this process, we  see that 
	\[
		f(\state|_{\La\cup  B(x,R)})
		-f(\state|_{\La\cup \cB^*(x,R)})
		=f(\state|_{\cB^*(x,R)})
		-f(\state|_{ B(x,R)}).
	\]
	Our assertion \eqref{eq: goal two} follows from the fact that
	$\state$ is supported on $\La\cup  B(x,R)$ and the fact
	that $\iota^W f(\state)=f(\state|_W)$ for any $W\subset X$.
\end{proof}

%
\subsection{Transferability}\label{subsec: resilience}
%

In this subsection, we will introduce the notion of transferability for a locale $X$,
which ensures that the pairing $h_f$ defined in Definition \ref{def: h case2} is \textit{symmetric}.

\begin{definition}\label{def: transferable}
	Let $X$ be a locale.
	We say that $X$ is \textit{transferable}, if $X$ is weakly transferable and satisfies either one of the following conditions.
	\begin{enumerate}[label=(\alph*)]
		\item There exists a ball $\Ba$ such that $X\setminus\Ba$ has three or more 
		disjoint connected components (which are all infinite since $X$ is weakly transferable).
		\item  For any $r\in\bbN$, there exists a ball $\Ba$ of radius 
		at least $r$ such that $X\setminus\Ba$ is connected.
	\end{enumerate}
	Note that by (b), if $X$ is strongly transferable, then it is transferable.
\end{definition}

\begin{example}\label{example: transferable}
	The following subgraphs $X=(X,E)$ of the Euclidean lattice $\bbZ^2=(\bbZ^2,\bbE)$
	give examples of locales which are transferable but \textit{not} strongly transferable.
	\begin{enumerate}
	\item Let
	$
		X\coloneqq\{ (x_1,x_2)\in\bbZ^2\mid x_1x_2=0\}
	$
	and $E\coloneqq (X\times X)\cap\bbE$.
	Then $(X,E)$ is transferable by Definition \ref{def: transferable} (a).
	\item Let
	$
		X\coloneqq\{ (x_1,x_2)\in\bbZ^2\mid x_1\geq0\}\cup \{ (x_1,x_2)\in\bbZ^2\mid  x_2=0\}
	$
	and $E\coloneqq (X\times X)\cap\bbE$.
	Then $(X,E)$ is transferable by
	Definition \ref{def: transferable} (b).
	\begin{figure}[ht]
		\centering
		\includegraphics[width=15cm]{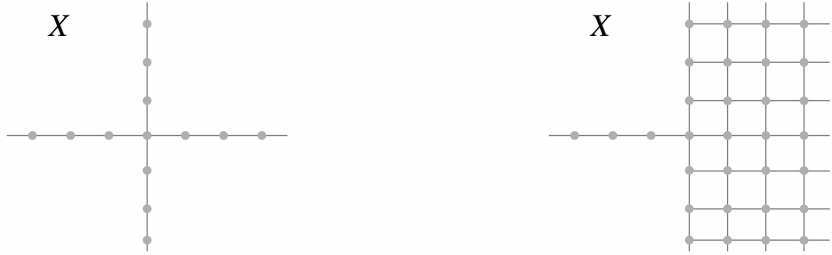}
		\caption{Examples of Transferable but not Strongly Transferable Locales}
	\end{figure}
	\end{enumerate}
\end{example}

\begin{remark}
	 The locale $\bbZ=(\bbZ,\bbE)$
	is weakly transferable but \textit{not} 
	transferable, since the complements of balls always have exactly two connected components.
	The locale $\bbZ^d=(\bbZ^d,\bbE^d)$ for $d>1$ is transferable, since it is strongly transferable.
\end{remark}

\begin{remark}\label{rem: Cayley}
	Let $G$ be a finitely generated \textit{free} group $G$ with set of generators $\cS$, and
	let $(G,E_\cS)$ be the Cayley graph given in Example \ref{example: locale} (4).
	If the number of free generators of $G$ is 
	equal to \textit{one},
	then $G=\bbZ$, hence $(G,E_\cS)$ is weakly transferable but \textit{not} transferable.
	If the number of free generators of $G$ is $d>1$, then $(G,E_\cS)$ is transferable.
	This follows from Definition \ref{def: transferable} (a) and
	the fact that $G\setminus \{\id_G\}$ has exactly $2d$ connected
	components (see Figure \ref{fig: 1} for the case $d=2$).
	
	More generally, for a finitely generated infinite group $G$, it is known that
	the number of \textit{ends} of a Cayley graph $(G,E_\cS)$ is either $1$, $2$ or infinity,
	and the number of ends is $2$ if and only if it has an infinite cyclic subgroup of finite index
	 (see \cite{Coh70}*{(1.6) Theorem}).  
	In the case that $(G,E_\cS)$ is weakly transferable,
	then it is strongly transferable, \textit{not} transferable, or transferable
	if and only if the number of ends is respectively $1$, $2$, or infinity.
\end{remark}

\begin{proposition}\label{prop: good}
	Suppose $X$ is weakly transferable.  Then $X$ is transferable if an only if
	$\sB_R$ has a unique equivalence class 
	for any $R>0$ in the sense of Definition \ref{def: cases}.
\end{proposition}

\begin{proof}
	Suppose (a) of Definition \ref{def: transferable} holds for the ball $\Ba$. 
	Let $Y_1$, $Y_2$, $Y_3$ be 
	three distinct connected components (there may be more) of $X\setminus \Ba$.
	For any $r\in\bbN$, there exists balls $\Ba_1$, $\Ba_2$, $\Ba_3$ of radii at least $r$
	such that
	$ B(\Ba_i,R)\subset Y_i$ for $i=1,2,3$ and $d_X(\Ba_i,\Ba_j)>R$ for any
	$i\neq j$.  In particular, by exchanging the numberings $i=1,2$ if necessary,
	we may assume that $(\Ba_1,\Ba_2)\in \sC^r_R$.  
	By construction, there exists paths from $\Ba$ to $Y_1$ and $\Ba$ to $Y_3$ 
	which do not pass through $Y_2$, hence $Y_1$ and $Y_3$ are in the same connected component
	of $X\setminus Y_2$. Hence in particular, $\Ba_1$ and $\Ba_3$ are in the same connected component
	of $X\setminus  B(\Ba_2,R)$, and we have $(\Ba_1,\Ba_2)\leftrightarrow(\Ba_3,\Ba_2)$.
	In the same manner, we may prove that $(\Ba_1,\Ba_3)\leftrightarrow (\Ba_1,\Ba_2)$,
	and $(\Ba_1,\Ba_3)\leftrightarrow (\Ba_2,\Ba_3)$. This shows that we have
	$(\Ba_1,\Ba_2)\sim_{r}(\Ba_3,\Ba_2)\sim_{r} (\Ba_3,\Ba_1)\sim_{r} (\Ba_2,\Ba_1)$.
	By Proposition \ref{prop: \Ba}, we see that $\sB^r_R$ has a unique equivalence class
	with respect to the equivalence relation $\sim_{r}$
	
	Next, suppose (b) of Definition \ref{def: transferable} holds.
	Fix a $r\in\bbN$ and $R>0$, and let $ B(x,R')$ be the ball given by (2)
	of radius $R'\geq k+R$ such that $X\setminus B(x,R')$ is connected. 
	If we let $\Ba_1\coloneqq  B(x,R'-R)$, then $\Ba_1$ is a ball of radius at least $r$
	such that $X\setminus B(\Ba_1,R)$ is connected.
	Take a ball $\Ba_2$ of radius at least $r$ and $d_X(\Ba_1,\Ba_2)>R$,
	and a finite connected $\La$ such that $\Ba_1\cup \Ba_2\subset \La$.
	Since $X$ is weakly transferable, $X\setminus  B(\Ba_2,R)$ decomposes 
	into a finite number of locales.
	We let $Y$ be the connected component of $X\setminus  B(\Ba_2,R)$
	containing $\Ba_1$, and we take $\Ba_3$ to be a ball of radius at least $r$
	satisfying $\Ba_3\subset Y\setminus  B(\Ba_1,R)$ and $d_X(\La,\Ba_3)>R$.
	Since $X\setminus  B(\Ba_1,R)$ is connected by our construction of $\Ba_1$, 
	we have $(\Ba_1,\Ba_2)\leftrightarrow (\Ba_1,\Ba_3)$.
	Also since $\Ba_1, \Ba_3\subset Y$, we have
	$(\Ba_2,\Ba_1)\leftrightarrow (\Ba_2,\Ba_3)$.
	Moreover, since $ B(\Ba_3,R)\cap\La=\emptyset$ and $\La$ is connected, 
	we have $(\Ba_1,\Ba_3)\leftrightarrow (\Ba_2,\Ba_3)$.
	 This shows that $(\Ba_1,\Ba_2)\sim_{r} (\Ba_1,\Ba_3)\sim_{r}(\Ba_2,\Ba_3)\sim_{r}(\Ba_2,\Ba_1)$,
	 which by Proposition \ref{prop: \Ba} shows  that $\sB^r_R$ has a 
	 unique equivalence class with respect to the equivalence relation $\sim_{r}$.

	Finally, we prove that if there exists $r\in\bbN$ such that
	for any ball $\Ba$ of radius at least $r$, the set  $X\setminus \Ba$ 
	has exactly two connected components,
	then for any $R>0$, 
	there are two equivalence classes in $\sB^r_R$ with respect to $\sim_{r}$.
	We first prove that in this situation, for any ball $\Ba$ of radius at least $r$,
	there exists a decomposition $X\setminus\Ba=Y^+_\Ba\cup Y^-_\Ba$ into infinite
	locales,
	where $Y^{\pm}_{\Ba}\cap Y^{\pm}_{\Ba'}$ are infinite sets and $Y^{\pm}_{\Ba}\cap 
	Y^{\mp}_{\Ba'}$
	are finite sets for any balls $\Ba,\Ba'$ of radius at least $r$.
	This follows from the fact that one of $Y^+_{\Ba}\cap Y^+_{\Ba'}$ 
	or $Y^+_{\Ba}\cap Y^-_{\Ba'}$ must be infinite since $Y^+_{\Ba}$ is infinite.
	If both are infinite, then there exists a ball $\Ba$ satisfying
	$\Ba\cup \Ba'\subset \Ba$ and $X\setminus\Ba$ has at least three
	disjoint connected components in 
	$Y^+_{\Ba}\cap Y^+_{\Ba'}$, $Y^+_{\Ba}\cap Y^-_{\Ba'}$ and $Y^-_{\Ba}$,
	contradicting our assertion.   Hence  exactly one of 
	$Y^+_{\Ba}\cap Y^+_{\Ba'}$ and $Y^+_{\Ba}\cap Y^-_{\Ba'}$ is infinite and the other is finite.
	So we choose the $\pm$ label of $\Ba'$ so that 
	$Y^+_{\Ba}\cap Y^+_{\Ba'}$ is infinite and $Y^+_{\Ba}\cap Y^-_{\Ba'}$ is finite.
	
	Consider any $R>0$. For each $(\Ba,\Ba')\in\sB^r_R$, 
	we denote $\Ba<\Ba'$ if $\Ba'\subset Y_{\Ba}^+$.
	This is equivalent to the condition that $\Ba\subset Y_{\Ba'}^-$ for the following reason.
	Consider $x\in Y_{\Ba}^-\cap Y_{\Ba'}^-$ and $y\in\Ba'\subset Y_{\Ba}^+$. 
	Since $X$ is connected,
	there exists a path $\vec p$ from $x$ to $y$, which must go through $\Ba$
	since $Y_{\Ba}^-$ and $ Y_{\Ba}^+$ are disjoint and connected by $\Ba$.
	This shows that $x$ hence $Y^-_{\Ba'}$ is connected to $\Ba$ in $X\setminus\Ba'$.
	Hence for any $(\Ba,\Ba')\in\sB^r_R$,
	either one of $\Ba<\Ba'$ or $\Ba>\Ba'$ hold.
	Moreover, if $\Ba<\Ba'$, then $\Ba'\subset Y^+_{ B(\Ba,R)}$ and 
	$\Ba\subset Y^-_{ B(\Ba',R)}$.
	This shows that for any $(\Ba_1,\Ba'_1), (\Ba_2,\Ba'_2)\in\sB^r_R$,
	we have $(\Ba_1,\Ba'_1)\leftrightarrow (\Ba_2,\Ba'_2)$ if and only if
	we have $\Ba_1<\Ba'_1$ and $\Ba_2<\Ba'_2$, or $\Ba_1>\Ba'_1$ and $\Ba_2>\Ba'_2$.
	In other words, pairs are in the same equivalence class if and only if their order is preserved.
	This proves that $\sB^r_R$ has exactly two equivalence classes as desired.
\end{proof}

\begin{corollary}
	If $X$ is transferable, then for any $f\in C(\State_*)$ such that $\partial f\in C^1_\unif(\State)$,
	the pairing $h_f$ defined in Definition \ref{def: h case2} is symmetric.
\end{corollary}

\begin{proof}
	By Proposition \ref{prop: good}, if $X$ is transferable, then $\sB_R$ has a unique 
	class for any $R>0$ in the sense of Definition \ref{def: cases}.
	Our assertion now follows from Lemma \ref{lem: commutative}.
\end{proof}

%
%
%
\section{Uniform Cohomology and the Main Theorem}\label{sec: cohomology}
%
%
%

In this section, we will introduce and calculate
the \textit{uniform cohomology} of a configuration space with
transition structure satisfying certain conditions.
We will then consider a free action of a group $G$ on the locale,
and prove our main theorem, Theorem \ref{thm: main},
which coincides with Theorem \ref{thm: A} of \S\ref{subsec: main}.

%
\subsection{Uniform Cohomology of the Configuration Space}\label{subsec: uniform cohomology}
%

In this subsection, we will define the uniform cohomology of a configuration 
space with transition structure.  We will then prove our key theorem,  Theorem \ref{thm: 1},
which states that under certain conditions,
any closed uniform form is integrable by a uniform function.
This result is central to the proof of our main theorem.

Consider the system $(S,X,\phi)$ and the associated configuration space with transition structure.
In what follows, we denote by $Z^1(\State_*)$  be the set of closed forms of $\State_*$, and
$Z^1_\unif(\State)\coloneqq C^1_\unif(\State)\cap Z^1(\State_*)$ the set of uniform
closed forms.   We first prove the following.
\begin{lemma}\label{lem: preserve}
	Let $f\in C^0_\unif(\State)$.  Then we have $\partial f\in Z^1_\unif(\State)$.
\end{lemma}

\begin{proof}
	By the definition of $f\in C^0_\unif(\State)$, there exists $R>0$ such that
	\[
		f=\sum_{\La\subset X,\,\diam{\La}\leq R}f_\La
	\]
	for $f_\La\in C_\La(\State)\subset C(S^\Lambda)$.
	Then for any $e\in E$, we have
	$\nabe f_\La\in C(S^{\La\cup  e})$ and $\nabe f_\La\equiv 0$ if $\La\cap e=\emptyset$,
	hence $\partial f\in C^1_{R+1}(\State)\subset C^1_\unif(\State)$.  The fact that 
	$\partial f\in Z^1(\State_*)$ follows from Lemma \ref{lem: exact},
	which states that exact forms are closed.
\end{proof}
Lemma \ref{lem: preserve} shows that 
the differential $\partial$ induces an homomorphism
\begin{equation}\label{eq: comp}
	\partial\colon C^0_\unif(\State)\rightarrow Z^1_\unif(\State).
\end{equation}
As in Definition \ref{def: uc},
we define the \textit{uniform cohomology} $H^m_\unif(\State)$ 
of the configuration space $\State$ with transition
structure to be the cohomology of the complex \eqref{eq: comp}.
Our choice for $C^0_\unif(\State)$ of restricting to functions satisfying $f(\star)=0$
implies that uniform cohomology is philosophically the reduced cohomology 
in the sense of topology of the pointed space consisting of the configuration space 
$\State$ and base configuration $\star\in\State$.

Now, let $X$ be a locale which is weakly transferable in the sense of Definition \ref{def: transferable} 
and assume that the interaction is irreducibly quantified.
Let $c_\phi=\dim_\bbR\Consv^\phi(S)$.
Our theorem concerning the integration of uniform forms is as follows.

\begin{theorem}\label{thm: 1}
	Let $(X,S,\phi)$ be a system such that the interaction $\phi$ is irreducibly quantified.
	If the locale $X$ is transferable, or the interaction $\phi$ is simple and $X$ is weakly transferable, 
	then for any $\omega\in Z^1_\unif(\State)$, there exists $f\in C^0_\unif(\State)$ such that 
	$\partial f=\omega$.  
\end{theorem}

The proof of Theorem \ref{thm: 1} is based on Lemma \ref{lem: splitting1} and Lemma
\ref{lem: splitting2} below. 
A homomorphism of monoids is a map of sets preserving the binary operation
and the identity element.
We say that a monoid $\sE$ with operation $+_{\sE}$
satisfies the (right)-\textit{cancellation property}, if $a+_{\sE} b=a'+_{\sE} b$
implies that $a=a'$ for any $a,a',b\in \sE$.
Note that any monoid obtained as a submonoid of a group satisfies the cancellation property.
We will first prove the following lemma concerning commutative monoids satisfying the cancellation
property.

\begin{lemma}\label{lem: retract}
	Assume that $\sE$ is a commutative monoid with operation $+_{\sE}$ satisfying the
	cancellation property.
	If there exists an injective homomorphism of monoids
	\[
		\iota\colon\bbR\hookrightarrow\sE,
	\]
	where we view $\bbR$ as a an abelian group via the usual addition,
	then there exists a homomorphism of monoids $\pi\colon\sE\rightarrow\bbR$
	such that
	$\pi\circ\iota=\id_\bbR$.
\end{lemma}

\begin{proof}
	The statement follows from the fact that $\bbR$ is a divisible abelian group (see Remark \ref{rem: divisible}).
	We will give a proof for the sake of completeness. 
	We consider the set of pairs $(\sN, \pi_{\!\sN})$, where $\sN$ is a submonoid of $\sE$ containing $i(\bbR)$
	and $\pi_{\!\sN}\colon\sN\rightarrow \bbR$ is an homomorphism of monoids such that 
	$\pi_{\!\sN}\circ\iota=\id_\bbR$.
	
	We consider an order on the set of such pairs such that $(\sN, \pi_{\!\sN})\leq (\sN', \pi_{\!\sN'})$
	if and only if $\sN\subset\sN'$ and $\pi_{\!\sN'}\big|_{\!\sN}=\pi_{\!\sN}$.
	Let $((\sN_i,\pi_{\!\sN_i}))_{i\in I}$ be a totally ordered set of such pairs.
	If we let $\sN\coloneqq\bigcup_{i\in I}\sN_i,$ and if we let 
	$\pi_{\!\sN}\colon\sN\rightarrow \bbR$ be a homomorphism of 
	monoids obtained as the collection of $\pi_{\!\sN_i}$, 
	then we have $\pi_{\!\sN}\circ\iota=\id_\bbR$.  Hence $(\sN,\pi_{\!\sN})$ is a maximal element of the totally
	ordered set $((\sN_i,\pi_{\!\sN_i}))_{i\in I}$.  By Zorn's lemma, there exists a maximal pair $(\sM,\pi_{\!\sM})$
	in the set of all such pairs.  
	
	We will prove by contradiction that $\sM=\sE$.
	Suppose $\sM\subsetneq\sE$.  Let $w\in\sE$ such that $w\not\in\sM$,
	and let  $\sM'\coloneqq\sM+_{\!\sE}\bbN w$.
	Note that since $\sE$ is commutative,
	if there exists a nontrivial algebraic relation between elements of $\sM$ and $\bbN w$,
	then there would exist $a,b\in\sM$ such that $a+_{\!\sE} mw=b+_{\!\sE}nw$
	for some $m,n\in\bbN$.  By the cancellation property of $\sE$, this implies
	that $a+_{\!\sE} nw=b$ for some $a,b\in\sM$ and integer $n>0$.
	In the case that there exists an integer $n>0$ and $a,b\in\sM$ such that $a+_{\!\sE} nw=b$,
	we define $u$ be an element in $\bbR$ such that
	\[
		nu=\bigl(\pi_{\!\sM}(b)-\pi_{\!\sM}(a)\bigr)\in\bbR.
	\]
	If $n',a',b'$ satisfies the same condition, then we have $n' a+_{\!\sE} n'nw=n'b$
	and $n a'+_{\!\sE} nn'w=nb'$.  Combining this equality,  we have
	\[
	n'a+_{\!\sE} n'nw+_{\!\sE} nb'=n'b+_{\!\sE} n a'+_{\!\sE} nn'w,
	\]	
	hence $n' a+_{\!\sE} nb'=n'a+_{\!\sE} n b'$ since $\sE$ is commutative and satisfies the cancellation property.
	This shows that we have
	\[
		n'(\pi_{\!\sM}(b) -  \pi_{\!\sM}(a)) = n( \pi_{\!\sM}(b')- \pi_{\!\sM}(a')),
	\]
	which implies that $u$ is independent of the choice of $a,b\in\sM$ and $n>0$.
	On the other hand, if for any integer $n>0$ and $a,b\in\sM$, we have  $a+_{\!\sE} nw\neq b$,
	then we let $u=0$.  
	In both cases, we define the function 
	$\pi_{\!\sM'}\colon\sM'\rightarrow \bbR$ by 
	$\pi_{\sM'}(a+nw)=\pi_{\!\sM}(a)+nu$ for any $a\in\sM$ and $n\in\bbN$.  
	This gives a homomorphism of monoids satisfying 
	$\pi_{\sM'}\circ\iota=\id_\bbR$, hence we have $(\sM,\pi_{\!\sM})<(\sM',\pi_{\!\sM'})$,
	which contradicts the fact that  $(\sM,\pi_{\!\sM})$ is maximal for such pairs.
	
	This shows that we have $\sM=\sE$.  Then $\pi\coloneqq\pi_{\!\sM}\colon\sE\rightarrow \bbR$ 
	is a homomorphism of monoids satisfying $\pi\circ\iota=\id_\bbR$ as desired.
\end{proof}

\begin{remark}\label{rem: divisible}
	Assume that $D$ is a \textit{divisible} abelian group, that is, 
	for any $v\in D$ and integer $n>0$,
	there exists $u\in D$ such that $un=v$.
	Then 
	we may prove that if $\iota\colon D\hookrightarrow\sE$ is an injective homomorphism
	into a commutative monoid $\sE$ satisfying the cancellation property,
	then there exists a homomorphism of monoids $\pi\colon\sE\rightarrow D$
	such that $\pi\circ\iota=\id_D$.	
	Indeed, by the Grothendieck construction,
	the monoid $A\coloneqq(\sE\times\sE)/\sim$ defined via the equivalence 
	relation $(a,b)\sim(a',b')$ if $a+_{\sE} b'=a'+_{\sE} b$ for any $(a,b), (a',b')\in\sE\times\sE$
	is an abelian group, since $\sE$ satisfies the cancellation property.
	In addition, we have an
	injective homomorphism $\sE\hookrightarrow A$ 
	given by mapping $a\in\sE$ to the class of $(a,0)$ in $A$.
	Since the composition $\iota_A\colon D\hookrightarrow\sE\hookrightarrow A$
	is an injective homomorphism of abelian 
	groups, from the fact that $D$ is divisible hence an injective object
	in the category of abelian groups (see \cite{Rob96}*{4.1.2}),
	there exists a homomorphism $\pi_A\colon A\rightarrow D$
	such that $\pi_A\circ\iota_A=\id_D$.
	Then $\pi\coloneqq\pi_A|_{\sE}$ satisfies $\pi\circ\iota=\id_D$ as desired.
	The authors thank Kei Hagihara and Shuji Yamamoto for discussion concerning this remark
	as well as the proof of Lemma \ref{lem: retract}.
\end{remark}

Next, let $\Val$ be a commutative monoid with operation $+$ and identity element $0$.
We first give a construction of extensions of monoids 
arising from a pairing  $H\colon \Val\times\Val\rightarrow\bbR$ satisfying $H(0,0)=0$ and
 \begin{equation}\label{eq: cc}
		H(\gra,\grb)+H(\gra+\grb,\grc)=H(\grb,\grc)+H(\gra,\grb+\grc)	
\end{equation}
for any $\gra,\grb,\grc\in\Val$.
We let $\sE \coloneqq \Val \times\bbR$, and we denote an element
of $\sE$ by $[\gra, u]$ for $\gra\in\Val$ and $u\in \bbR$.  
Define the binary operation $+_{\sE}$ on $\sE$ by
\[
	[\gra,u]+_{\sE}[\grb,v]\coloneqq[\gra+\grb,u+v-H(\gra,\grb)]
\]
for any $\gra,\grb\in\Val$ and $u,v\in\bbR$.
The element $[0,0]$ is an identity element with respect to $+_{\sE}$, and the
cocycle condition \eqref{eq: cc} ensures that $+_{\sE}$ is associative.
Hence $\sE$ is a monoid with respect to the operation $+_{\sE}$,
which fits into the exact sequence of monoids
\begin{equation}\label{eq: extension}
	\xymatrix{
			0\ar[r]&\bbR\ar[r]&\sE\ar[r]&\Val\ar[r]&0.
	}
\end{equation}
Here, the arrows are homomorphisms of monoids
with the injection $\bbR\rightarrow\sE$ 
given by $u\mapsto [0,u]$, and the surjection $\sE\rightarrow\Val$ given by $[\gra,u]\mapsto\gra$.
If the pairing $H$ is symmetric, then we see that $\sE$ is a commutative monoid.
Note that if $\Val$ satisfies the right cancellation property,
then $\sE$ also satisfies the right cancellation property.  In fact,
if $[\gra,u]+_{\sE}[\grb,v]=[\gra',u]+_{\sE}[\grb,w]$ for some $\gra,\gra',\grb\in\Val$ and $u,u',v\in\bbR$, 
then we have
\[
	[\gra+\grb,u+v-H(\gra,\grb)]=[\gra'+\grb,u'+v-H(\gra',\grb)].
\]
This shows that $\gra+\grb=\gra'+\grb$, hence $\gra=\gra'$ since $\Val$ satisfies
the right cancellation property.
Then we have $u+v-H(\gra,\grb) =u'+v-H(\gra',\grb)= u'+v-H(\gra,\grb)$ in $\bbR$,
hence $u=u'$.  This shows that $[\gra,u]=[\gra',u']$, proving that the monoid $\sE$ 
satisfies the right cancellation property.

\begin{lemma}\label{lem: splitting1}
	Let $\Val$ be a commutative monoid satisfying the cancellation property,
	and let $H\colon \Val\times\Val\rightarrow\bbR$ be a symmetric pairing satisfying
	the cocycle condition \eqref{eq: cc}.
	Then there exists a function
	\[
		h\colon\Val\rightarrow\bbR
	\]
	such that $H(\gra,\grb)=h(\gra)+h(\grb) - h(\gra+\grb)$ for any $\gra,\grb\in\Val$.
\end{lemma}

\begin{proof}	
	If we prove our assertion for the pairing $\wt H(\gra,\grb)\coloneqq H(\gra,\grb)-H(0,0)$ and show
	that  there exists a function $\wt h\colon\Val\rightarrow\bbR$ satisfying our condition, 
	then our assertion is proved also for $H(\gra,\grb)$ by taking $h(\gra)\coloneqq\wt h(\gra) + H(0,0)$.
	Hence by replacing $H$ by $\wt H$, we may assume that $H(0,0)=0$.
	Consider the extension $\sE$ given in \eqref{eq: extension} corresponding to the pairing $H$.
	Then $\sE$ is a commutative monoid since $H$ is symmetric.  Hence by Lemma \ref{lem: retract},
	there exists a homomorphism of monoids $\pi\colon\sE\rightarrow\bbR$ such that 
	$\pi([0,u])=u$ for any $u\in\bbR$.
	 For any $\gra\in\Val$, 
	choose an arbitrary $u\in\bbR$ and let
	\begin{equation}\label{eq: tilde}
		\wt\gra\coloneqq [\gra,u]+_{\sE}[0,-\pi([\gra,u])]\in\sE.
	\end{equation}
	If $v\in\bbR$, then we have
	$[\alpha,v]=[\alpha,u]+_{\sE} [0,w]$ for some $w\in\bbR$ since the classes of 
	$[\alpha,v]$ and $[\alpha,u]$ coincide on $\Val\cong\sE/\bbR$, hence
	\begin{align*}
		 [\gra,v]+_{\sE}[0,-\pi([\gra,v])]&=[\gra,u]+_{\sE} [0,w]+_{\sE}[0,-\pi([\gra,u]+_{\sE} [0,w])]\\
		&=[\gra,u]+_{\sE} [0,w]+_{\sE}[0,-\pi([\gra,u])]+_{\sE}[0,-\pi([0,w])]\\
		&=[\gra,u]+_{\sE} [0,w]+_{\sE}[0,-\pi([\gra,u])]+_{\sE}[0,-w]\\
		&=[\gra,u]+_{\sE}[0,-\pi([\gra,u])]=\wt\gra,
	\end{align*}
	hence $\wt\gra$ is independent of the choice of $u\in\bbR$.
	We define $h\colon\Val\rightarrow\bbR$ to be the function defined by
	\[
		\wt\gra=[\gra,h(\gra)].
	\]
	Note that by \eqref{eq: tilde}, for any $\gra,\grb\in\Val$, we have
	\begin{align*}
		\wt\gra+_{\sE}\wt\grb&=[\gra,u]+_{\sE}[0,-\pi([\gra,u])]+_{\sE}[\grb,H(\gra,\grb)]
		+_{\sE}[0,-\pi([\grb,H(\gra,\grb)])]\\
		&=[\gra+\grb,u]+_{\sE}[0,-\pi([\gra,u])-\pi([\grb,H(\gra,\grb)])]\\
		&=[\gra+\grb,u]+_{\sE}[0,-\pi([\gra+_{\sE}\grb,u])]=\wt{\gra+\grb}.
	\end{align*}
	This shows that
	\begin{align*}
		[\gra+\grb,h(\gra+\grb)]&=\wt{\gra+\grb}=\wt\gra+_{\sE}\wt\grb\\
		&=[\gra,h(\gra)]+_{\sE}[\grb,h(\grb)]=[\gra+\grb,h(\gra)+h(\grb)-H(\gra,\grb)].
	\end{align*}
	Hence we have
	$
		h(\gra+\grb)=h(\gra)+h(\grb)-H(\gra,\grb)
	$	
	for any $\gra,\grb\in\Val$ as desired.
\end{proof}

Next, we consider the case when $\Val\cong\bbN$ or $\Val\cong\bbZ$.

\begin{lemma}\label{lem: splitting2}
	Assume that $\Val\cong\bbN$ or $\Val\cong\bbZ$,
	viewed as a commutative monoid with respect to addition,
	and let $H\colon \Val\times\Val\rightarrow\bbR$ be a pairing satisfying
	the cocycle condition \eqref{eq: cc}.
	Then there exists a function 
	\[
		h\colon\Val\rightarrow\bbR
	\]
	such that $H(\gra,\grb)=h(\gra)+h(\grb) - h(\gra+\grb)$ for any $\gra,\grb\in\Val$.
\end{lemma}

\begin{proof}
	Again, as in the proof of Lemma \ref{lem: splitting2},
	by replacing the pairing $H$ by $H-H(0,0)$, we may assume that $H(0,0)=0$.
	We first treat the case $\Val\cong\bbN$, which we regard as a commutative 
	monoid with respect to the usual addition. 
	Consider the extension $\sE$ given in \eqref{eq: extension} corresponding to the pairing $H$.
	Note that $[1,0]$ gives an element in $\sE$.
	For any $\gra\in\Val$ corresponding to $n\in\bbN$, 
	we define $h(\gra)$ to be the element in $\bbR$ given by the  formula
	\[
		n[1,0] = [1,0]+_{\sE}\cdots+_{\sE}[1,0]=[\gra,h(\gra)],
	\]
	where $n[1,0]$ is the $n$-fold sum of  $[1,0]$ with respect to the operator $+_{\sE}$.
	Then for $\grb\in\Val$ corresponding to $n'\in\bbN$, we have
	$
		n'[1,0] = [ \grb, h(\grb)]
	$
	and
	\begin{equation}\label{eq: 1}
		 [ \gra, h(\gra)]+_{\sE} [ \grb, h(\grb)]=(n+n')[1,0] = [ \gra+\grb, h(\gra+\grb)].
	\end{equation}
	By the definition of the operation $+_{\sE}$, we have
	\[
		 [ \gra, h(\gra)]+_{\sE} [ \grb, h(\grb)]=[\gra+\grb,h(\gra)+h(\grb)-H(\gra,\grb)].
	\]
	Since the operation $+_{\sE}$ is associative, comparing this equality with \eqref{eq: 1},
	we have $ h(\gra+\grb)=h(\gra)+h(\grb)-H(\gra,\grb)$
	for any $\gra,\grb\in\Val$.
	Hence we see that $h\colon\Val\rightarrow\bbR$ satisfies the 
	requirement of our assertion.
	The case for $\Val\cong\bbZ$ may be proved precisely in the same manner.
	In this case, we regard $\Val$ and $\sE$ as abelian groups instead of 
	commutative monoids.
\end{proof}

\begin{remark}\label{rem: splitting}
	For the case that $\Val$ is an abelian group, the extension 
	$\sE$ in Lemma \ref{lem: splitting1} corresponding to the symmetric
	cocycle $H\colon\Val\times\Val\rightarrow\bbR$
	is a commutative group.
	The existence of $h\colon\Val\rightarrow\bbR$ 
	for this case corresponds to the well-known fact that
	such an extension $\sE$ splits, since the additive group of $\bbR$ is
	divisible hence an injective object in the category of abelian groups
	 (see \cite{Rob96}*{4.1.2}).
	The statement for Lemma \ref{lem: splitting2}
	corresponds to the fact that the group cohomology
	$H^2(\bbZ,\bbR)=\{0\}$.  The existence of $h\colon\Val\rightarrow\bbR$ 
	ensures that the cocycle $H\colon\Val\times\Val\rightarrow\bbR$ is in fact symmetric
	in this case.
\end{remark}

We are now ready to prove Theorem \ref{thm: 1}.

\begin{proof}[Proof of Theorem \ref{thm: 1}]
	Suppose $\omega\in Z_\unif(\State) = C^1_\unif(\State)\cap Z^1(\State_*)$.  
	Since $\omega\in Z^1(\State_*)$, by Lemma \ref{lem: exact}, there exists $f\in \CS$
	such that $\partial f=\omega$.  Furthermore, since $\omega\in C^1_\unif(\State)$,
	there exists $R>0$ such that $\partial f=\omega\in C^1_R(\State)$.
	Hence by 
	Definition \ref{def: h case2} and Proposition \ref{prop: cocycle}
	(see also Proposition \ref{prop: h_f}),
	there exists a pairing
	$h_f\colon \Val\times\Val\rightarrow\bbR$ satisfying the cocycle  condition \eqref{eq: cocycle}.
	By Lemma \ref{lem: commutative}
	if $X$ is transferable (see also Proposition \ref{prop: h_f})
	and by Remark \ref{rem: splitting} if the interaction is simple, we see that $h_f$
	is symmetric.  Note that here we have used the fact that $\cM\cong\bbN$ or $\cM\cong\bbZ$
	if the interaction is simple.
	Hence by 
	Proposition \ref{prop: def h two} (see also Proposition \ref{prop: h_f}),
	we have
	\[
		\iota^{\La\cup\La'} f(\state)-\iota^{\La} f(\state)-\iota^{\La'}f(\state)
		=h_f(\bsxi_\La(\state),\bsxi_{\La'}(\state))
	\]
	for any $(\La,\La')\in\sA_R$ and $\state\in\State$.  
	Note that $\Val$ satisfies the cancellation property since it is a submonoid of an abelian group.
	Hence by Lemma \ref{lem: splitting1}, there exists $h\colon\Val\rightarrow\bbR$
	such that 
	\begin{equation}\label{eq: coboundary}
		h_f(\gra,\grb)=h(\gra)+h(\grb)-h(\gra+\grb)
	\end{equation}
	for any $\gra,\grb\in\Val$.
	We define the function $g\in\CS$ by 
	\[
		g\coloneqq f+h\circ \bsxi_{\!X}.
	\] 
	We will prove that $g$ is uniform.
	By Lemma \ref{lem: horizontal}, 
	for any conserved quantity $\xi\colon S\rightarrow\bbR$, the function  $\xi_X\in\CS$
	is horizontal.  This implies that $\nabe\xi_X=0$, hence
	$\xi_X(\state^e)=\xi_X(\state)$ for any $e\in E$.
	This shows that $h\circ \bsxi_{\!X}(\state^e)= h\circ \bsxi_{\!X}(\state)$ for any $e\in E$,
	hence $\nabe(h\circ \bsxi_{\!X})=0$, which implies that $\partial (h\circ \bsxi_{\!X})=0$.
	This gives the formula $\partial g=\partial f=\omega$.	
	Furthermore, noting that 
	$\iota^{\La} (h\circ \bsxi_{\!X}(\state))=h\circ \bsxi_\La(\state)$ for any $\La\subset X$, 
	we have
	\begin{align*}
		\iota^{\La\cup\La'}g(\state)-\iota^{\La} g(\state)-\iota^{\La'}g(\state)
		&=\iota^{\La\cup\La'} f(\state)-\iota^{\La} f(\state)-\iota^{\La'}f(\state) \\
		&\qquad+h\circ \bsxi_{\La\cup\La'}(\state)- h\circ \bsxi_\La(\state)-h\circ \bsxi_{\La'}(\state)=0
	\end{align*}
	for any $(\La,\La')\in\sA_R$ and $\state\in\State$, 
	where we have used the coboundary condition \eqref{eq: coboundary} and the fact that
	$\bsxi_{\La\cup\La'}(\state) = \bsxi_{\La}(\state)+\bsxi_{\La'}(\state)$.	
	From the definition of the pairing given in Definition \ref{def: h case2}, 
	we see that $h_g\equiv0$. 
	Hence by Proposition \ref{prop: important wr}, 
	we see that $g\in C^0_\unif(\State)$.
	Our assertion is now proved by replacing $f$ by $g$.
\end{proof}

By using Theorem \ref{thm: 1},
we may calculate the uniform cohomology of $\State$ as follows.

\begin{theorem}\label{thm: cohomology}
	Let $X$ be a locale, and assume that the interaction is irreducibly quantified.
	If either $X$ is transferable, or the interaction is simple and $X$ is weakly transferable,
	then we have
	\begin{align*}
		H^m_\unif(\State) \cong
		\begin{cases}
			\Consv^\phi(S)  &m=0,\\
			 \{0\}  &  m\neq0.
		\end{cases}
	\end{align*}
	In particular, we have 
	$\dim_\bbR H^0_\unif(\State)=\dim_\bbR\Consv^\phi(S)$.
\end{theorem}

\begin{proof}
	As stated in \S\ref{subsec: overview}, Theorem \ref{thm: cohomology} is equivalent to the fact
	that the sequence
	\begin{equation}\label{eq: SES3}
	\xymatrix{
	0\ar[r]&\Consv^\phi(S)\ar[r]&C^0_\unif(\State)\ar[r]^<<<<{\partial}&Z^1_\unif(\State)\ar[r]&0
	}
	\end{equation}
	is exact.
	A conserved quantity $\xi\colon S\rightarrow\bbR$
	defines a uniform function $\xi_X\colon\State\rightarrow\bbR$
	whose definition $\xi_X\coloneqq\sum_{x\in X}\xi_x$ is
	the canonical expansion \eqref{eq: expansion} with each $\xi_x\in C_{\{x\}}(\State)$.
	This shows that 
	$\xi_X$ is uniform satisfying $\xi_X(\star)=0$, hence \eqref{eq: inclusion} induces an inclusion
	\begin{equation}\label{eq: inclusion again}
		 \Consv^\phi(S)\hookrightarrow C^0_\unif(\State).
	\end{equation}
	On the other hand, if $f\in C^0_\unif(\State)$ satisfies $\partial f=0$,
	then by Theorem \ref{thm: 2},
	there exists a conserved quantity $\xi\colon S\rightarrow\bbR$
	such that $f(\state)=\xi_X(\state)$,
	which shows that $f$ is in the image of \eqref{eq: inclusion again}.
	This proves that we have an isomorphism 
	\begin{equation}\label{eq: H0}
		 \Consv^\phi(S)\cong\Ker\partial=H^0_\unif(\State).
	\end{equation}
	From Theorem \ref{thm: 1}, we see that the differential $\partial$ is surjective,
	hence the short exact sequence \eqref{eq: SES3} is exact.
	Our assertion now follows from the definition of uniform cohomology
	given in Definition \ref{def: uc}.
\end{proof}

\begin{remark}
	The fact that $H^m_\unif=\{0\}$ for $m\neq0$ reflects the fact
	that we are viewing the configuration space as modeling a space with a simple 
	topological structure whose only topological
	feature is its connected components.
	The $H^0_\unif$ is expressed in terms of the conserved 
	quantities and is finite dimensional if
	$\Consv^\phi(S)$ is finite dimensional.	
\end{remark}

\begin{example}
	In each of the examples of Example \ref{example: interactions2},
	we have the following.
	\begin{enumerate}
		\item In the case of the multi-species exclusion process, we have $H^0_\unif(\State)\cong\bbR^{\kappa}$,
		where $\kappa>0$ is such that $S=\{0,\ldots,\kappa\}$.
		\item In the case of the generalized exclusion process, 
		we have $H^0_\unif(\State)\cong\bbR$.
		\item In the case of the lattice gas with energy process,
		we have $H^0_\unif(\State)\cong\bbR^2$.
		\item For the interaction of Example \ref{example: interactions2} (4), 
		we have $H^0_\unif(\State)\cong\bbR$.
		\item For the Glauber Model of Example \ref{example: interactions2} (5), 
		we have $H^0_\unif(\State)\cong\{0\}$.
	\end{enumerate}
\end{example}

%
\subsection{Group Action on the Configuration Space}\label{subsec: group}
%

In this subsection, we first review basic definitions and results concerning group cohomology
of a group $G$ acting on an $\bbR$-linear space $V$.
We will then consider the action of a group $\Group$ on the locale $X$.

We say that an $\bbR$-linear space $V$ is a $\Group$-module,
if any $\grs\in\Group$ gives an $\bbR$-linear homomorphism
$\grs\colon V\rightarrow V$ such that $(\grt\grs)(v)=\grt(\grs(v)) =\grt\circ\grs(v)$ 
for any $\grs,\grt\in\Group$ and $v\in  V$.
In what follows, we will often simply denote $\grs(v)$ by $\grs v$.
For any $\Group$-module $V$, we denote by $V^\Group$ the 
$\Group$-invariant subspace of $V$, defined by
$V^\Group\coloneqq\{v\in V\mid \grs v=v\,\forall\grs\in\Group\}$.

\begin{definition}\label{def: group cohomology}
	Let $V$ be a $\Group$-module.  The \textit{zeroth} group cohomology $H^0(\Group,V)$
	of $\Group$ with coefficients in $V$ is given as 
	\[
		H^0(\Group,V)\coloneqq V^G.
	\]
	Furthermore, we let 
	\begin{align*}
		Z^1(\Group,V)&\coloneqq\{\psi\colon\Group\rightarrow
		 V\mid\psi(\grs\grt)=\grs\psi(\grt)+\psi(\grs)
		\,\forall \grs,\grt\in\Group\},\\
		B^1(\Group,V)&\coloneqq\{\psi\in Z^1(\Group,V)\mid \exists v\in V,  \psi(\grs) = (\grs-1)v
		\,\forall\grs\in\Group\},
	\end{align*}
	where $1$ is the identity element of $\Group$.  Then the \textit{first} group cohomology $H^1(\Group,V)$ 
	of $\Group$ with coefficients in $V$ is given as 
	\[
		H^1(\Group,V)\coloneqq Z^1(\Group,V)/B^1(\Group,V).
	\]
	In particular, if the action of $\Group$ on $V$ is \textit{trivial}, in other words, if $\grs v=v$ 
	for any $v\in V$ and $\grs\in\Group$, then we have $H^0(\Group,V)=V$ and
	\begin{equation}\label{eq: hom}
		H^1(\Group, V)=\Hom(\Group,V),
	\end{equation}
	where $\Hom(\Group,V)$ denotes the set of homomorphisms of groups from $\Group$ to $V$.
\end{definition}

\begin{remark}
	The group  cohomology of $\Group$ with coefficients in $V$ is usually defined using the
	right derived functor of the functor $\Hom_{\bbZ[\Group]}(\bbZ,-)$ applied to $V$.
	In other words, $H^m(\Group,V)\coloneqq  \Ext^m_{\bbZ[\Group]}(\bbZ,V)$ for any integer 
	$m\in\bbZ$ (see for example \cite{AW65}*{\S1}). Definition \ref{def: group cohomology} 
	is the well-known description of this
	cohomology group in terms of explicit cocycles (see \cite{AW65}*{\S2}).
\end{remark}

Let $V$ and $V'$ be $\Group$-modules.
We say that an $\bbR$-linear homomorphism
\[
	\grhom\colon V\rightarrow V'
\]
is  a $\Group$-homomorphism, 
if $\grhom(\grs v)=\grs\grhom(v)$ for any $\grs\in\Group$ and $v\in V$.
By definition, we have $\grhom(V^\Group)\subset\grhom(V)^\Group$,
where $\grhom(V)$ is the $\Group$-submodule of $V$ defined to be the image of $V$
with respect to $\grhom$.  
Note that $\grhom$ gives an exact sequence
\[
	\xymatrix{
		0\ar[r]&\Ker\grhom\ar[r]&V\ar[r]&\grhom(V)\ar[r]&0
	}
\]
of $\Group$-modules, which by the standard theory of cohomology of groups
(see for example \cite{AW65}*{(1.3)})
gives rise to the long exact sequence
\begin{equation}\label{eq: LES}
	\xymatrix{
		0\ar[r]&(\Ker\grhom)^\Group\ar[r]&V^\Group\ar[r]^-{\grhom}& \grhom(V)^\Group\ar[r]^-{\delta}&
		H^1(\Group,\Ker\grhom)\ar[r]&H^1(\Group,V)\ar[r]&\cdots.
	}
\end{equation}
The homomorphism $\delta$ is given explicitly as follows.
For any $\omega\in\grhom(V)^\Group$, choose a $v\in V$
such that $\grhom(v)=\omega$.  Then $\delta(\omega)\in H^1(\Group,\Ker\grhom)$ is the class given 
by the cocycle satisfying
\begin{equation}\label{eq: delta}
	\delta(\omega)(\grs)=(1-\grs)v
\end{equation}
for any $\grs\in G$ (see \cite{AW65}*{\S 2 p.97}).
Note that since $\omega=\pi(v)$ is invariant under the action of $\Group$, we have 
$\pi(\delta(\omega)(\grs))=(1-\grs)\pi(v)=0$, hence
$\delta(\omega)(\grs)\in\Ker\grhom$ for any $\grs\in\Group$.
Our choice of the sign of the homomorphism $\delta$ in \eqref{eq: delta}
is to ensure compatibility with standard sign conventions used in probability theory.

In what follows, let $X$ be a locale, and let $\Group$ be a group.

\begin{definition}	
	An \textit{automorphism} of a locale $X$ is a
	bijective map of sets $\grs\colon X\rightarrow X$ such that
	$\grs(E)=E\subset X\times X$.  The set $\Aut(X)$ of all automorphisms of $X$
	form a group with respect to the operation given by composition of automorphisms.
	We say that $X$ has an action of $\Group$, if there exists a homomorphism of groups 
	$\Group\rightarrow\Aut(X)$ so that  any $\grs\in\Group$ induces an 
	automorphism $\grs\colon X\rightarrow X$.
\end{definition}

\begin{example}
	\begin{enumerate}
	\item
	Let $d$ be an integer $>0$.
	Consider the Euclidean lattice $X=(\bbZ^d,\bbE^d)$ and let $\Group=\bbZ^d$.
	For any $\grt\in\Group$, if we define the automorphism 
	$\grt\colon X\rightarrow X$ by $\grt(x)\coloneqq x+\grt$ for any $x\in X$,
	then this gives an action of $\Group$ on $X$.
	\item The group $\Group=\bbZ^2$ acts on the triangular and hexagonal lattices via translation.
	The group $\Group=\bbZ^3$ acts on the diamond lattice also via translation.
	\item If $\Group$ is a finitely generated group with set of minimal generators 
	$\cS$, then left multiplication 
	by elements of $\Group$ gives an action of the group
	$\Group$ on the Cayley
	graph $(\Group,E_\cS)$.
	\end{enumerate}
\end{example}

From now until the end of \S\ref{subsec: group2}, 
we assume that $X$ has an action of a group $\Group$.  If we denote the group action from the left,
Then $\State_*=(\State_*,\Phi_*)$ has a natural $\Group$-action
given by $\state^\grs\coloneqq(s_{\grs(x)})$ for any $\state=(\eta_x)$ and $\grs\in\Group$.
Then $C^0(\State_*)=\CS$ and 
$C^1(\State_*)$ have natural $\Group$-actions 
given for any $\grs\in\Group$
by $\grs(f)=f\circ\grs$ for any function
$f\in \CS$, and $\grs(\omega)=\omega\circ\grs$ for any form $\omega\in C^1(\State_*)$.
Since the action of $\Group$ preserves the distance on the locale  $X$ and
preserves closed forms on $C^1(\State_*)$,
the groups
$C_{\unif}(\State)$ and $Z^1_{\unif}(\State)$
have induced $\Group$-module structures.

In case of functions and forms,
we will use the term \textit{shift-invariant} interchangeably with the term \textit{$\Group$-invariant},
when the group $\Group$ is understood.
We say that a subset $\La_0\subset X$ 
is a \textit{fundamental domain} of $X$ for the action of $\Group$,
if it represents the set of orbits of the vertices of 
$X$ for the action  of $\Group$.

\begin{lemma}\label{lem: OK}
	Suppose $X$ has an action of a group $\Group$, and assume that
	the set of orbits of the vertices of $X$ for the action  of $\Group$
	is finite.  Then for any shift-invariant 
	uniform function $F\in C^0_\unif(\State)$, there exists a local function $f\in C_\loc(\State)$  satisfying $f(\star)=0$ such that
	\begin{equation}\label{eq: OK}
		F=\sum_{\grt\in\Group}\grt(f)
	\end{equation}
	in $C^0_\unif(\State)$.
\end{lemma}

\begin{proof}
	By definition, $F(\star)=0$.
	Since $F$ is uniform, there exists $R>0$ such that the expansion
	\eqref{eq: expansion} of $F$ in terms of local functions with exact support is given by
	\[
		F=\sum_{\La\subset X, \diam{\La}\leq R} F_\La.
	\]
	Since $F$ is shift-invariant, we have $\grt(F_\La)=\grt(F)_{\grt(\La)}=F_{\grt(\La)}$ 
	for any $\grt\in\Group$.
	Let $\sI^*_{\!R}$ be the set of nonempty finite $\La\subset X$ such that $\diam{\La}\leq R$.
	Then $\sI^*_{\!R}$ has a natural action of $\Group$.  We denote by $\sim$ the equivalence relation
	on $\sI^*_{\!R}$ given by $\La\sim\La'$ if $\La'=\grt(\La)$ for some $\grt\in\Group$.
	Let $\La_0\subset X$ be a fundamental domain
	of $X$ for the action of $\Group$.
	Then any equivalence class of $\sI^*_{\!R}$ with respect to 
	the relation $\sim$ contains a representative that intersects with $\La_0$.
	Since $\La_0$ is finite, and the diameters of the sets in $\sI^*_{\!R}$ are bounded, this implies
	that $\sI^*_{\!R}/\sim$ is finite.
	For each equivalence class $C$ of $\sI^*_{\!R}/\sim$, let 
	$C_0\coloneqq\{ \La\in C\mid \La\cap \La_0\neq\emptyset\}$, which is again finite.
	If we let
	\[
		f_C\coloneqq \frac{1}{|C_0|}\sum_{\La\in C_0}F_\La,
	\]
	then it is a finite sum hence gives a local function in $C_\loc(\State)$.  Then since $\sI^*_{\!R}/\sim$ is finite,
	\[
		f\coloneqq\sum_{C\in\sI^*_{\!R}/\sim}f_C
	\]
	again defines a local function in $C_\loc(\State)$,
	which by construction satisfies \eqref{eq: OK} as desired.
\end{proof}

The action of $\Group$ on $\Consv^\phi(S)$ viewed as a subspace of 
$C^0_\unif(\State)$ is given by the trivial action.
Applying  \eqref{eq: LES} to the short exact sequence \eqref{eq: SES3}, 
we obtain the long exact sequence
\begin{equation}\label{eq: LES2}
	\xymatrix{
		0\ar[r]& \Consv^\phi(S)\ar[r]& \bigl(C^0_\unif(\State)\bigr)^\Group
		\ar[r]^-{\partial}&\bigl(Z^1_\unif(\State)\bigr)^\Group& \\
		&\ar[r]^-{\delta} & H^1\bigl(\Group,  \Consv^\phi(S)\bigr)
		\ar[r]&H^1\bigl(\Group,C^0_\unif(\State)\bigr)\ar[r]&\cdots.
	}
\end{equation}
Moreover, since $G$ acts trivially on $\Consv^\phi(S)$,
by \eqref{eq: hom},
we have 
\[
	H^1\bigl(\Group,  \Consv^\phi(S))=\Hom(\Group, \Consv^\phi(S)).
\]

\begin{remark}
	We may 
	view the cohomology group $H^1\bigl(\Group, \Consv^\phi(S)\bigr)$ as a group which philosophically 
	reflects the \textit{first} reduced cohomology group of 
	the quotient space $\State/\Group$ with fixed base point $\star/\Group$.
	Intuitively, 
	we are viewing $\State$
	as a model of the configuration space on  $X$ which we view as
	an infinitely magnified version of a point in a macroscopic 
	space.
	In this context, the cohomology $H^1\bigl(\Group, \Consv^\phi(S)\bigr)$
	is regarded as representing the flow of the conserved quantities at this point
	induced from the action of $\Group$.
	More generally, for $m\in\bbZ$, we may view the $m$-th group cohomology 
	$H^m\bigl(\Group, \Consv^\phi(S)\bigr)$ as a group philosophically 
	reflecting the $m$-th reduced cohomology group of 
	the quotient space $\State/\Group$ with fixed base point $\star/\Group$.
\end{remark}

%
\subsection{Group Cohomology of the Configuration Space}\label{subsec: group2}
%

In this subsection, we will prove Theorem \ref{thm: main}, 
which is the main theorem of this article.
We say that an action of a group $\Group$ on a locale $X$ is \textit{free}, if 
$\grs(x)=\grt(x)$ implies that $\grs=\grt$ for any $x\in X$ and $\grs,\grt\in\Group$.  
Throughout this subsection, we assume that $X$ has
a free action of a group $\Group$.

As in \S\ref{subsec: model}, denote by
$\cC=\bigl(Z^1_\unif(\State)\bigr)^\Group$ 
the space of shift-invariant closed uniform forms, and
by $\cE = \partial\bigl(C^0_\unif(\State)^\Group\bigr)$ 
the image by $\partial$ of the space of shift-invariant uniform functions.
The main theorem of this article is the following, given as Theorem \ref{thm: A} in \S\ref{subsec: main}.
 
\begin{theorem}\label{thm: main}
	Let the system $(X,S,\phi)$
	be as in Theorem \ref{thm: cohomology}, and
	assume that the action of $\Group$ on the locale $X$ is free.  Then the
	boundary morphism $\delta$ of \eqref{eq: LES2} gives a
	canonical isomorphism
	\begin{equation}\label{eq: isom}
		\cC/\cE \cong \Hom(\Group, \Consv^\phi(S))
	\end{equation}
	Moreover, a choice of a fundamental domain for the action of $G$ on $X$
	gives an $\bbR$-linear homomorphism $\lambda\colon\Hom(G,\Consv^\phi(S))\rightarrow\cC$ such that
	$\delta\circ\lambda=\id$,
	which gives a decomposition
	\[
		 \cC\cong \cE\oplus \Hom(\Group, \Consv^\phi(S)).
	\]
\end{theorem}

In order to prove Theorem \ref{thm: main},
we first prove Proposition \ref{prop: calculation} concerning the existence of
a section of $\delta$.
Let $\La_0$ be a fundamental domain of $X$ for the action of $\Group$.
Since the action of $\Group$ on $X$ is free, any $x\in X$ may be uniquely written as 
$\grs(x_0)$ for some $x_0\in \La_0$ and $\grs\in\Group$.
Then for $\xi\in \Consv^\phi(S)$, we have
\[
	\xi_X=\sum_{\grs\in\Group}\xi_{\grs(\La_0)},
\]
where $\xi_{W}\coloneqq\sum_{x\in W}\xi_x$ for any $W\subset X$.

\begin{proposition}\label{prop: calculation}
	Let $X$ be a locale with a free action of a group $\Group$,
	and let $\La_0 \subset X$ be a fundamental domain of $X$ for the action of $\Group$.
	For any $\psi\in Z^1\bigl(\Group, \Consv^\phi(S)\bigr)$,
	we let $\omega_\psi\coloneqq\partial(\theta_\psi)$,
	where
	\[
		\theta_\psi\coloneqq\sum_{\grt\in\Group}
		\psi(\grt)_{\grt(\La_0)}\in C^0_\unif(\State).
	\]
	Then we have $\delta(\omega_\psi)=\psi$.
\end{proposition}

\begin{proof}
	By definition, the map is $\bbR$-linear.
	Since $\psi$ is a cocycle with values in $\Consv^\phi(S)$ and 
	the group $\Group$ acts trivially on $\Consv^\phi(S)$, 
	we have $\psi(\grs\grt)=\psi(\grt)+\psi(\grs)$
	for any $\grs,\grt\in\Group$.
	Note that
	\begin{multline*}
		\grs(\theta_\psi)=\sum_{\grt\in\Group}
		\grs\bigl(\psi(\grt)_{\grt(\La_0)}\bigr)
		= \sum_{\grt\in\Group}
		\psi(\grt)_{\grs\grt(\La_0)}\\
		= \sum_{\grt\in\Group}
		\bigl(\psi(\grs\grt)_{\grs\grt(x_0)}-\psi(\grs)_{\grs\grt(\La_0)}\bigr)
		= \sum_{\grt\in\Group}
		\bigl(\psi(\grt)_{\grt(\La_0)}-\psi(\grs)_{\grt(\La_0)}\bigr) =\theta_\psi-\psi(\grs)_X,
	\end{multline*}
	hence we have $(1-\grs)\theta_\psi=\psi(\grs)_X$ for any $\grs\in\Group$.  
	Since $\psi(\grs)$ is a conserved quantity, $\psi(\sigma)_X$ is horizontal by Lemma \ref{lem: horizontal},
	hence we have $(1-\grs)\omega_\psi= (1-\grs)\partial\theta_\psi= \partial\psi(\grs)_X=0$
	for any $\grs\in\Group$.
	This implies that we have $\omega_\psi\in\cC$.
	By the explicit description of the homomorphism 
	$\delta$ in \eqref{eq: delta}, we see that $\delta(\omega_\psi)=\psi$ as desired.
\end{proof}

We may now prove Theorem \ref{thm: main}.

\begin{proof}[Proof of Theorem \ref{thm: main}]
By the definition of $\cC$ and $\cE$, the long exact sequence \eqref{eq: LES2}
gives the exact sequence
 \[
	 \xymatrix{
	 0\ar[r]& \cE\ar[r]&\cC\ar[r]^<<<<\delta
	 &
	\Hom(\Group, \Consv^\phi(S)).
	}
	 \]
	 By Proposition \ref{prop: calculation}, for any $\psi\in \Hom(\Group, \Consv^\phi(S))$,
	 there exists $\omega_\psi\in\cC$ such that $\delta(\omega_\psi)=\psi$.
	  This implies that $\delta$ is surjective.
	 By construction, the map $\psi\mapsto\omega_\psi$ is $\bbR$-linear,
	 hence we have a decomposition
	 $
		 \cC\cong \cE\oplus \Hom(\Group, \Consv^\phi(S)),
	$
	given explicitly by mapping any $\omega\in\cC$ to the element
	$( \omega-\omega_\psi, \psi)\in \cE\oplus \Hom(\Group, \Consv^\phi(S))$,
	where $\psi\coloneqq\delta(\omega)$.
\end{proof}

As a corollary of Theorem \ref{thm: main}, we have the following result, which coincides with
Corollary \ref{cor: A} of the introduction.

\begin{corollary}\label{cor: main}
	Let the system $(X,S,\phi)$
	and the $\Group$-action be as in Theorem \ref{thm: main}.
	Assume in addition that the abelian quotient $\Group^\ab$ of $\Group$ is of finite rank $d$.
	If we choose a generator of the free part of $\Group^\ab$,
	then we have  an isomorphism
	$
		\Hom(\Group, \Consv^\phi(S))
		 \cong\bigoplus_{\indj=1}^d\Consv^\phi(S).
	$
	A choice of a fundamental domain of $X$ 
	for the action of $\Group$ gives a decomposition
	\begin{equation}\label{eq: V}
		 \cC\cong \cE\oplus \bigoplus_{\indj=1}^d\Consv^\phi(S).
	\end{equation}
\end{corollary}	

\begin{proof}
	We have
	\[
		\Hom(\Group, \Consv^\phi(S))
		=\Hom\bigl(\Group^\ab/\Group^\ab_\tors, \Consv^\phi(S)\bigr),
	\]
	where $\Group^\ab_\tors$ is the torsion subgroup of $G^\ab$ and
	the equality follows from the fact that $\Consv^\phi(S)$ is abelian
	and torsion free.
	This implies that any element in $\Hom(\Group, \Consv^\phi(S))$ is determined by the image of the generators
	of the free part of $\Group^\ab$, hence if we fix such a set of generators,
	then we have an isomorphism
	\[
		\Hom(\Group, \Consv^\phi(S))\cong\bigoplus_{\indj=1}^d\Consv^\phi(S).
	\]
	Our assertion now follows from Theorem \ref{thm: main}.
\end{proof}

\begin{remark}\label{rem: explicit}
	Let the system $(X,S,\phi)$ be as in Theorem \ref{thm: cohomology},
	and assume that $\Group=\bbZ^d$ and that the action of $\Group$ on $X$ is free.
	The standard basis of $\Group=\bbZ^d$ gives an isomorphism
	$\Hom(\Group, \Consv^\phi(S))\cong\bigoplus_{\indj=1}^d\Consv^\phi(S)$,
	by associating to the element
	\[	
		\psi=(\zeta^{(1)},\ldots,\zeta^{(d)})
		\in\bigoplus_{\indj=1}^d\Consv^\phi(S)
	\]
	the cocycle $\psi\colon\Group\rightarrow\Consv^\phi(S)$
	given by $\psi(\grt)\coloneqq\sum_{j=1}^d\grt_j\zeta^{(j)}$ for any $\grt=(\grt_j)\in\bbZ^d$.
	If we fix a fundamental domain $\La_0$ of $X$ for the action of $\Group$,
	then the map $\theta_\psi$ in 
	Proposition \ref{prop: calculation} is given by
	\[
		\theta_{\psi}=\sum_{j=1}^d\Bigl(\sum_{\tau\in G}\tau_\indj\zeta_{\tau(\La_0)}^{(j)}\Bigr).
	\]
	Hence the form in $\cC$ corresponding to $\psi=(\zeta^{(1)},\ldots,\zeta^{(d)})$ is given by 
	the form
	$\omega_\psi=\partial\bigl(\sum_{j=1}^d\sum_{\tau\in G}\tau_\indj\zeta_{\tau(\La_0)}^{(j)}\bigr)$
	as stated in \eqref{eq: omega} of \S\ref{subsec: overview}.
	Then Corollary \ref{cor: main} implies that any shift-invariant closed local form 
	$\omega$ decomposes as
	\begin{equation}\label{eq: decompose}
		\omega = \partial\bigl(F+\theta_{\psi}\bigr)=\partial F+\omega_\psi
	\end{equation}
	for some shift-invariant uniform function $F$ in $C^0_\unif(\State)$
	and $\psi\in\bigoplus_{\indj=1}^d\Consv^\phi(S)$.
	For the case when $c_\phi=\dim_{\bbR}\Consv^\phi(S)$ is finite, if we fix a basis $\xi^{(1)},\ldots,\xi^{(c_\phi)}$
	of $\Consv^\phi(S)$, then we have 
	$
		\zeta^{(j)}=\sum_{i=1}^{c_\phi}a_{ij}\xi^{(\indi)}
	$ 
	for some $a_{ij}\in\bbR$, $i=1,\ldots,c_{\phi}$, $j=1,\ldots,d$.
	This shows that
	\[
		\theta_{\psi}=\sum_{i=1}^{c_\phi}\sum_{j=1}^d
		a_{ij}\Bigl(\sum_{\tau\in G}\tau_\indj\xi^{(\indi)}_{\tau(\La_0)}\Bigr),
	\]
	which with \eqref{eq: decompose} gives the representation of Theorem \ref{thm: A2} of the introduction.
\end{remark}

\appendix

%
%
%
\section{The Cohomology of Graphs}\label{sec: A}
%
%
%

A \emph{cohomology} is an algebraic method to extract invariants of a mathematical object.
In this section, we review well-known facts concerning the definition of the cohomology of a graph.
Let $(X,E)$ be any symmetric directed graph.

Following \cite{Sun13}*{\S4.6},  we define the cohomology of the graph $(X,E)$ as follows.

\begin{definition}\label{def: coh graph}
	For any symmetric directed graph $(X,E)$, we let
	\begin{align*}
		C(X) &\coloneqq\Map(X, \bbR),&
		C^1(X) &\coloneqq\Map^\alt(E, \bbR),
	\end{align*}
	where $\Map^{\alt}(E, \bbR)
	\coloneqq\{\omega\colon E\rightarrow \bbR 
	\mid \forall e\in E\,\,\omega(\bar e)=-\omega(e) \}$.
	Furthermore, we define the differential
	\begin{equation}\label{eq: diff A}
		\partial\colon C(X)\rightarrow C^1(X),\qquad f\mapsto\partial f
	\end{equation}
	by  $\partial f(e)\coloneqq f(t(e)) - f(o(e))$ for any $e\in E$.	
	We define the \textit{cohomology} of $(X,E)$ by
	\begin{align*}
		H^0(X)&\coloneqq \Ker\partial, &  H^1(X)&\coloneqq C^1(X)/\partial C(X),
	\end{align*}
	and $H^m(X)\coloneqq\{0\}$
	for any $m\in\bbN$ such that $m\neq0,1$.
\end{definition}

We call any $\omega$ in $C^1(X)$ a \emph{form} on $(X,E)$.
A form is the analogue of a differential form.
We say that a form $\omega$ is exact, if there exists $f\in C(X)$ such that $\omega=\partial f$.
Then the cohomology $H^1(X)$ is the quotient space of the space of forms by the space of exact forms.
 
The cohomology $H^m(X)$ reflects the topology of the graph $(X,E)$.
In fact, classes of the $0$-th cohomology $H^0(X)$ corresponds to functions constant on each 
connected components of $(X,E)$,
and the classes of the $1$-st cohomology $H^1(X)$ corresponds to functions on closed paths in $(X,E)$.
In what follows, we will make these statements precise.  

The statement for $H^0(X)$ may be given as follows.

\begin{lemma}\label{lem: horizontal A}
	For any $f\in C(X)$, the function $f$ is horizontal (i.e. $\partial f=0$) if and only if $f$ is constant on the connected components of $(X,E)$.
\end{lemma}

\begin{proof}
	Let $f\in H^0(X)=\ker\partial$ so that $\partial f=0$.  
	Suppose $x,x'\in X$ are in the same connected components of $(X,E)$.
	Then there exists a path $\vec p=(e^1,\ldots,e^N)$ 
	from $x$ to $x'$
	such that $o(e^1)=x$, $t(e^N)=x'$,
	and $t(e^i)=o(e^{i+1})$ for $0<i<N$.
	Then
	\[
		f(x')-f(x)=\sum_{i=1}^N (f(t(e^i))-f(o(e^i))).
	\]
	Since $\partial f=0$, we have $f(t(e^i))-f(o(e^i))=0$ for any $0<i<N$, hence this implies that $f(x')=f(x)$.
	This shows that $f$ is constant on the connected components of $(X,E)$.
	Conversely, if $f$ is constant on the connected components of $(X,E)$,
	then we have $\partial f=0$.  This shows that $H^0(X)=\Ker\partial$ corresponds to
	the $\bbR$-linear space of functions on $X$ which are constant on the connected
	components.  
\end{proof}

As an analogy of the line integral of differential forms, we may define an integration of forms 
along a path in $(X,E)$.
For any path $\vec p=(e^1,\ldots,e^N)$ in $(X,E)$ and form $\omega\in C^1(X)$, we define the integral
of $\omega$ along $\vec p$ by
\begin{equation}\label{eq: integral A}
	\int_{\vec p}\omega\coloneqq\sum_{i=1}^N\omega(e^i).
\end{equation}
We say that a path $\vec p$ in $(X,E)$ is \emph{closed}, if $o(\vec p)=t(\vec p)$.  

\begin{definition}\label{def: closed A}
	We say that a form $\omega\in C^1(X)$ is 
	\emph{closed}, if $\int_{\vec p}\omega=0$ for any closed path $\vec p$ in $(X,E)$.
\end{definition}

\begin{lemma}\label{lem: equivalent A}
	A form $\omega\in C^1(X)$ is exact if and only if it is closed.
\end{lemma}

\begin{proof}
	Suppose $\omega=\partial f$ for some $f\in C(X)$.  Then for any closed path $\vec p=(e^1,\ldots,e^N)$, we have
	\[
		\int_{\vec p}\omega=\sum_{i=1}^N\omega(e^i)=\sum_{i=1}^N (f(t(e^i))-f(o(e^i)))=f(t(\vec p))-f(o(\vec p))=0.
	\]
	This shows that $\omega$ is closed.
	Conversely, suppose $\omega$ is closed.   Fix $x_0\in X$ for each connected component of $X$.
	For any $x\in X$, let $\vec p_{x_0,x}$ be the path from $x_0$ to $x$,
	where $x_0$ is the fixed point in the connected component containing $x$.
	Let
	\[
		f(x)\coloneqq\int_{\vec p_{x_0,x}}\omega.
	\]
	Since $\omega$ is closed, the integral is independent of the choice of the path $\vec p_{x_0,x}$.
	By construction, for any $e\in E$, we have
	\[
		\partial f(e)=f(t(e))-f(o(e))=
		\int_{\vec p_{x_0,t(e)}}\omega-\int_{\vec p_{x_0,o(e)}}\omega=
		\int_{\vec p_{o(e),t(e)}}\omega=\omega(e),
	\]
	where $\vec p_{o(e),t(e)}=(e)$ is the path of length $1$ from $o(e)$ to $t(e)$.
	This shows that $\omega=\partial f\in C^1(X)$, hence $\omega$ is exact as desired.
\end{proof}

Let $Z^1(X)\subset C^1(X)$ be the space of closed forms on $X$.  By Lemma \ref{lem: equivalent A},
we have $Z^1(X)=\partial C(X)$, hence $H^1(X)=C^1(X)/Z^1(X)$.
For any closed path $\vec p$ in $(X,E)$, the integration along $\vec p$ gives a well-defined homomorphism
\begin{equation}\label{eq: integration A}
	\int_{\vec p}\colon H^1(X)\rightarrow\bbR, \qquad \omega\mapsto\int_{\vec p}\omega.
\end{equation}
In fact, the \emph{first homology} $H_1(X,\bbZ)$ consists of finite formal sums of closed paths of $(X,E)$,
and we have
\begin{equation}\label{eq: dual}
	H^1(X)=\Hom_\bbZ(H_1(X,\bbZ),\bbR),
\end{equation}
where $\Hom_\bbZ(H_1(X,\bbZ),\bbR)$ denotes the 
space of group homomorphisms from $H_1(X,\bbZ)$ to $\bbR$.
In the equality \eqref{eq: dual}, a class of a form
$\omega\in H^1(X)$ maps to a a homomorphism
$H_1(X,\bbZ)\rightarrow\bbR$ mapping a closed path $\vec p$ to $\int_{\vec p}\omega$.
To further make this statement precise, we next review the construction of the homology groups 
$H_0(X,\bbZ)$ and $H_1(X,\bbZ)$ of the graph $(X,E)$.

Following \cite{Sun13}*{\S 4.1},
let 
\begin{align*}
	C_0(X,\bbZ)&\coloneqq\bigoplus_{x\in X}\bbZ x=\biggl\{\sum_{x\in X}^{\text{finite}} a_x x \mid a_x\in\bbZ\biggr\},\\
	C_1(X,\bbZ)&\coloneqq\bigoplus^\alt_{e\in E}\bbZ e=\biggl\{ \sum_{e\in E}^{\text{finite}} b_e e \mid b_e\in\bbZ
	\biggr\}/\sim,
\end{align*}
where $\sim$ is the equivalence relation given by $e \sim -\bar e$ for any $e\in E$.
In other words, $C_0(X,\bbZ)$ and $C_1(X,\bbZ)$ are free abelian groups
generated by  the set of vertices and edges of $(X,E)$,
i.e. certain
sets of \emph{finite formal sums} over the set of vertices and edges of $(X,E)$.
We call any element of  $C_0(X,\bbZ)$ or $C_1(X,\bbZ)$ a \emph{chain}.
We define the boundary morphism
\begin{equation}\label{eq: homology}
	d\colon C_1(X,\bbZ)\rightarrow C_0(X,\bbZ)
\end{equation}
to be the group homomorphism given by $\partial(e)\coloneqq t(e)-o(e)$.
In other words, 
\[
	d\Bigl(\sum_{e\in E}b_e e\Bigr)=\sum_{e\in E} b_e(t(e)-o(e))\in C_0(X,\bbZ)
\]
for any $\sum_{e\in E}b_e e\in C_1(X,\bbZ)$.

\begin{definition}
	We define the homology of $(X,E)$ by 
	\begin{align*}
		H_0(X,\bbZ)&\coloneqq C_0(X,\bbZ)/d C_1(X,\bbZ), &
		H_1(X,\bbZ)&\coloneqq\Ker d,
	\end{align*}
	and $H_m(X,\bbZ)=\{0\}$ for $m\neq0,1$.
\end{definition}

Note that \emph{lower numberings} (as in $H_0$, $H_1,\ldots$) are used for objects pertaining to homology, 
in contrast with upper numberings (as in $H^0$, $H^1,\ldots$) for objects pertaining to cohomology.
The homology $H_0(X,\bbZ)$ is related to the connected components of $(X,E)$ as follows.

\begin{lemma}\label{lem: H0}
	If we denote by $\pi_0(X)$ the connected components
	of the graph $(X,E)$, then we have $H_0(X,\bbZ)\cong\bigoplus_{C\in\pi_0(X)}\bbZ C$.
\end{lemma}

\begin{proof}
	Our assertion follows from the fact that 
	the classes in $H_0(X,\bbZ)$ corresponding to vertices $x,x'\in X$ coincides if and 
	only if there exists a path $\vec p=(e^1,\ldots,e^N)$ from $x$ to $x'$, in other words, 
	if and only if $x$ and $x'$ are in the same connected component of $(X,E)$.  
\end{proof}

The first homology $H_1(X,\bbZ)$ is related to closed paths as follows.  A path $\vec p=(e^1,\ldots,e^N)$
in $(X,E)$ defines a chain $\alpha\coloneqq e^1+\cdots+e^N$ in $C_1(X,\bbZ)$.
Then
\[
	d(\alpha)=\sum_{i=1}^N (t(e^i)-o(e^i))=t(\vec p)-o(\vec p)\in C_0(X,\bbZ),
\]
where $t(\vec p)=t(e^N)$ and $o(\vec p)=o(e^1)$.
In particular, if the path $\vec p$ is closed, then we have $d(\alpha)=0$.
This shows that any closed path $\vec p$ in $(X,E)$ defines a homology class $\alpha\in H_1(X,\bbZ)=\Ker d$.
Moreover, we have the following.

\begin{lemma}\label{lem: H1}
	Any chain $\alpha\in H_1(X,\bbZ)$ may be expressed as a sum of closed paths.
\end{lemma}

\begin{proof}
	This fact is proved in \cite{Sun13}*{\S 4.2}.
\end{proof}

\begin{figure}[ht]
	\centering
	\includegraphics[width=9cm]{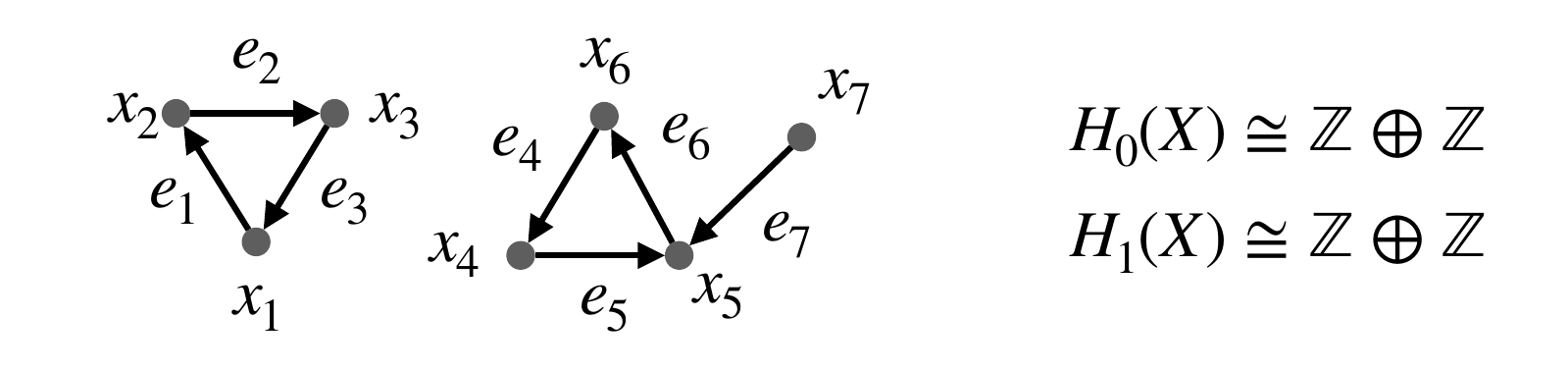}
	\caption{Example of Calculation of Homology}\label{fig: 7}
\end{figure}

In the example of Figure \ref{fig: 7}, let $X=\{x_1,\ldots,x_7\}$ and $E=\{e_1,\ldots,e_7\}\cup\{\bar e_1,\ldots, \bar e_7\}$.
Then the pair $(X,E)$ for a symmetric directed graph.  Then $\vec p_1=(e_1,e_2,e_3)$ and
$\vec p_2=(e_4,e_5,e_6)$ are closed paths of $(X,E)$.    
The vertices $x_1,x_2,x_3$ are in one connected component, and $x_4,x_5,x_6,x_7$ are in a different connected component.  Hence $\pi_0(X)=\{ C(x_1), C(x_4)\}$, where $C(x)$ denotes the connected component of $X$ containing $x\in X$.
This shows that
\[
	H_0(X)=\mathbb{Z} C(x_1)\oplus\mathbb{Z} C(x_4)\cong \mathbb{Z}\oplus\mathbb{Z}.
\]
If we let $\alpha_1\coloneqq e_1+e_2+e_3\in C_1(X,\bbZ)$,
then
\[
	d(\alpha_1)=d(e_1+e_2+e_3)=d(e_1)+d(e_2)+d(e_3)=(x_2-x_1)+(x_3-x_2)+(x_1-x_3)=0.
\]	
Hence $\alpha_1\in H_1(X)$.  Similarly $\alpha_2\coloneqq e_4+e_5+e_6\in H_1(X)$.
In fact, we have
\[
	H_1(X)=\mathbb{Z}\alpha_1\oplus\mathbb{Z}\alpha_2\cong\mathbb{Z}\oplus\mathbb{Z}.
\]

The relation between homology and cohomology is given as follows.

\begin{proposition}\label{prop: dual}
	Let $(X,E)$ be a symmetric directed graph.  We have
	\begin{align*}
		H^0(X)&=\Hom_\bbZ(H_0(X,\bbZ),\bbR),  &
		H^1(X)&=\Hom_\bbZ(H_1(X,\bbZ),\bbR).
	\end{align*}
\end{proposition}

\begin{proof}
	Note that a homomorphism $\Hom_\bbZ(C_0(X,\bbZ),\bbR)$ is determined
	by the image of $x\in X$, we have $\Hom_\bbZ(C_0(X,\bbZ),\bbR)=\Map(X,\bbR)$.
	Similarly, we have $\Hom_\bbZ(C_1(X,\bbZ),\bbR)=\Map^\alt(E,\bbR)$.
	Thus by \eqref{eq: cohomology}, we have
	\begin{align*}
		C^0(X)&=\Hom_\bbZ(C_0(X,\bbZ),\bbR),  &
		C^1(X)&=\Hom_\bbZ(C_1(X,\bbZ),\bbR).
	\end{align*}
	The differential $\partial\colon C^0(X)\rightarrow C^1(X)$ of \eqref{eq: diff A} 
	is the homomorphism induced via pull-back
	of the differential $d\colon C_1(X,\bbZ)\rightarrow C_0(X,\bbZ)$ of \eqref{eq: homology}.
	Hence by duality, we have
	\begin{align*}
		H^0(X)&=\Ker\bigl(\partial\colon \Hom_\bbZ(C_0(X,\bbZ),\bbR)\rightarrow \Hom_\bbZ(C_1(X,\bbZ),\bbR)\bigr)\\
		&= \Hom_\bbZ\bigl(C_0(X,\bbZ)/d C_1(X,\bbZ),\bbR\bigr)= \Hom_\bbZ(H_0(X,\bbZ),\bbR),\\
		H^1(X)&= \Hom_\bbZ(C_1(X,\bbZ),\bbR)/\partial \Hom_\bbZ(C_1(X,\bbZ),\bbR))\\
		&= \Hom_\bbZ(\Ker d,\bbR)= \Hom_\bbZ(H_1(X,\bbZ),\bbR).
	\end{align*}
	This proves our assertion.
\end{proof}

By Proposition \ref{prop: dual}, we see that Lemma \ref{lem: horizontal A} is in fact a consequence of
Lemma \ref{lem: H0}.  
In particular, $\dim_\bbR H^0(X)$ corresponds to the number of connected components of $(X,E)$.
Moreover, by Proposition \ref{prop: dual} and Lemma \ref{lem: H1},
we see that a class in $H^1(X)$ corresponds to the space of functions on closed paths of $(X,E)$
given by integration as in \eqref{eq: integration A}. 

%
%
%
\section{Examples}\label{sec: B}
%
%
%

%
\subsection{The Exclusion Process}\label{subsec: B1}
%

In this subsection, we focus on the exclusion process and give explicit descriptions of objects newly introduced in this article, such as the uniform functions, uniform forms and uniform cohomology.

Consider $S=\{0,1\}$ with the base state $*=0$ and $\phi(s,s')=(s',s)$. We consider any locale $X=(X,E)$. For this setting, $\Consv^\phi(S)=\{\xi : \{0,1\} \to \mathbb{R}  ; \xi(0)=0\}$ and so $c_{\phi}=1$. The class of functions with exact support $\Lambda$, which is defined in \ref{def: exact support} and denoted by $C_{\Lambda}(S^X)$, is explicitly given as 
\[
	C_{\Lambda}(S^X) = \{ f : S^X \to \mathbb{R} \ |  \ f (\eta)= a \Pi_{x \in \Lambda}\eta_x, a \in \mathbb{R}\},
\]
namely, it is a one-dimensional space. Then, Proposition \ref{prop: expansion} implies
that any function $f \in \CS$ is uniquely given as 
\[
	f = \sum_{\Lambda \subset X, |\Lambda| < \infty}a_{\Lambda} \Pi_{x \in \Lambda}\eta_x
\]
for some $a_{\Lambda} \in \mathbb{R}$ for each $\Lambda$. In particular, the space of uniform functions for this case is given as the set of functions of the form
\[
	f = \sum_{\substack{\La\subset X\\\diam{\La}\leq R}} a_{\Lambda} \Pi_{x \in \Lambda}\eta_x
\]
for some $R>0$.
In particular, $\sum_{x \in X} \eta_x$ is a uniform function. Also, if $X=\mathbb{Z}^d$, then $\sum_{x \in \mathbb{Z}^d}x_j \eta_x$ is a uniform function for $j=1,2,\dots,d$. 

The \emph{zeroth} cohomology $H^0_\unif(\State)$ is the space of functions $f \in C^0_\unif(\State)$ such that $\nabla_e f=0$ for any $e \in E$. Theorem \ref{thm: 2} claims that  $H^0_\unif(\State)$ coincides with the one-dimensional space $\{a \sum_{x \in X} \eta_x \ | \ a \in \mathbb{R}\}$. 

Next, we give some examples of the pairing $h_f$ in Proposition \ref{prop: h_f}. First note that $\Val= \{0,1,2,\dots\}$ for this setting. Let $f=(\sum_{x \in X} \eta_x)^2=\sum_{x,y \in X}\eta_x \eta_y \in C(\State_*)$.
Then $\partial f =0 \in  C^1_R(\State)$ for any $R >0$. For any $\La, \La'$ such that $d_X(\La,\La')>0$, we have
\begin{align*}
	\iota^{\La\cup\La'} f(\state)-\iota^{\La} f(\state)-\iota^{\La'}f(\state) 
	&=\Bigl(\sum_{x \in \La\cup\La'}\eta_x\Bigr)^2 -\Bigl(\sum_{x \in \La}\eta_x\Bigr)^2 
	- \Bigl(\sum_{x \in \La'}\eta_x\Bigr)^2 = 2 \sum_{x \in \La}\eta_x  \sum_{y \in \La'}\eta_y \\
	&= 2 \xi_{\La}(\state)\xi_{\La'}(\state). 
\end{align*}
Namely, $h_f(\alpha, \beta)=2\alpha \beta$ for any $\alpha,\beta\in\cM$. 
Note that $h_f \neq 0$ in this case.

%
\subsection{Counterexample for the case when $X=(\bbZ,\bbE)$ and  $c_\phi=2$.}\label{subsec: B2}
%

In Theorem \ref{thm: 1}, we assumed that the interaction is simple in the weakly transferable case.
In this subsection, we give an example of a weakly transferable locale $X$ and an
interaction $\phi$ that is irreducibly quantified and satisfies $c_\phi>1$,
but the uniform cohomology $H^1_\unif(\State)$ 
does not satisfy the conclusion of Theorem \ref{thm: 1}.
More precisely, we prove the  following.

\begin{proposition}\label{prop: counterexample}
	Suppose $\kappa=2$ so that $S=\{0,1,2\}$, and let  $X=(\bbZ,\bbE)$, which is
	weakly transferable.  If we consider the configuration with transition structure $\State$
	for the interaction $\phi\colon S\times S\rightarrow S\times S$
	given in Example \ref{example: interactions} (2),
	then we have  $H^1_\unif(\State)\neq0$.
\end{proposition}

\begin{proof}
	We will explicitly construct a form $\omega\in Z^1_\unif(\State)$
	which is not integrable by a function $F\in C^0_\unif(\State)$.
	For any $W\subset\State_*$, let $\bsone_W$ be the characteristic function
	of $W$ on $\State_*$.
	Let $\omega=(\omega_e)_{e\in E}\in\Map(\Phi,\bbR)$ such that 
	\[
		\omega_e\coloneqq\bsone_{\{\eta_{o(e)}=2,\eta_{t(e)}=1\}}
		-\bsone_{\{\eta_{o(e)}=1,\eta_{t(e)}=2\}}
	\]
	if $t(e)\geq o(e)$ and $\omega_{e}\coloneqq\omega_{\bar e}$ otherwise.
	We let $f\in C(\State_*)$  be the function
	\[
		f=\sum_{y>x}\bsone_{\{\eta_x=1,\eta_y=2\}}.
	\]
	Then for any $\state\in\State_*$ and $e\in \bbE$,
	we have
	\[
		\nabe f(\state)=f(\state^e)-f(\state)
		=\sum_{y>x}\bsone_{\{\eta_x=1,\eta_y=2\}}(\state^e)-
		\sum_{y>x}\bsone_{\{\eta_x=1,\eta_y=2\}}(\state).
	\]
	If we suppose  $t(e)\geq o(e)$, then the sum of $\nabe f(\state)$ 
	cancel outside $e=(x,y)$, and we have
	\[
		\nabe f(\state)
		=\bsone_{\{\eta_{o(e)}=2,\eta_{t(e)}=1\}}(\state)-\bsone_{\{\eta_{t(e)}=1,\eta_{o(e)}=2\}}(\state)=\omega_e.
	\]
	Furthermore, 
	we have $\nabla_{\bar e}f(\state)=\omega_e=\omega_{\bar e}$, hence $\partial f=\omega$,
	which by Lemma \ref{lem: exact} shows that $\omega\in Z^1(\State_*)$.
	By definition, we have $\omega\in C^1_0(\State)\subset C^1_\unif(\State)$,
	hence $\omega\in Z^1_\unif(\State)$.
	
	We prove our assertion by contradiction.
	Suppose there exists $F\in C^0_\unif(\State)$ such that  $\partial F=\omega$.
	By taking $F-F(*)$ if necessary, we may assume that $F(*)=0$. 
	For any $\state,\state'\in\State_*$,  if $\bsxi_{\!X}(\state)=\bsxi_{\!X}(\state')$, then
	since the interaction is irreducibly quantified by Proposition \ref{prop: irreducibly quantified}, 
	there exists a path $\vec\gamma$ from $\state$ to $\state'$.  
	Since $\partial(f-F)=0$, by Lemma \ref{lem: special}, we have $(f-F)(\state)=(f-F)(\state')$.
	This shows that there exists $h\colon\Val\rightarrow\bbR$ such that
	$
		f(\state)-F(\state)=h\circ\bsxi_{\!X}(\state)
	$
	for any $\state\in\State_*$.
	By  Remark \ref{rem: converse}, we  have
	$
		\iota^{\La\cup\La'}F=\iota^{\La} F+\iota^{\La'}F
	$
	for any pair $(\La,\La')\in\sA_R$.   Hence 
	\begin{align*}
		\iota^{\La\cup\La'}f(\state)-\iota^{\La} f(\state)-\iota^{\La'}f(\state)
		&= \iota^{\La\cup\La'}(f-F)(\state)-\iota^{\La} (f-F)(\state)-\iota^{\La'}(f-F)(\state)\\
		&= h\circ\bsxi_{\La\cup\La'}(\state)-h\circ\bsxi_\La(\state)-h\circ\bsxi_{\La'}(\state)\\
		&= h\circ(\bsxi_\La(\state)+\bsxi_{\La'}(\state))-h\circ\bsxi_\La(\state)-h\circ\bsxi_{\La'}(\state).
	\end{align*}
	In particular, for $\state,\state'\in\State_*$, if $\bsxi_\La(\state)=\bsxi_{\La'}(\state')$
	and $\bsxi_\La(\state')=\bsxi_{\La'}(\state)$, then we have
	\[
		\iota^{\La\cup\La'}f(\state)-\iota^{\La} f(\state)-\iota^{\La'}f(\state)
		=\iota^{\La\cup\La'}f(\state')-\iota^{\La} f(\state')-\iota^{\La'}f(\state').
	\]
	In what follows, we prove that this it not the case.
	For any $(\La,\La')\in\sA_R$,  we have
	\begin{align*}
		\iota^{\La\cup\La'}f-\iota^{\La} f-\iota^{\La'}f
		&=\sum_{\substack{y>x\\x,y\in\La\cup\La'}}\bsone_{\{\eta_x=1,\eta_y=2\}}
		-\sum_{\substack{y>x\\x,y\in\La}}\bsone_{\{\eta_x=1,\eta_y=2\}}-\sum_{\substack{y>x\\x,y\in\La'}}\bsone_{\{\eta_x=1,\eta_y=2\}}\\
		&=\sum_{\substack{y>x\\ y\in\La,x\in\La'}}\bsone_{\{\eta_x=1,\eta_y=2\}}
		+\sum_{\substack{y>x\\ x\in\La,y\in\La'}}\bsone_{\{\eta_x=1,\eta_y=2\}}.
	\end{align*}
	Since $\La$ and $\La'$ are connected,
	replacing $\La$ by $\La'$ if necessary, we may assume that $y>x$ for any $y\in\La$ and
	$x\in\La'$.   Then we have
	\[
	\iota^{\La\cup\La'}f-\iota^{\La} f-\iota^{\La'}f
	=\sum_{\substack{y>x\\y\in\La,x\in\La'}}\bsone_{\{\eta_x=1,\eta_y=2\}}=\xi_{\La'}^{(1)}\xi_{\La}^{(2)},
	\]
	where $\xi^{(1)}$ and $\xi^{(2)}$ 
	are the conserved quantities defined in Example \ref{example: interactions} (2).
	Fix $y\in \La$ and $y'\in\La'$.  We let $\state=(\eta_x)\in\State_*$ be an element 
	such that $\eta_{y'}=1$, $\eta_y=2$ and is at base state outside $y$ and $y'$,
	and  we let $\state'=(\eta'_x)\in\State_*$ be an element 
	such that $\eta'_{y'}=2$, $\eta'_y=1$ and is at base state outside $y$ and $y'$.
	Then we have $\xi_{\La'}^{(1)}(\state)\xi_{\La'}^{(2)}(\state)=1$
	but  $\xi_{\La'}^{(1)}(\state')\xi_{\La'}^{(2)}(\state')=0$, hence we have
	\[
	\iota^{\La\cup\La'}f(\state)-\iota^{\La} f(\state)-\iota^{\La'}f(\state)
		\neq\iota^{\La\cup\La'}f(\state')-\iota^{\La} f(\state')-\iota^{\La'}f(\state'),
	\]
	which gives a contradiction as desired.
\end{proof}

%
\subsection*{Acknowledgement}
%

The authors would like to sincerely thank Megumi Harada and Hiroyuki Ochiai 
for carefully reading a preliminary 
version of this article and pointing out important improvements.
The authors would also like to thank Kei Hagihara and Shuji Yamamoto 
for detailed discussions, especially concerning the definition of 
irreducibly quantified interactions.  We are grateful to David Croydon
for pointing out that Cayley graphs also fit into our formalism.
The authors also thank Takeshi Katsura, Asuka Takatsu, Ryokichi Tanaka, and
Kazuki Yamada for comments and discussion.
The first and third authors would like to thank the Department of Mathematics at 
Columbia University, where part of this research was conducted.
The authors are very grateful to the referee for carefully reading the article and
giving helpful comments which helped to improve the article.


\end{document}